\documentclass[]{article}

\usepackage{amsmath,amsthm}
\usepackage{amssymb}
\usepackage[letterpaper]{geometry}
\usepackage{comment}
\usepackage{amsmath,amsthm,amsfonts,amssymb,bm,mathtools,mathrsfs}
\usepackage{enumitem}
\usepackage{tikz-cd}
\usepackage{float}
\usetikzlibrary{patterns,automata,arrows,shapes,snakes,topaths,trees,backgrounds,positioning,through,calc}
\usepackage{quiver}
\usepackage[hidelinks]{hyperref}
\usepackage{cleveref,caption}
\usepackage{multirow}
\title{Fundamental Logic Through the Lens of Modality}
\author{Wesley H. Holliday$^\dagger$ and Guillaume Massas$^\ddagger$ \\ {\small $\dagger$ University of California, Berkeley and $\ddagger$ Chapman University}}
\date{July 22, 2026}

\input{fitch.sty}

\usepackage{natbib}
\setcitestyle{round}

\theoremstyle{definition}
\newtheorem{proposition}{Proposition}[section]
\newtheorem{lemma}[proposition]{Lemma}

\newtheorem{theorem}[proposition]{Theorem}
\newtheorem{corollary}[proposition]{Corollary}

\newtheorem{definition}[proposition]{Definition}
\newtheorem{fact}[proposition]{Fact}
\newtheorem{example}[proposition]{Example}
\newtheorem{remark}[proposition]{Remark}
\newtheorem{question}{Question}
\crefname{chapter}{Chapter}{Chapters}
\crefname{lemma}{Lemma}{Lemmas}
\crefname{theorem}{Theorem}{Theorems}
\crefname{definition}{Definition}{Definitions}
\crefname{proposition}{Proposition}{Propositions}
\crefname{notation}{Notation}{}
\crefname{corollary}{Corollary}{Corollaries}
\crefname{section}{Section}{Sections}
\crefname{remark}{Remark}{Remarks}

\usepackage{subfigure}

\newcommand{\op}{\mathrel{\mathchoice{\hbox{$\triangleleft \mkern-8.3mu\mid$}}
{\hbox{$\triangleleft \mkern-8.3mu\mid$}}
{\hbox{$\scriptstyle{\triangleleft \mkern-3.6mu\mid}$}}
{\hbox{$\scriptscriptstyle{\triangleleft \mkern-4.1mu\mid}$}}}}
\newcommand{\po}{\mathrel{\mathchoice{\hbox{$\mid\mkern-8.3mu\triangleright$}}
{\hbox{$\mid\mkern-8.3mu\triangleright$}}
{\hbox{$\scriptstyle{\mid\mkern-3.6mu\triangleright}$}}
{\hbox{$\scriptscriptstyle{\mid\mkern-4.1mu\triangleright}$}}}}
\newcommand{\opo}{\mathrel{\mathchoice{\hbox{$\triangleleft \mkern-8.3mu \mid  \mkern-8.3mu\triangleright$}}
{\hbox{$\triangleleft \mkern-8.3mu \mid  \mkern-8.3mu\triangleright$}}
{\hbox{$\scriptstyle{\triangleleft \mkern-3.6mu \mid  \mkern-3.6mu\triangleright}$}}
{\hbox{$\scriptscriptstyle{\triangleleft \mkern-4.1mu \mid  \mkern-4.1mu\triangleright}$}}}}

\makeatletter
\newcommand*{\shortdashleftarrow}{%
  \mathrel{\mathchar"0\hexnumber@\symAMSa4C\dabar@}}
\newcommand*{\shortdashrightarrow}{%
  \mathrel{\dabar@\mathchar"0\hexnumber@\symAMSa4B}}
\makeatother

\newcommand{\sset}{\subseteq}

\newcommand{\Po}{\mathscr{P}}
\newcommand{\me}{\land}
\newcommand{\bigme}{\bigwedge}
\newcommand{\jo}{\lor}
\newcommand{\bigjo}{\bigvee}

\newcommand{\mb}{\mathbf}
\renewcommand{\phi}{\varphi}

\newcommand{\upset}[1]{\mathord{\uparrow \! #1}}
\newcommand{\dnset}[1]{\mathord{\downarrow \! #1}}

\begin{document}

\maketitle

\begin{abstract} Fundamental logic is a non-classical logic based only on the introduction and elimination rules for conjunction, disjunction, negation, and the quantifiers in a Fitch-style natural deduction system. In this paper, we attempt to obtain a better understanding of fundamental logic and its semantics through the lens of modality. Using modal logic, we develop means of mutual understanding between the fundamental logician, on the one hand, and the orthologician and intuitionistic logician, on the other: we prove that the G\"{o}del-McKinsey-Tarski (GMT) translation of intuitionistic logic into the classical modal logic $\mathsf{S4}$ is a full and faithful embedding of fundamental logic into the \textit{orthological} version of $\mathsf{S4}$; that the Goldblatt translation of orthologic into the classical modal logic $\mathsf{KTB}$ is a full and faithful embedding of fundamental logic into an \textit{intuitionistic} version of $\mathsf{KTB}$; and that the GMT translation is a full and faithful embedding of intuitionistic logic into a modal extension of fundamental logic.
\end{abstract}

\setcounter{tocdepth}{2}
\tableofcontents

\section{Introduction}

There is a rich tradition in logic of gaining insight into non-classical logics by translating them into classical modal logics. Perhaps the most famous example is G\"{o}del's \citeyearpar{Godel1933b} translation of intuitionistic logic into the classical modal logic $\mathsf{S4}$, which McKinsey and Tarski \citeyearpar{McKinsey1948} proved to be full and faithful. Another well-known example is Goldblatt's \citeyearpar{Goldblatt1974} translation of orthologic into the classical modal logic $\mathsf{KTB}$. In this paper, we attempt to use the lens of modality to gain insight into another non-classical logic, namely \textit{fundamental logic} (\citealt{Holliday2023}). The twist is that we will do so by translating fundamental logic into modal logics based on other, more familiar, \textit{non-classical} logics---intuitionistic logic and orthologic.

As we will review below, fundamental logic is a non-classical logic based only on the introduction and elimination rules for conjunction, disjunction, negation, and the quantifiers in a Fitch-style natural deduction system.  Adding to fundamental logic the rule of Reductio Ad Absurdum yields a proof system for orthologic, while adding the rule that Fitch called Reiteration yields a proof system for the $\to$-free fragment of intuitionistic logic. Deductively, fundamental logic stands to orthologic as intuitionistic logic stands to classical logic: the G\"{o}del-Gentzen double negation translation of classical logic into intuitionistic logic (\citealt{Godel1933b}, \citealt{Gentzen1936}) is a full and faithful translation of orthologic into fundamental logic. Semantically, fundamental logic can be characterized by a relational semantics, using a relation of ``openness'' between partial states; adding extra constraints on openness yields semantics for orthologic or intuitionistic logic. 

In this paper, we develop means of mutual understanding between the fundamental logician, on the one hand, and the orthologician and intuitionistic logician, on the other, with the following main results. 

\begin{theorem}\label{FirstThm} The G\"{o}del-McKinsey-Tarski translation of intuitionistic logic into classical $\mathsf{S4}$ is also a full and faithful translation of fundamental logic into \emph{ortho}-$\mathsf{S4}$, i.e., $\mathsf{S4}$ on an orthological base.
\end{theorem}

\begin{theorem}\label{SecondThm} The Goldblatt translation of orthologic into classical $\mathsf{KTB}$ is also a full and faithful translation of fundamental logic into \emph{intuitionistic} $\mathsf{KTB}$, defined in Fischer Servi's \citeyearpar{FischerServi1984} style of intuitionistic modal~logic.
\end{theorem}

\begin{theorem}\label{ThirdThm} The G\"{o}del-McKinsey-Tarski translation of intuitionistic logic into classical $\mathsf{S4}$ is also a full and faithful translation of the $\to$-free fragment of intuitionistic logic into $\mathsf{FN4}$, an extension of $\mathsf{S4}$ on a fundamental base that captures a strong notion of necessity.
\end{theorem}

\noindent A natural question concerns adding a conditional connective $\to$ to this picture. As discussed in \citealt{Holliday2023}, there are multiple options for adding a conditional to fundamental logic. In recent work building on the results of the present paper, Zhicheng Chen (personal communication) takes the ``preconditional'' approach of \citealt{Holliday2024b} and extends Theorems \ref{FirstThm} and \ref{SecondThm} accordingly. In this paper, we add a different conditional to fundamental logic, based on the standard introduction and elimination rules for $\to$, and extend Theorem~\ref{ThirdThm} accordingly to cover full intuitionistic first-order logic.

The rest of the paper is organized as follows. In Section \ref{PrelimSection}, we review all of the necessary background concerning non-classical and modal logics, deductively and semantically. In Sections~\ref{OS4Section}, \ref{KTBSection}, and \ref{FN4Section}, we prove Theorems \ref{FirstThm}, \ref{SecondThm}, and \ref{ThirdThm}, respectively. In Section \ref{FurtherDirections}, we discuss three further directions for exploration suggested by our three main results. We conclude with some brief final reflections in Section~\ref{Conclusion}.

\section{Preliminaries}\label{PrelimSection}

\subsection{Logics}\label{LogicsSec}

Given a set $\mathsf{Prop}$ of propositional variables, we define the language $\mathcal{L}(\wedge,\vee,\neg)$ by the following grammar:
\[\varphi::= \top\mid p\mid (\varphi\wedge\varphi)\mid (\varphi\vee\varphi)\mid \neg\varphi,\]
where $p\in\mathsf{Prop}$. As an abbreviation, we define $\bot:=\neg\top$.

\begin{definition} Following \citealt{Holliday2023},\footnote{This definition is slightly different than the one in \citealt{Holliday2023}, though they agree concerning which formulas are derivable from a given formula, with the exception of $\top$, which was not included in \citealt{Holliday2023}. The definition in \citealt{Holliday2023} includes a clause for extending a proof with a subproof, and the clauses for the introduction and elimination rules only add formulas, not subproofs and formulas as we do here. The presentation here will be more convenient for our purposes in Section~\ref{FN4Section}.} the set of \textit{proofs of fundamental} (\textit{propositional}) \textit{logic} is defined inductively as the smallest set containing for each formula $\varphi$ the sequence $\langle \varphi\rangle$ and satisfying the following closure conditions for $1\leq i,j\leq n$:
\begin{itemize}
\item If $\langle \sigma_1,\dots,\sigma_n\rangle$ is a proof, then so is $\langle \sigma_1,\dots,\sigma_n,\top\rangle$ ($\top$I).
\item If $\langle \sigma_1,\dots,\sigma_n\rangle$ is a  proof and $\sigma_i,\sigma_j$  are formulas, then $\langle \sigma_1,\dots,\sigma_n,\sigma_i\wedge\sigma_j\rangle$ is a proof ($\wedge$I).
\item If $\langle \sigma_1,\dots,\sigma_n\rangle$ is a  proof and $\sigma_i$ is a formula of the form $\varphi\wedge\psi$, then $\langle \sigma_1,\dots,\sigma_n,\varphi\rangle$ and $\langle \sigma_1,\dots,\sigma_n,\psi\rangle$ are proofs ($\wedge$E).
\item If $\langle \sigma_1,\dots,\sigma_n\rangle$ is a  proof and  $\sigma_i$ is a formula, then for any formula $\varphi$,  both $\langle \sigma_1,\dots,\sigma_n,\sigma_i\vee\varphi\rangle$ and $\langle \sigma_1,\dots,\sigma_n,\varphi\vee \sigma_i\rangle$ are proofs ($\vee$I).
\item If $\langle \sigma_1,\dots,\sigma_n\rangle$ is a  proof, $\sigma_i$ is a formula of the form $\varphi\vee\psi$, $\tau_1$ is a proof beginning with $\varphi$ and ending with $\chi$, and $\tau_2$ is a proof beginning with $\psi$ and ending with $\chi$, then $\langle \sigma_1,\dots,\sigma_n, \tau_1,\tau_2,\chi\rangle$ is a proof ($\vee $E).
\item If $\langle \sigma_1,\dots,\sigma_n\rangle$ is a  proof, $\sigma_i$ is a formula  $\psi$, and $\tau$ is a proof beginning with $\varphi$ and ending with $\neg\psi$, then $\langle \sigma_1,\dots,\sigma_n,\tau,\neg\varphi\rangle$ is a proof ($\neg$I).
\item If $\langle \sigma_1,\dots,\sigma_n\rangle$ is a  proof and $\sigma_i$ and $\sigma_j$ are formulas of the form $\varphi$ and $\neg\varphi$, respectively, then for any formula $\psi$, $\langle \sigma_1,\dots,\sigma_n,\psi\rangle$ is a proof ($\neg$E). 
\end{itemize}
\end{definition}

The inductive cases for $\wedge$, $\vee$, and $\neg$ are represented diagrammatically in Figure \ref{FitchRules}. Indentation is used to represent proofs within proofs, called \textit{subproofs}. The horizontal line at the beginning of a proof highlights the first item in the proof, which is considered the \textit{assumption} of the proof. The last item in the proof is considered the \textit{conclusion} of the proof. For the base case, we represent the proof $\langle \varphi\rangle$ diagrammatically as a proof with assumption $\varphi$ and conclusion $\varphi$:
\[\begin{nd}
\hypo [\,]  {1} {\varphi}
\have [\,] {2} {\varphi}
\end{nd}\]

\begin{definition}
Define a binary relation $\vdash_\mathsf{F}$ on $\mathcal{L}(\wedge,\vee,\neg)$ as follows: $\varphi\vdash_\mathsf{F}\psi$ if there is some proof in fundamental logic whose assumption is $\varphi$ and whose conclusion is $\psi$.
\end{definition}

\begin{figure}[H]
\begin{center}

\begin{minipage}{2in}
\[\begin{nd}
\have [\vdots] {4} {\vdots}
\have [i] {5}  {\varphi_s}
\have [\vdots] {8}   {\vdots}
\have [j] {9} {\varphi_t}
\have [\vdots] {10}   {\vdots}
\have [k] {11}   {\varphi_1\wedge\varphi_2}\ai{5,9}
\end{nd}\]
\end{minipage}\;\;\,\begin{minipage}{2.25in}
\[\begin{nd}
\have [\vdots] {4} {\vdots}
\have [i] {5}  {\varphi_1\wedge\varphi_2}
\have [\vdots] {8}   {\vdots}
\have [j] {9} {\varphi_s} \ae{5}
\end{nd}\]
\end{minipage}
\begin{minipage}{2.45in}
\[\begin{nd}
\have [\vdots] {1} {\vdots} 
\have [i] {2} {\varphi_s} 
\have [\vdots] {3} {\vdots} 
\have [j] {4} {\varphi_1\vee\varphi_2}\oi{2}
\end{nd}\]
\end{minipage}\;\;\,\begin{minipage}{2.4in}

\[\begin{nd}
\have [\vdots] {} {\vdots} 
\have [i] {0} {\varphi\vee\psi}
\have [\vdots] {1} {\vdots} 
\open
\hypo [j] {2} {\varphi}
\have [\vdots] {3} {\vdots}
\have [k] {5} {\chi}
\close
\open
\hypo [l] {7}  {\psi}
\have [\vdots] {8} {\vdots}
\have [m] {9} {\chi}
\close
\have [n] {n} {\chi} \oe{0,2-5,7-9}
\end{nd}\]\vspace{.1in}
\end{minipage}

\begin{minipage}{2in}
\[\begin{nd}
\have [\vdots] {0} {\vdots}
\have [i] {3}   {\psi}
\have [\vdots] {}  {\vdots}
\open
\hypo [j] {1} {\varphi}
\have [\vdots] {4}   {\vdots}
\have [k] {5}   {\neg\psi}
\close
\have [l] {6} {\neg\varphi} \ni{1-5,3}
\end{nd}\]
\end{minipage}\begin{minipage}{2.25in}
\[\begin{nd}
\have [\vdots] {4} {\vdots}
\have [i] {5}  {\varphi}
\have [\vdots] {6} {\vdots}
\have [j] {7}  {\neg\varphi}
\have [\vdots] {8} {\vdots}
\have [k] {9}{\psi} \ne{5, 7}
\end{nd}\]
\end{minipage}
\end{center}
\caption{Fitch-style introduction and elimination rules for fundamental logic, where $s,t\in\{1,2\}$.}\label{FitchRules}
\end{figure}

\textit{Orthologic} (\citealt{Goldblatt1974}), denoted $\mathsf{O}$, adds to fundamental logic the rule of Reductio Ad Absurdum, depicted in Figure \ref{RAAFig}:
\begin{itemize}
\item If $\langle \sigma_1,\dots,\sigma_n\rangle$ is a  proof, $\sigma_i$ is a formula  $\psi$, and $\tau$ is a proof beginning with $\neg\varphi$ and ending with $\neg\psi$, then $\langle \sigma_1,\dots,\sigma_n,\tau,\varphi\rangle$ is a proof (RAA).
\end{itemize}

\begin{figure}[H]
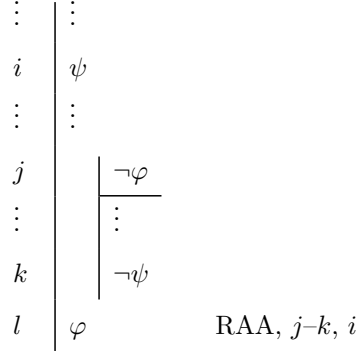

\begin{center}
\begin{minipage}{2.5in}
\[\begin{nd}
\have [\vdots] {0} {\vdots}
\have [i] {3}   {\psi}
\have [\vdots] {4}   {\vdots}
\open
\hypo[ j] {1} {\neg \varphi}
\have [\vdots] {}  {\vdots}
\have [k] {5}   {\neg\psi}
\close
\have [l]{6} {\varphi} \RAA{1-5,3}
\end{nd}
\]\end{minipage}
\end{center}
\caption{The Reductio ad Absurdum (RAA) rule of orthologic.}\label{RAAFig}
\end{figure}

\textit{Intuitionistic logic} in the $\{\wedge,\vee,\neg\}$-fragment, denoted $\mathsf{IPC}^{-}$, adds to fundamental logic the rule of Reiteration, depicted in Figure \ref{ReitFig}. For a rigorous inductive definition, this requires defining the notion of a \textit{proof from a set $R$} of reiterable formulas, which is done in Appendix~\ref{AppendixIPC}.

\begin{figure}[H]
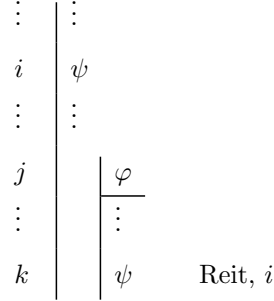

\begin{center}
\begin{minipage}{2.5in}
\[\begin{nd}
\have [\vdots] {0} {\vdots}
\have [i] {3}   {\psi}
\have [\vdots] {4}   {\vdots}
\open
\hypo[ j] {1} {\varphi}
\have [\vdots] {}  {\vdots}
\have [k] {5}   {\psi} \r{3}
\end{nd}
\]\end{minipage}
\end{center}
\caption{The Reiteration rule of intuitionistic logic.}\label{ReitFig}
\end{figure}

\textit{Classical logic}, denoted $\mathsf{CPC}$, is obtained by adding both the Reiteration and Reductio ad Absurdum rules. Quantifiers can also be added to the language, with their standard introduction and elimination rules. To keep this introduction brief, we postpone a discussion of first-order logics to \cref{FirstOrderExt1}. The relationships between fundamental logic, the $\to$-free fragment of intuitionistic logic, orthologic and classical logic are summarized in the ``fundamental diamond'' of Figure~\ref{fundiamond}.

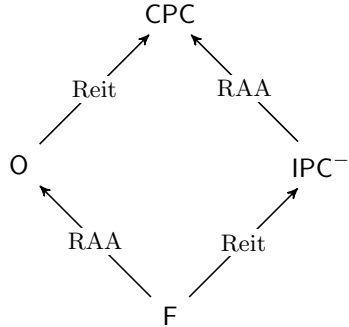
\begin{figure}[H]
    \centering
    \begin{tikzpicture}[
        ->,>={stealth'[length=6pt]},
        shorten >=1pt,shorten <=1pt,
        semithick,
        desc/.style={fill=white,inner sep=2pt,font=\small}
    ]
        \node (CPC) at (2,4) {$\mathsf{CPC}$};
        \node (O)   at (0,2) {$\mathsf{O}$};
        \node (IPC) at (4,2) {$\mathsf{IPC}^-$};
        \node (F)   at (2,0) {$\mathsf{F}$};

        \path (O)   edge node[desc] {Reit} (CPC);
        \path (IPC) edge node[desc] {RAA}  (CPC);
        \path (F)   edge node[desc] {RAA}  (O);
        \path (F)   edge node[desc] {Reit} (IPC);
    \end{tikzpicture}
    \caption{The ``fundamental diamond'' relating fundamental logic, orthologic, $\mathsf{IPC^-}$, and classical logic.}
    \label{fundiamond}
\end{figure}

Finally, we turn to modal logics. The language $\mathcal{L}(\wedge,\vee,\neg,\Box)$ is given by the following grammar:
\[\varphi::= \top\mid p\mid (\varphi\wedge\varphi)\mid (\varphi\vee\varphi)\mid \neg\varphi\mid \Box\varphi,\]
where $p\in\mathsf{Prop}$. Following Fitch~\citeyearpar{Fitch1966}, we make a distinction between ordinary subproofs and $\Box$-subproofs, where the latter are used for $\Box$-Introduction and $\Box$-Elimination, as shown in Figure~\ref{BoxRules}. For a rigorous inductive definition, we must keep track of a set $B$ of boxed formulas to which $\Box$E can be applied, which is done in Appendix~\ref{AppendixModal}. Let $\mathsf{FK}$ (resp.~$\mathsf{OK}$) be the minimal modal extension of fundamental logic (resp.~orthologic) with $\Box$I and $\Box$E. The logics $\mathsf{FS4}$ and $\mathsf{OS4}$ add to $\mathsf{FK}$ and $\mathsf{OK}$, respectively, the \textsf{T} Rule and \textsf{4} Rule depicted in Figure~\ref{S4Rules}, which can be stated as inductive clauses in an obvious way.

\begin{figure}[H]
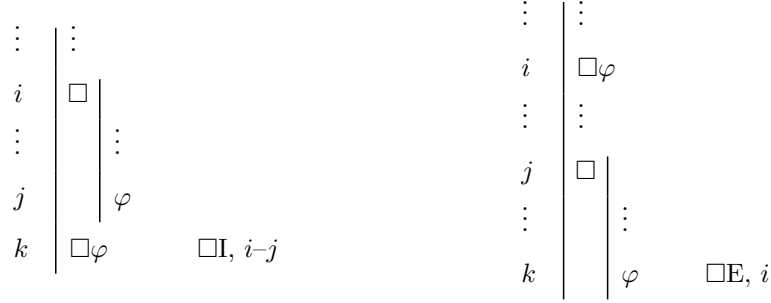

\begin{center}
\begin{minipage}{2.5in}
\[\begin{nd}
\have [\vdots] {0} {\vdots}
\open
\have [i] {3}   {\hspace{-.24in}\Box\;\;\;}
\have [\vdots] {4}   {\vdots}
\have [j] {6}   {\varphi}
\close
\have [k]{7} {\Box\varphi} \boxi{3-6}
\end{nd}
\]\end{minipage}\hspace{.1in}\begin{minipage}{2.5in}\[\begin{nd}
\have [\vdots] {} {\vdots}
\have [i] {0} {\Box\varphi}
\have [\vdots] {1} {\vdots}
\open
\have [j] {3}   {\hspace{-.24in}\Box\;\;\;}
\have [\vdots] {4}   {\vdots}
\have [k] {6}   {\varphi} \boxe{0}
\end{nd}
\]
\end{minipage}
\end{center}
\caption{Introduction and elimination rules for $\Box$.}\label{BoxRules}
\end{figure}

\begin{figure}[h!]
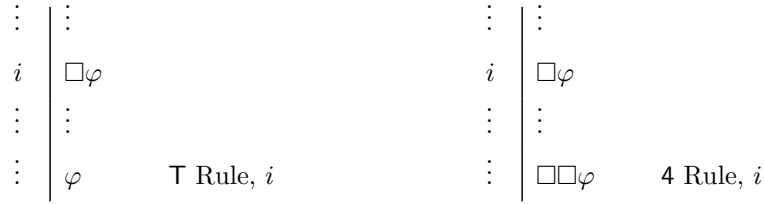

\begin{center}
\begin{minipage}{2.5in}
\[\begin{nd}
\have [\vdots] {0} {\vdots}
\have [i] {1} {\Box\varphi}
\have [\vdots] {2} {\vdots}
\have [\vdots] {3} {\varphi} \TAx{1}
\end{nd}
\]\end{minipage}\begin{minipage}{2.5in}\[\begin{nd}
\have [\vdots] {0} {\vdots}
\have [i] {1} {\Box\varphi}
\have [\vdots] {2} {\vdots}
\have [\vdots] {3} {\Box\Box\varphi}  \FourAx{1}
\end{nd}
\]
\end{minipage}\end{center}

\caption{$\mathsf{S4}$ principles for $\Box$.}\label{S4Rules}
\end{figure}

\subsection{Algebras}

We now introduce classes of algebras corresponding to the logics of the previous section. Given a bounded lattice $L$, we denote its meet and join operations by $\wedge$ and $\vee$, trusting that no confusion will arise between lattice operations and logical connectives; we denote its minimum and maximum by $0$ and $1$, respectively, and we assume throughout that $0\neq 1$.

Algebraic semantics for fundamental logic can be given using bounded lattices equipped with what some authors (\citealt{Dzik2006,Dzik2006b}, \citealt{Almeida2009}) call a weak pseudocomplementation.

\begin{definition} A \textit{fundamental lattice} is an algebra $(L,\neg)$ where $L$ is a bounded lattice and $\neg$ is a unary operation on $L$ satisfying the following for all $a,b\in L$:
\begin{enumerate}
\item if $a\leq \neg b$, then $b\leq\neg a$ (dual self-adjointness);
\item $a\wedge\neg a=0$ (semicomplementation).
\end{enumerate}
A \textit{fundamental embedding} of a fundamental lattice $(L,\neg)$ into a fundamental lattice $(L',\neg')$ is a lattice embedding $e:L\to L'$ that also preserves $\neg$.\end{definition}

Dual self-adjointness corresponds to the introduction rule for negation ($\neg$I), while semicomplementation corresponds to the elimination rule for negation ($\neg$E). Instead of dual self-adjointness, one may equivalently use antitonicity ($a\leq b$ implies $\neg b\leq\neg a$) and double inflationarity ($a\leq \neg\neg a$).

As usual, a valuation $\theta:\mathsf{Prop}\to L$ on a fundamental lattice extends to an interpretation $\tilde{\theta}:{\mathcal{L}(\wedge,\vee,\neg)\to L}$ in the obvious recursive way. We can then define a semantic consequence relation as follows: $\varphi\vDash_\mathbb{F}\psi$ iff for all fundamental lattices $L$ and valuations $\theta:\mathsf{Prop}\to L$, we have $\tilde{\theta}(\varphi)\leq \tilde{\theta}(\psi)$. Then it is straightforward to prove the soundness and completeness theorem stating that $\varphi\vdash_\mathsf{F}\psi$ iff $\varphi\vDash_\mathbb{F}\psi$ (see \citealt[\S~3]{Holliday2023}).

Algebraic semantics for orthologic can be given using special fundamental lattices.

\begin{definition} An \textit{ortholattice} is a fundamental lattice $(L,\neg)$ such that for all $a\in L$, $\neg\neg a\leq a$.
\end{definition}

Algebraic semantics for the $\to$-free fragment of intuitionistic logic can be given using the following special case of fundamental lattices and generalization of Heyting algebras.

\begin{definition} A \textit{pseudocomplemented distributive lattice} is an algebra $(L,\neg)$ where $L$ is a bounded distributive lattice and $\neg$ is a unary operation on $L$ such that for all $a,b\in L$, $a\wedge b =0$ iff $a\leq\neg b$.
\end{definition}

Finally, algebraic semantics for the modal logic $\mathsf{FK}$ (resp.~$\mathsf{OK}$) can be given using fundamental lattices (resp.~ortholattices) expanded with a normal box operation.

\begin{definition}\label{FunNecLat} A \textit{fundamental necessity lattice} (resp.~orthomodal lattice) is an algebra $(L,\neg,\Box)$ where $(L,\neg)$ is a fundamental lattice (resp.~ortholattice) and $\Box$ is a unary operation on $L$ such that $\Box 1=1$ and $\Box (a\wedge b)=\Box a\wedge\Box b$ for all $a,b\in L$.
\end{definition}

\subsection{Frames}

The ``intended'' semantics for fundamental logic is a relational, modal-style semantics. In particular, \citealt{Holliday2023} studies a relational semantics that combines Plo\v{s}\v{c}ica's \citeyearpar{Ploscica1995} relational representation of lattices with Birkhoff's \citeyearpar{Birkhoff1940} relational interpretation of negation. For references to many related semantics, see \S~4.1 of \citealt{Holliday2023}. Here we review only the technical details needed for our later developments.

\subsubsection{Relational Frames}

We start with some general facts about closure operators on powerset lattices induced by relations on sets.

\begin{definition} A \textit{relational frame} is a pair $(X,\op)$ where $X$ is a nonempty set and $\op$ is a binary relation on~$X$. A state $x\in X$ is \textit{absurd} if there is no $y\op x$; otherwise it is \textit{non-absurd}. 
\end{definition}

We call $\op$ the relation of \textit{openness} (for motivation, see \citealt[Remark 4.2]{Holliday2023}). When $y\op x$, we say that $y$ is \textit{open to} $x$.

\begin{definition} Given a relational frame $(X,\op)$, we define a map $c_{\op} : \wp(X)\to\wp(X)$ by
\[c_{\op} (U) = \{x\in X\mid \forall x'\op x\,\exists x''\po x'\colon x''\in U\}\]
A \textit{proposition in} $(X,\op)$ is a subset $U\subseteq X$ such that $U = c_{\op} (U)$.\end{definition}

\begin{lemma}\label{Closure} For any relational frame $(X,\op)$, $c_{\op}$ is a closure operator:
\begin{enumerate}
\item\label{Closure1} for any $U,V\subseteq X$, if $U\subseteq V$, then $c_{\op} (U)\subseteq c_{\op} (V)$;
\item for any $U\subseteq X$, $c_{\op}(U)=c_{\op} (c_{\op} (U))$;
\item\label{Closure3} for any $U\subseteq X$, $U\subseteq c_{\op}(U)$.
\end{enumerate}
\end{lemma}

\begin{proposition} For any relational frame $(X,\op)$, the set of all propositions in $(X,\op)$ ordered by inclusion forms a complete lattice $\chi_{\op}(X)$ where 
\begin{eqnarray*}
\underset{i\in I}{\bigwedge}U_i & =& \underset{i\in I}{\bigcap}U_i  \\
\underset{i\in I}{\bigvee}U_i & =& c_{\op}\bigg(\underset{i\in I}{\bigcup}U_i\bigg)= \{x\in X\mid \forall x'\op x\;\exists x''\po x'\;\exists i\in I\colon x''\in U_i\}.
\end{eqnarray*}
The set of absurd states of $(X,\op)$ is the minimum element $0$ of this lattice. Moreover, the operation $\neg_{\op}$ on propositions defined by 
\[\neg_{\op} U = \{x\in X\mid \forall x'\op x \, (x'\not\in U)\}\]
satisfies the following conditions:
\begin{enumerate}
\item $\neg_{\op} X=0$;
\item if $U\subseteq V$, then $\neg_{\op} V\subseteq\neg_{\op} U$.
\end{enumerate}
 Finally, the operator $c_{\op}$ on the powerset of $X$ corresponds to the composition of the operation $\neg_{\op}$ with the operation $\neg_{\po}$ induced by the converse $\po$ of the relation $\op$: 
\[\neg_{\po} U = \{x\in X\mid \forall x'\po x \, (x'\not\in U)\}\]
In other words, we have that:
\[c_{\op}(U) = \neg_{\op}\neg_{\po}U\] for any $U \sset X$. We call $(\chi_{\op}(X),\neg_{\op})$ the \textit{dual algebra} of $(X,\op)$.\end{proposition}

\begin{lemma}\label{PrerefineLem} In a relational frame $(X,\op)$, for any $x,y\in X$, the following are equivalent:
\begin{enumerate}
\item for all $z\in X$, if $z\op y$, then $z\op x$;
\item for any proposition $U$, if $x\in U$, then $y\in U$.
\end{enumerate}
If these conditions hold, we say that \textit{$y$ prerefines $x$}, denoted $y \leq_{\op} x$. Moreover, when $y$ prerefines $x$ according to the \emph{converse} of $\op$, i.e., when $y \leq_{\po} x$, we say that $y$ \textit{postrefines} $x$.
\end{lemma}

\subsubsection{Fundamental Frames}

Let us now characterize relational frames whose dual algebra is a fundamental lattice.

\begin{definition} A \textit{fundamental frame} is a relational frame $(X,\op)$ satisfying the following conditions:
\begin{enumerate}
\item Pseudo-reflexivity: for all non-absurd $x\in X$, there is a $y\op x$ that prerefines $x$;
\item Pseudo-symmetry: for all $x\in X$ and $y\op x$, there is a $z\op y$ that prerefines $x$.
\end{enumerate}
\end{definition}

\begin{lemma}[\citealt{Holliday2023}, Prop.~4.14.1-2] For any relational frame $(X,\op)$:
\begin{enumerate}
\item $\op$ is pseudo-reflexive iff for every proposition $U$ in $(X,\op)$, we have $U\wedge\neg_{\op} U= 0$;
\item $\op$ is pseudo-symmetric iff for every proposition $U$ in $(X,\op)$, we have $U\subseteq \neg_{\op}\neg_{\op} U$.
\end{enumerate}
\end{lemma}

Intuitively, the openness relation $\op$ in a fundamental frame $(X,\op)$ can be understood by viewing points in $X$ as information states carrying both positive and negative information. Let us say that a state $x$ \textit{accepts} a proposition $U$ whenever $x \in U$, and $x$ \textit{rejects} a proposition $U$ whenever $x \op y$ implies $y \notin U$, i.e., no state to which $x$ is open accepts $U$. Under this interpretation, the openness relation can be viewed as a kind of ``one-sided compatibility'' between states relative to the propositions they accept and reject. Given two states $x, y$, we have that $y \op x$ if and only if there is no proposition $U$ that $x$ accepts and $y$ rejects.\footnote{For more on this characterization of the openness relation in fundamental frames, see \citealt[Remark~4.2]{Holliday2023}} 

In the next sections, we will make use of some canonical ways of embedding a fundamental lattice into the dual algebra of some fundamental frame. Here we recall two specific constructions.

\begin{theorem}[\citealt{Holliday2023}, Thm.~4.24.4] \label{prinfrm} Let $L$ be a fundamental lattice, $\mathrm{V}$ a join-dense subset of $L$, and $\Lambda$ a meet-dense subset of $L$. Then where
\begin{itemize}
\item  $X=\{(a,\neg a)\mid a\in\mathrm{V},a\neq 0\}\cup \{(1,b)\mid b\in \Lambda, b\neq 1\}$ and
\item $\op$ is the binary relation on $X$ defined by $(a,b)\op (c,d)$ iff $c\not\leq b$,
\end{itemize}
the following holds:
\begin{enumerate}
\item $(X,\op)$ is a fundamental frame,\footnote{In fact, $\op$ is not only pseudo-reflexive but reflexive.} and
\item the map
\[ a \mapsto a^\downarrow := \{(c,d) \in X \mid c \leq a\}\] is a fundamental embedding of $L$ into the dual algebra of $(X,\op)$ that preserves all existing meets and joins.
\end{enumerate}
\end{theorem}

Note that this result holds, in particular, when $\mathrm{V} = \Lambda = L$.

\begin{theorem}[\citealt{Holliday2023}, Thm.~4.30]\label{RefCanEmbed}
    Let $L$ be a fundamental lattice, and let the \emph{reflexive canonical frame} of $L$ be the relational frame $\mathsf{FI}(L) := (X, \op)$ where:
\begin{itemize}
    \item $X$ is the set of all pairs $(F,I)$ such that $F$ is a proper filter on $L$, $I$ is a proper ideal on $L$, $F\cap I=\varnothing$, and for each $a \in F$, we have $\neg a \in I$;
    \item The relation $\op$ on $X$ is given by the following condition for any $(F,I),(G,J)\in X$:
    \[(F,I)\op (G,J) \Leftrightarrow G \cap I = \varnothing.\]
\end{itemize}

Then $\mathsf{FI}(L)$ is a fundamental frame in which the relation $\op$ is reflexive. Moreover, the map
\[a \mapsto \widehat{a} := \{(F,I)\in X \mid a \in F\}\] is a fundamental embedding of $L$ into the dual algebra of $\mathsf{FI}(L)$.
\end{theorem}

\subsubsection{Modal Frames}

\citealt{Holliday2022,Holliday2024} studies fundamental logic and related logics extended with modalities. The semantics interprets $\Box$ using an accessibility relation $\shortdashrightarrow$ in the usual J\'{o}nsson-Tarski style (\citealt{Jonsson1952a}), but we must impose an interaction condition between accessibility and openness to ensure that the modal operator sends propositions to propositions.\footnote{\citealt{Holliday2022,Holliday2024} writes $(X,\op, R)$ where we would write $(X,\op,\shortdashrightarrow)$ in the notation of this paper, but it will be more convenient for our purposes here to write modal frames as  $(X,\op,\shortdashleftarrow)$.}

\begin{definition}\label{ModalFrameDef} A \textit{modal frame} is a triple $(X,\op,\shortdashleftarrow)$ where $X$ is a nonempty set and $\op$ and $\shortdashleftarrow$ are binary relations on $X$ satisfying the following condition (see Figure \ref{ModalFrameCon}):
\[\mbox{if $x \shortdashrightarrow y\po z$, then $\exists x'\op x$ $\forall x''\po x' $ $\exists y''$: $x''\shortdashrightarrow y''\po z$.}\]
A \textit{fundamental modal frame} is a modal frame $(X,\op,\shortdashleftarrow)$ such that $(X,\op)$ is a fundamental frame.
\end{definition}

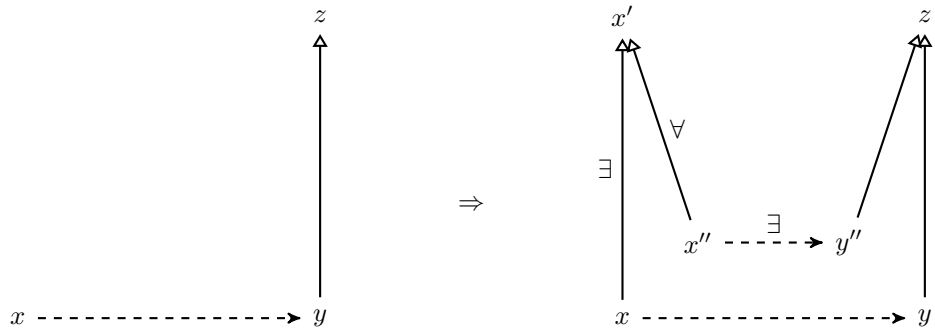
\begin{figure}[H]
\begin{center}
\begin{tikzpicture}[->,>={stealth'[length=6pt]},shorten >=1pt,shorten <=1pt, auto,node
distance=2cm,thick,every loop/.style={<-,shorten <=1pt}]
\tikzstyle{every state}=[fill=gray!20,draw=none,text=black]

\node (x) at (-1,0) {{$x$}};
\node (y) at (3,0) {{$y$}};
\node (z) at (3,4) {{$z$}};
\node at (5,1.5) {{\textit{$\Rightarrow$}}};

\path (x) edge[dashed,->] node {{}} (y);
\path[-{Triangle[open]},draw,thick]  (y) to node {{}} (z);

\node (x2) at (7,0) {{$x$}};
\node (x'2) at (7,4) {{$x'$}};
\node (x''2) at (8,1) {{$x''$}};
\node (y''2) at (10,1) {{$y''$}};
\node (y2) at (11,0) {{$y$}};
\node (z2) at (11,4) {{$z$}};

\path (x2) edge[dashed,->] node {{}} (y2);
\path (x''2) edge[dashed,->] node[above] {{$\exists$}} (y''2);
\path[-{Triangle[open]},draw,thick] (y2) to node {{}} (z2);
\path[-{Triangle[open]},draw,thick]  (x2) to node {{$\exists$}} (x'2);
\path[-{Triangle[open]},draw,thick]  (x''2) to node[right] {{$\forall$}} (x'2);
\path[-{Triangle[open]},draw,thick]  (y''2) to node {{}} (z2);

\end{tikzpicture}
\end{center}
\caption{Illustration of the modal frame condition in Definition \ref{ModalFrameDef}. A solid line from $w$ to $v$ indicates $w\po v$, and a dashed line from $w$ to $v$ indicates $w\shortdashrightarrow v$.}\label{ModalFrameCon}
\end{figure}

\begin{lemma} For any modal frame $(X,\op,\shortdashleftarrow)$ and proposition $U$,
\[\Box_{\shortdashleftarrow} U=\{x\in X\mid \mbox{for all } y \shortdashleftarrow x,\,  y\in U\}\]
is also a proposition, so we regard $\Box_{\shortdashleftarrow}$ as a map from propositions to propositions.
\end{lemma}

When $\shortdashleftarrow$ is clear from context, we write $\Box U$ instead of $\Box_{\shortdashleftarrow} U$.

\begin{definition} The \textit{dual modal algebra} of a modal frame $(X,\op,\shortdashleftarrow)$ is the algebra $(L,\neg_{\op},\Box_{\shortdashleftarrow})$ where $(L,\neg_{\op})$ is the dual algebra of $(X,\op)$.
\end{definition}

Note that standard relational semantics for classical modal logic is a special case of the definitions above. Any such relational frame $(X, \shortdashleftarrow)$ can be viewed as a modal frame $(X, \op, \shortdashleftarrow)$, where $\op$ is the identity on $X$. The condition in Figure~\ref{ModalFrameCon} is then trivially satisfied, and the dual algebra is just the complex algebra $(\Po(X), \Box_{\shortdashleftarrow})$ of the relational frame $(X, \shortdashleftarrow)$.

Finally, let us recall that fundamental modal frames provide canonical representations of fundamental necessity lattices.

\begin{theorem}[\citealt{Holliday2024}, Thm.~3.10, Thm.~5.3]\label{FunModalRep}
    Let $(L, \neg, \Box)$ be a fundamental necessity lattice. The \textit{canonical modal frame} of $(L, \neg,\Box)$ is defined as the triple $(X, \op, \shortdashleftarrow)$ where:
    \begin{itemize}
        \item $X$ is the set of all pairs $(F,I)$ of a filter $F$ and an ideal $I$ of $L$ such that  for any $a \in L$, $a \in F$ implies $\neg a \in I$;
        \item $(F,I) \op (G,J)$ iff $G \cap I = \varnothing$;
        \item $(F,I) \shortdashleftarrow (G,J)$ iff for each $a\in L$, $\Box a \in G$ implies $a \in F$.
    \end{itemize}
Then $(X, \op, \shortdashleftarrow)$ is a fundamental modal frame. Moreover, the map \[a \mapsto \widehat{a} := \{(F,I) \in X \mid a \in F\}\] is a fundamental modal embedding of $(L,\neg,\Box)$ into $(\chi_{\op}(X), \neg_{\op},\Box_{\shortdashleftarrow})$, the dual modal algebra of $(X, \op, \shortdashleftarrow)$.
\end{theorem}

\section{The GMT Translation into Ortho-\textsf{S4}}\label{OS4Section}

In this section, we establish our first modal translation for fundamental logic. Our motivation is the following. The standard G\"odel-McKinsey-Tarski translation of intuitionistic logic into $\mathsf{S4}$ gives, intuitively, a way for the classical logician to understand the intuitionistic logician. Given a pair of formulas $(\phi,\psi)$ of the language of intuitionistic logic, the classical logician can determine whether $\psi$ is an intuitionistic consequence of $\phi$ by first translating the two formulas into modal formulas $\tau(\phi)$ and $\tau(\psi)$ and then checking whether $\tau(\psi)$ is a consequence of $\tau(\phi)$ over $\mathsf{S4}$. Our aim is to generalize this approach beyond the distributive setting. The target logic of our translation is therefore not $\mathsf{S4}$, but rather the natural orthological counterpart to $\mathsf{S4}$, which we call $\mathsf{OS4}$. We start by introducing $\mathsf{OS4}$ and characterizing it as the logic of orthomodal preordered frames, before presenting a correspondence between \textbf{OS4}-frames and fundamental frames that will induce a translation of fundamental logic into $\mathsf{OS4}$.

\subsection{Ortho-\textsf{S4}}

From the semantic perspective, the modal logic $\mathsf{S4}$ can be both understood as the logic of modal algebras $(B,\Box)$ where $\Box$ is an interior operator (i.e., a non-increasing, idempotent necessity operator), and as the logic of reflexive and transitive relational frames. Our first task is to transfer these notions from the classical to the orthological setting.

\begin{definition}
    An \textit{\textbf{OS4}-lattice} is an orthomodal  lattice $(O, \neg, \Box)$ as in Definition~\ref{FunNecLat} such that for all $a\in O$, $\Box a \leq a$ and $\Box  a \leq \Box\Box a$.
\end{definition}

In orthomodal logic, the duality between $\Box$ and $\diamondsuit$ is preserved, since $\diamondsuit a$ is defined as $\neg \Box \neg a$. However, because orthocomplements are not pseudo-complements, many validities of a classical modal logic may fail in its orthomodal counterpart. Consider the following example in the case of $\mathsf{S4}$ and $\mathsf{OS4}$. In any \textbf{S4}-algebra $(B,\Box)$, one can easily check that $\Box\diamondsuit\Box a \land \Box \diamondsuit \Box b = \Box \diamondsuit \Box (a \land b)$ for any $a,b \in B$. However, this equality fails for \textbf{OS4}-lattices, as the counterexample in Figure~\ref{OS4LatFig} shows.

\begin{figure}[H]
\begin{center}
\begin{tikzpicture}[->,>={stealth'[length=6pt]},shorten >=1pt,shorten <=1pt, auto,node
distance=2cm,semithick,every loop/.style={<-,shorten <=1pt}]
\tikzstyle{every state}=[fill=gray!20,draw=none,text=black]

\node  (1) at (3,4) {{$1$}};
\node (x) at (0,2) {{$\neg a$}};
\node  (y) at (2,2) {{$a$}};
\node (y') at (4,2) {{$b$}};
\node  (z) at (6,2) {{$\neg b$}};
\node (0) at (3,0) {{$0$}};
\path (1) edge[-] node {{}} (y);
\path (1) edge[-] node {{}} (y');
\path (1) edge[-] node {{}} (x);
\path (1) edge[-] node {{}} (z);
\path (x) edge[-] node {{}} (0);
\path (y) edge[-] node {{}} (0);
\path (y') edge[-] node {{}} (0);
\path (z) edge[-] node {{}} (0);

\path (y) edge[loop left, min distance=1.25cm, dashed] node {{}} (y);
\path (y') edge[loop right, min distance=1.25cm, dashed] node {{}} (y');
\path (x) edge[->, bend right, dashed] node {{}} (0);
\path (z) edge[->, bend left, dashed] node {{}} (0);
\path (0) edge[loop below, min distance=1.25cm, dashed] node {{}} (0);
\path (1) edge[loop above, min distance=1.25cm, dashed] node {{}} (1);
\end{tikzpicture}
\end{center}
\caption{An \textbf{OS4}-lattice in which $\Box\diamondsuit\Box a\land\Box\diamondsuit\Box b=1$ and $\Box\diamondsuit\Box (a\land b)=0$. Solid edges give the Hasse diagram of the ortholattice, and dashed arrows represent the $\Box$ operation.}\label{OS4LatFig}
\end{figure}
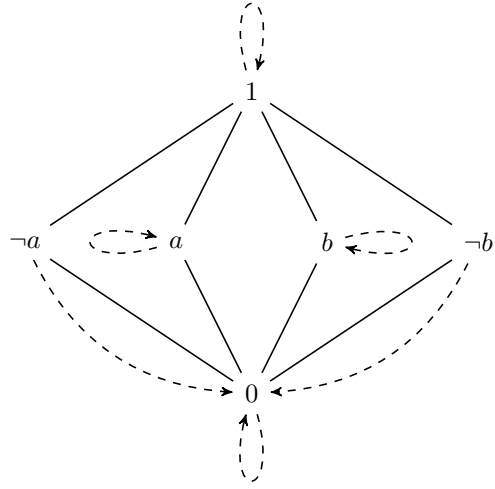

Let us now consider a particular class of modal frames. Whenever $(X,\op, \shortdashleftarrow)$ is a modal frame such that $\op$ is symmetric, we will call it an \emph{orthomodal frame}, and we use the symbol $\top$ to denote the relation $\op$ in a way that emphasizes its symmetry.

\begin{definition}\label{OS4-frame}
    An \textit{\textbf{OS4}-frame} is a modal frame $(X,\top,\leq)$ as in Definition \ref{ModalFrameDef} such that:
    \begin{enumerate}
        \item $\top$ is a reflexive and symmetric relation on $X$;
        \item $\leq$ is a reflexive and transitive relation on $X$.
    \end{enumerate}
\end{definition}

When $x\top y$, we say that $x$ is \textit{compatible with} $y$. Intuitively, this means that nothing supported at $x$ is rejected at $y$ or vice versa. When $x\leq y$, we say that $x$ is an \textit{extension of} $y$. Intuitively, this means that facts about what is \textit{known} or \textit{proved} are preserved from $y$ to $x$. However, facts about what is unknown or unproved are not necessarily preserved. As a result,  $x\top y\leq z$ does not in general imply $x\top z$; for example, $z$ could settle that some proposition $P$ is unknown, which leads to incompatibility with $x$, according to which $P$ \textit{is} known, whereas the extension $y$ of $z$ achieves compatibility with $x$ by agreeing that $P$ is known.

Given an \textbf{OS4}-frame $(X, \top, \leq)$, the dual algebra induced by $\top$ forms a modal ortholattice $(\chi_\top(X),\Box_\leq) $, with the modal operation $\Box_\leq$ defined by \[\Box_\leq U = \{x \in X \mid \mbox{for all } y\leq x, \, y \in U\}\] for any $U \in \chi_\top(X)$. It is routine to check that $(\chi_\top(X), \Box_\leq)$ is an \textbf{OS4}-lattice.

Recall the definition of $\mathsf{OS4}$ from Section~\ref{LogicsSec}.

\begin{theorem}
    The logic $\mathsf{OS4}$ is sound and complete with respect to \textbf{OS4}-lattices and \textbf{OS4}-frames.
\end{theorem}

\begin{proof}
    It is routine to check that $\mathsf{OS4}$ is valid on any \textbf{OS4}-lattice. Since the dual algebra of any \textbf{OS4}-frame is an \textbf{OS4}-lattice, this means that $\mathsf{OS4}$ is also sound with respect to \textbf{OS4}-frames. For completeness, we use standard techniques. First, the Lindenbaum algebra of $\mathsf{OS4}$ is easily seen to be an \textbf{OS4}-lattice, from which the completeness of $\mathsf{OS4}$ with respect to \textbf{OS4}-lattices immediately follows. Second, one can verify that the reflexive and symmetric canonical modal frame of any \textbf{OS4}-lattice $(O,\neg,\Box)$ (i.e., the restriction of the canonical modal frame of $(O,\neg,\Box)$ to pairs $(F,I)$ such that $F\cap I = \varnothing$ and $a \in F$ iff $\neg a \in I$) is an \textbf{OS4}-frame and that $(O, \neg, \Box)$ embeds into its dual algebra via the usual mapping. Applied to the Lindenbaum algebra of $\mathsf{OS4}$, this yields the completeness of the logic with respect to \textbf{OS4}-frames.
\end{proof}

\subsection{From \textbf{OS4}-Frames to Fundamental Frames}

Let us now establish a connection between $\mathsf{OS4}$ and fundamental logic by studying the relationship between \textbf{OS4}-frames and fundamental frames. We start with a canonical way of turning the former into the latter.

\begin{definition}
    Given an \textbf{OS4}-frame $(X, \top, \leq)$, we define its \textit{induced openness relation} $\op$ on $X$ by 
    \[x \op y \Leftrightarrow \exists z: x \top z \leq y.\]
\end{definition}

Note that the interaction condition between $\top$ and $\leq$ in \cref{ModalFrameDef} can be rewritten using the openness relation $\op$ as follows:

\[ z \op x \Rightarrow \exists x' \top x \, \forall x'' \top x': z \op x''.\]

\begin{lemma}\label{funredlma}
    Let $(X, \top, \leq)$ be an \textbf{OS4}-frame. Then $(X, \op)$ is a fundamental frame.
\end{lemma}

\begin{proof}
    We must check pseudo-reflexivity and pseudo-symmetry. Pseudo-reflexivity is clear, since both $\top$ and $\leq$ are reflexive. For pseudo-symmetry, suppose $x \op y$. We must find $y' \leq_{\op} y$ such that $y' \op x$. Since $x \op y$, there is $z \in X$ such that $x \top z \leq y$. We claim that $z$ is the required witness. Note first that $z \top x \leq x$, so $z \op x$. Hence we only need to show that $z \leq_{\op} y$. Suppose $x' \op z$. This means  there is a $z' \in X$ such that $x' \top z' \leq z$. Since $z \leq y$, it follows that $x' \top z' \leq y$, so $x' \op y$. This shows that $z \leq_{\op} y$, which completes the proof that $(X, \op)$ is a fundamental frame.
\end{proof}

Lemma~\ref{funredlma} motivates the following definition, illustrated with an example in Figure~\ref{FunReductFig}.

\begin{definition}
    Let $(X, \top, \leq)$ be an \textbf{OS4}-frame. The \textit{fundamental reduct} of $(X, \top, \leq)$ is the fundamental frame $(X, \op)$ where $\op$ is the induced openness relation of $(X, \top, \leq)$.
\end{definition}

\begin{figure}[H]
\begin{center}
\begin{tikzpicture}[->,>={stealth'[length=6pt]},shorten >=1pt,shorten <=1pt, auto,node
distance=2cm,thick,every loop/.style={<-,shorten <=1pt}]
\tikzstyle{every state}=[fill=gray!20,draw=none,text=black]

\node (0) at (2,0) {{$y$}};
\node (1) at (2,-4) {{$v$}};
\node (2) at (-2,-4) {{$w$}};
\node (3) at (-2,0) {{$x$}};

\path[{Triangle[open]}-{Triangle[open]},draw,thick] (0) to node {{}} (2);
\path[{Triangle[open]}-{Triangle[open]},draw,thick] (0) to node {{}} (3);
\path[{Triangle[open]}-{Triangle[open]},draw,thick] (2) to node {{}} (1);
\path[{Triangle[open]}-{Triangle[open]},draw,thick] (3) to node {{}} (1);

\path (0) edge[dashed,<->, bend right] node {{}} (3);

\end{tikzpicture}\qquad\qquad\qquad \begin{tikzpicture}[->,>={stealth'[length=6pt]},shorten >=1pt,shorten <=1pt, auto,node
distance=2cm,thick,every loop/.style={<-,shorten <=1pt}]

  \tikzstyle{directed_edge}=[-{Triangle[open]},draw,thick]
  \tikzstyle{edge}=[{Triangle[open]}-{Triangle[open]},draw,thick]

\node (0) at (2,0) {{$y$}};
\node (1) at (2,-4) {{$v$}};
\node (2) at (-2,-4) {{$w$}};
\node (3) at (-2,0) {{$x$}};

  \draw[directed_edge] (0) to (1);
  \draw[edge] (0) to (2);
  \draw[edge] (0) to (3);
  \draw[edge] (1) to (2);
  \draw[edge] (1) to (3);
  \draw[directed_edge] (3) to (2);
\end{tikzpicture}
\end{center}
\caption{An \textbf{OS4}-frame (left) and its fundamental reduct (right). A solid line from $w$ to $v$ indicates $v\top w$ on the left and $v\op w$ on the right; a dashed line from $w$ to $v$ on the left indicates $v\leq w$. The dual algebra of the \textbf{OS4}-frame is the $\textbf{OS4}$-lattice in Figure \ref{OS4LatFig}.}\label{FunReductFig}
\end{figure}
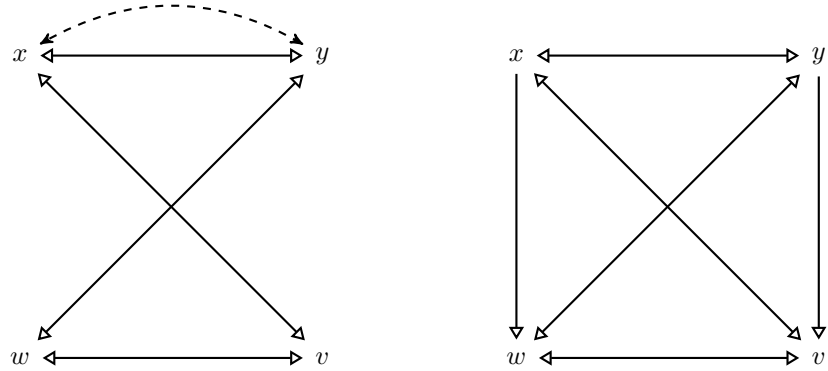

The relationship between the dual algebra $\chi_\top(X)$ of an \textbf{OS4}-frame $(X, \top, \leq)$ and the dual algebra $\chi_{\op}(X)$ of its fundamental reduct $(X, \op)$ is given by the following lemma.

\begin{lemma}
    Let $(X, \top, \leq)$ be an \textbf{OS4}-frame and $(X, \op)$ its fundamental reduct. Then:
    \begin{enumerate}
        \item $\chi_{\op} (X) = ran(\Box_\leq)$;
        \item The map $\Box_\leq: \chi_\top(X) \to \chi_{\op} (X)$ is right adjoint to the inclusion $\iota: \chi_{\op} (X) \to \chi_\top(X)$.
    \end{enumerate}
\end{lemma}

\begin{proof}
 Note that since $(\chi_\top(X), \Box_\leq)$ is an \textbf{OS4}-lattice, $\Box_\leq$ is an interior operator. By standard order-theoretic considerations, this means that the map $\Box_\leq: \chi_\top(X) \to ran(\Box_\leq)$ is right adjoint to the inclusion map $\iota: ran(\Box_\leq) \to \chi_\top(X)$. Hence part $2$ follows immediately from part $1$. For the proof of part $1$, suppose first that $U$ is of the form $\Box_\leq V$ for some $V \in \chi_\top(X)$, and suppose that $x \notin U$. Then there is a $y \leq x$ such that $y \notin V$. Since $V \in \chi_\top(X)$, this means there is a $z \top y$ such that $z' \notin V$ for any $z' \top z$. Now notice that $z \op x$. We claim that whenever $z \op w$, $w \notin U$. Indeed, if there is a $w \in U$ such that $z \op w$, then there is a $z'$ such that $z \top z' \leq w$. But $z' \leq w$ implies that $z' \in V$, contradicting the fact that $z \top z'$ implies $z' \notin V$. This shows that $U \in \chi_{\op}(X)$. For the converse, suppose that $U \in \chi_{\op}(X)$. Note that this implies $U = \neg_{\op} \neg_{\po} U$. We claim that $U = \Box_\leq \neg_\top \neg_{\po} U$. Clearly, this is enough to conclude that $U \in ran(\Box_\leq)$. Now note the following chain of equivalences:
 \begin{align*}
     x \in \Box_\leq \neg_\top \neg_{\po} U &\Leftrightarrow \forall y \, \forall z (z \top y \leq x \Rightarrow z \notin \neg_{\po} U) \\
     &\Leftrightarrow \forall z \op x,\, z \notin \neg_{\po} U \\
     &\Leftrightarrow x \in \neg_{\op} \neg_{\po} U.
 \end{align*}
This completes the proof.
\end{proof}

Let us now investigate further the relationship between the two relations on an \textbf{OS4}-frame and the fundamental relation on its fundamental reduct. We start from the following definition.

\begin{definition}
    Let $(X, \op)$ be a fundamental frame. The \textit{symmetric kernel} of $\op$ is the relation $\opo \,=\, \op \cap \po$.
\end{definition}

The following observation can be extracted from the proof of \cref{funredlma}.

\begin{remark} \label{rmk1}
    Let $(X, \top, \leq)$ be an \textbf{OS4}-frame and $(X, \op)$ its fundamental reduct. Then $\top$ is a subrelation of the symmetric kernel of $\op$, and $\leq$ is a subrelation of the prerefinement relation $\leq_{\op}$ (recall Lemma~\ref{PrerefineLem}).
\end{remark}

\begin{proof}
    Suppose $x \top y$. Then we have that $x \top y \leq y$ and $y \top x \leq x$, which implies that $x \op y$ and $y \op x$. Moreover, suppose that $x \leq y$, and let $z \op x$. Then there is a $z' \in X$ such that $z \top z' \leq x \leq y$, so $z \op y$. This shows that $x \leq_{\op} y$.
\end{proof}

Given this observation, one may wonder when the relation $\top$ coincides with the symmetric kernel of $\op$, and when the relation $\leq$ coincides with the prerefinement $\leq_{\op}$. This leads to the following definitions.

\begin{definition}\label{SkelReg}
    Let $(X, \top, \leq)$ be an \textbf{OS4}-frame.
    \begin{enumerate}
        \item\label{Skel} $(X, \top ,\leq)$ is \textit{skeletal} if for any $x, y \in X$, if there are $z_1, z_2$ such that $x \top z_1 \leq y$ and $y \top z_2 \leq x$, then $x \top y$ (see Figure~\ref{SkelFig}).
        \item $(X, \top, \leq)$ is \textit{regular} if for any $x \in X$, $\dnset x = \{y \in X \mid y \leq x\} \in \chi_\top(X)$.
    \end{enumerate}
\end{definition}

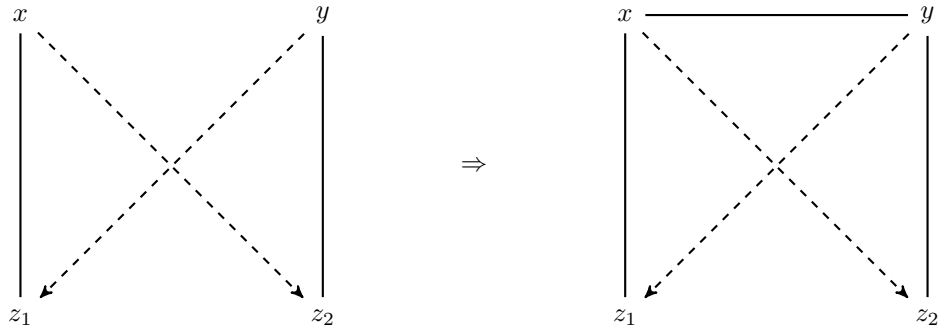
\begin{figure}[H]

\begin{center}
\begin{tikzpicture}[->,>={stealth'[length=6pt]},shorten >=1pt,shorten <=1pt, auto,node
distance=2cm,thick,every loop/.style={<-,shorten <=1pt}]
\tikzstyle{every state}=[fill=gray!20,draw=none,text=black]

\node (x) at (0,0) {{$x$}};
\node (z_1) at (0,-4) {{$z_1$}};
\node (y) at (4,0) {{$y$}};
\node (z_2) at (4,-4) {{$z_2$}};

\path (x) edge[-] node {{}} (z_1);
\path (y) edge[-] node {{}} (z_2);
\path (x) edge[->,dashed] node {{}} (z_2);
\path (y) edge[->,dashed] node {{}} (z_1);

\node at (6,-2) {{\textit{$\Rightarrow$}}};

\node (x') at (8,0) {{$x$}};
\node (z_1') at (8,-4) {{$z_1$}};
\node (y') at (12,0) {{$y$}};
\node (z_2') at (12,-4) {{$z_2$}};

\path (x') edge[-] node {{}} (y');
\path (x') edge[-] node {{}} (z_1');
\path (y') edge[-] node {{}} (z_2');
\path (x') edge[->,dashed] node {{}} (z_2');
\path (y') edge[->,dashed] node {{}} (z_1');

\end{tikzpicture}
\end{center}
\caption{The skeletal condition of Definition \ref{SkelReg}.\ref{Skel}.}\label{SkelFig}
\end{figure}

\begin{example} The $\mathbf{OS4}$ frame on the left of Figure \ref{FunReductFig} is not regular, because $\mathord{\downarrow}x=\{x,y\}\not\in \chi_\top(X)$.
\end{example}

\begin{lemma} \label{backlma}
    Let $(X, \top, \leq)$ be an \textbf{OS4}-frame and $(X, \op)$ its fundamental reduct. Then:
    \begin{enumerate}
        \item $(X, \top, \leq)$ is skeletal iff $\top$ coincides with the symmetric kernel of $\op$;
        \item $(X, \top, \leq)$ is regular iff $\leq$ coincides with $\leq_{\op}$.
    \end{enumerate}
\end{lemma}

\begin{proof}
    For part $1$, the definition of skeletal is clearly equivalent to the symmetric kernel of $\op$ being a subrelation of $\top$. Together with \cref{rmk1}, this yields the desired result. For part $2$, by \cref{rmk1} again, it is enough to show that $\leq_{\op}$ is a subrelation of $\leq$ if and only if $(X, \top, \leq)$ is regular. For the right-to-left direction, suppose that $x \nleq y$, so $x \notin \dnset y$. Since $(X, \top, \leq)$ is regular, there is a $z\top x$ such that $z' \notin \dnset y$ for any $z'\top z$. But this means that $z\op x$ and not $z \op y$, establishing that $x \not \leq_{\op} y$. For the converse direction, assume that $\forall y \top z \, \exists y' \top y: y' \leq x$. We must show that $z \leq x$ to show that $\mathord{\downarrow} x \in \chi_\top(X)$. Since, by assumption, $\leq_{\op}$ is a subrelation of $\leq$, it is enough to show that $w \op x$ whenever $w \op z$. So assume that $w \op z$. Since $(X, \top, \leq)$ is an $\mathbf{OS4}$-frame, there is an $x' \top z$ such that $w \op x''$ for any $x'' \top x'$. By our assumption on $z$, we have an $x'' \top x'$ such that $ x'' \leq x$. But then we have some $q \in X$ such that $w \top q \leq x'' \leq x$, so $w \op x$, as desired. 
\end{proof}

\subsection{From Fundamental Frames to \textbf{OS4}-Frames}

We have shown how to define a fundamental frame $(X, \op)$ from an \textbf{OS4}-frame $(X, \top, \leq)$ in such a way that the dual algebra of $(X, \op)$ is precisely the sublattice of $\Box_\leq$-fixpoints of the dual \textbf{OS4}-lattice of $(X, \top, \leq)$. Let us now investigate how to construct \textbf{OS4}-frames from fundamental frames. Recall first that, given a fundamental frame $(X, \op)$, the \textit{postrefinement} relation induced by $\op$ is the relation $\leq_{\po}$ such that $y \leq_{\po} x$ iff $\forall z \, (y \op z \Rightarrow x \op z)$ for any $x, y \in X$.

\begin{definition}
    Let $(X, \op)$ be a fundamental frame such that $\op$ is reflexive, and let $\opo$, $\leq_{\op}$, and $\leq_{\po}$ be the symmetric kernel of $\op$, prerefinement relation, and postrefinement relation, respectively.
    \begin{enumerate}
        \item $(X, \op)$ is \textit{prefactoring} if $y \op x$ implies $\exists z: y \opo z \leq_{\op} x$;
        \item $(X, \op)$ is \textit{postfactoring} if $x \op y$ implies $\exists z: y \opo z \leq_{\po} x$.
    \end{enumerate}
    Finally, $(X, \op)$ is \textit{balanced} if it is both prefactoring and postfactoring.
\end{definition}

Note that in any fundamental frame $(X, \op)$, $\exists z: y \opo z \leq_{\op} x$ implies $y \op x$, and $\exists z: y \opo z \leq_{\po} x$ implies $x \op y$. This explains the sense in which fundamental frames in which the converse holds are ``factoring'': the fundamental relation $\op$ can be factored as a composition of its symmetric kernel and the prerefinement (resp.~postrefinement) preorder it induces. Note also that being prefactoring is a natural strengthening of pseudo-symmetry, which states that $y \op x$ implies that there is $z$ such that $y \po z \leq_{\op} x$.

Intuitively, balanced frames can be understood as those fundamental frames where positive and negative information carried by states can always be separated. Recall that if $y \op x$, then the propositions supported at $x$ are not rejected at $y$. Whenever this happens, the prefactoring condition guarantees that there is a state containing at least as much positive information as $x$ that also does not reject any proposition satisfied by $y$. Dually, the postfactoring condition guarantees that there is also a state containing at least as much negative information as $y$ that also does not satisfy anything rejected at $x$. Let us now see what role these two properties play in constructing an \textbf{OS4}-frame out of a fundamental frame.

\begin{lemma}\label{orthmodcomp}
    Let $(X, \op)$ be a prefactoring fundamental frame. Then $(X, \opo, \leq_{\op})$ is an \textbf{OS4}-frame iff $(X, \op)$ is balanced.
\end{lemma}

\begin{proof}
    Assume that $(X, \op)$ is prefactoring. Since $\opo$ is reflexive and symmetric, and $\leq_{\op}$ is a preorder, we only need to check that the following condition holds if and only if $(X, \op)$ is postfactoring:
    \[z \op x \Rightarrow \exists x' \opo x \, \forall x'' \opo x': z \op x''.\]
Suppose first that $(X, \op)$ is postfactoring, and assume that $z \op x$. This implies there is a $y$ such that $x \opo y \leq_{\po} z$. Clearly, for any $x''$ such that $x'' \opo y$, we have that $y \op x''$ and hence also $z \op x''$. But this means that $y$ is the required witness $x'$. Conversely, suppose  that the condition above holds, and let $x \op y$. By assumption, we have an $x' \opo y$ such that for all $x'' \opo x'$, $x \op x''$. We claim that $x' \leq_{\po} x$. Suppose that $x' \op z$. Since $(X, \op)$ is prefactoring, there is a $z' \opo x'$ such that $z' \leq_{\op} z$. By our assumption on $z'$, this means that $x \op z' \leq_{\op} z$, so $x \op z$, as desired.\end{proof}

Lemma~\ref{orthmodcomp} motivates the following definition, illustrated with an example in Figure~\ref{BalCompFig}.

\begin{definition}
    Let $(X, \op)$ be a balanced fundamental frame. The \textit{orthomodal companion} of $(X, \op)$ is the \textbf{OS4}-frame $(X, \opo, \leq_{\op})$.
\end{definition}

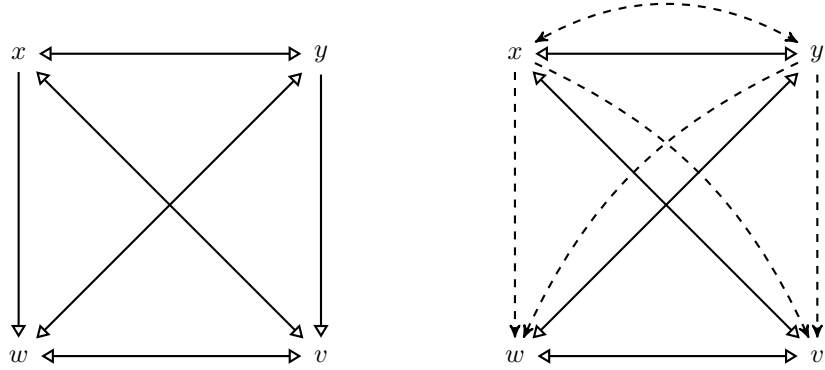
\begin{figure}[H]
\begin{center}
 \begin{tikzpicture}[->,>={stealth'[length=6pt]},shorten >=1pt,shorten <=1pt, auto,node
distance=2cm,thick,every loop/.style={<-,shorten <=1pt}]

  \tikzstyle{directed_edge}=[-{Triangle[open]},draw,thick]
  \tikzstyle{edge}=[{Triangle[open]}-{Triangle[open]},draw,thick]

\node (0) at (2,0) {{$y$}};
\node (1) at (2,-4) {{$v$}};
\node (2) at (-2,-4) {{$w$}};
\node (3) at (-2,0) {{$x$}};

  \draw[directed_edge] (0) to (1);
  \draw[edge] (0) to (2);
  \draw[edge] (0) to (3);
  \draw[edge] (1) to (2);
  \draw[edge] (1) to (3);
  \draw[directed_edge] (3) to (2);
\end{tikzpicture}\qquad\qquad\qquad  \begin{tikzpicture}[->,>={stealth'[length=6pt]},shorten >=1pt,shorten <=1pt, auto,node
distance=2cm,thick,every loop/.style={<-,shorten <=1pt}]
\tikzstyle{every state}=[fill=gray!20,draw=none,text=black]

\node (0) at (2,0) {{$y$}};
\node (1) at (2,-4) {{$v$}};
\node (2) at (-2,-4) {{$w$}};
\node (3) at (-2,0) {{$x$}};

\path[{Triangle[open]}-{Triangle[open]},draw,thick] (0) to node {{}} (2);
\path[{Triangle[open]}-{Triangle[open]},draw,thick] (0) to node {{}} (3);
\path[{Triangle[open]}-{Triangle[open]},draw,thick] (2) to node {{}} (1);
\path[{Triangle[open]}-{Triangle[open]},draw,thick] (3) to node {{}} (1);

\path (0) edge[dashed,->] node {{}} (1);
\path (0) edge[dashed,->, bend right=20] node {{}} (2);
\path (0) edge[dashed,<->, bend right] node {{}} (3);
\path (3) edge[dashed,->, bend left=20] node {{}} (1);
\path (3) edge[dashed,->] node {{}} (2);
\end{tikzpicture}
\end{center}
\caption{A balanced fundamental frame (left) and its orthomodal companion (right). Note that the fundamental frame is the same as on the right of Figure \ref{FunReductFig}.}\label{BalCompFig}
\end{figure}

\begin{lemma} \label{keylma}
    Let $(X, \op)$ be a balanced fundamental frame and $(X, \opo, \leq_{\op})$ its orthomodal companion with dual \textbf{OS4}-lattice $(\chi_{\opo}(X),\Box_{\leq_{\op}})$. Then:

    \begin{enumerate}
        \item $\chi_{\op} (X) \sset \chi_{\opo}(X)$;
        \item for any $A \in \chi_{\opo}(X)$, $\Box_{\leq_{\op}} \neg_{\opo} A = \neg_{\op} A$;
        \item the map $\Box_{\leq_{\op}}$ is right-adjoint to the inclusion $\iota: \chi_{\op}(X) \to \chi_{\opo}(X)$, and $ran(\Box_{\leq_{\op}}) = \chi_{\op} (X)$.
    \end{enumerate}
\end{lemma}

\begin{proof}
     For part $1$, it is enough to show that for any $A \in \chi_{\op}(X)$, $\neg_{\opo} \neg_{\opo} A \sset \neg_{\op} \neg_{\po} A$. We argue by contraposition. Suppose that there is a $y \op x$ such that $z \notin A$ for all $z \po y$. By postfactoring, there is a $y' \in X$ such that $x \opo y' \leq_{\po} y$. Now for any $z$ such that $z \opo y'$, we have in particular that $y' \op z$, so $y \op z$. But this implies that $y' \in \neg_{\opo} A$, and therefore $x \notin \neg_{\opo} \neg_{\opo} A$. Contraposing, we obtain that $\neg_{\opo} \neg_{\opo} A \sset \neg_{\op} \neg_{\po} A$.

     For part $2$, fix $A \in \chi_{\opo}(X)$. Let us first show that $\Box_{\leq_{\op}} \neg_{\opo} A \sset \neg_{\op} A$. Assume $x \in \Box_{\leq_{\op}} \neg_{\opo} A$, and let $y \op x$. By prefactoring, we have an $x'$ such that $y \opo x' \leq_{\op} x$. By our assumption on $x$, we get that $x' \in \neg_{\opo} A$ and thus that $y \notin A$. Hence $x \in \neg_{\op} A$. Conversely, suppose that $x \in \neg_{\op} A$ and that we have $z \opo y \leq_{\op} x$. Then in particular $z \op y$, which, together with $y \leq_{\op} x$, implies that $z \op x$. Hence $z \notin A$. But this means that $y \in \neg_{\opo} A$ whenever $y \leq_{\op} x$, and thus $x \in \Box_{\leq_{\op}} \neg_{\opo} A$, as desired.

     Finally, for part $3$, note first that $\iota$ is well defined by part $1$. Moreover, by part $2$, the codomain of $\Box_{\leq_{\op}}$ is a subset of $\chi_{\op}(X)$, since any $B \in \chi_{\opo}(X)$ is of the form $\neg_{\opo} A$ for some $A \in \chi_{\opo}(X)$. Note also that the right adjoint of an injective map is always surjective. Hence we only need to verify that whenever $A \in \chi_{\op}(X)$ and $B \in \chi_{\opo} (X)$, we have that $A \sset B$ iff $A \sset \Box_{\leq_{\op}} B$. The right-to-left direction is clear since $\leq_{\op}$ is reflexive, so assume that $A \sset B$, and let $y \leq_{\op} x \in A$. Since elements of $\chi_{\op}(X)$ are closed under prerefinements, we have that $y \in A \sset B$. Hence $A \sset \Box_{\leq_{\op}} B$, which completes the proof.
\end{proof}

The previous result establishes that we can always realize the dual algebra of a balanced fundamental frame as the fixpoints of the $\Box$ operator on the dual \textbf{OS4}-lattice of its orthomodal companion. Let us conclude with the following observation regarding orthomodal companions.

\begin{lemma} \label{reglma}
    For any balanced fundamental frame $(X, \op)$, its orthomodal companion $(X, \opo, \leq_{\op})$ is skeletal and regular.
\end{lemma}

\begin{proof}
    Clearly, if $x \opo z_1 \leq_{\op} y$ and $y \opo z_2 \leq_{\op} x$, we have that $x \op y$ and $y \op x$. This shows that $(X,\opo, \leq_{\op})$ is skeletal. For regularity, letting $\dnset{x} = \{ y \in X \mid y \leq_{\op} x\}$ for any $x \in X$, it is enough to show that for any $x \in X$, $\neg_{\opo} \neg_{\opo} \dnset x \sset \dnset x$. Arguing by contraposition, suppose that $y \not \leq_{\op} x$. Then there is $z \op y$ for which it is not the case that $z\op x$. By postfactoring, there is a $z'$ such that $y \opo z' \leq_{\po} z$. Now we claim that $z' \in \neg_{\opo} \dnset x$. If true, this shows that $y \notin \neg_{\opo} \neg_{\opo} \dnset x$, which completes the proof. So suppose that $w \opo z'$. In particular, we have $w \po z' \leq_{\po} z$, so $z \op w$. Since we also have that not $z \op x$, it follows that $w \not \leq_{\op} x$, so $w \notin \dnset x$.\end{proof}

\begin{proposition}\label{corcor}
    There is a one-to-one correspondence between balanced fundamental frames and skeletal, regular \textbf{OS4}-frames, obtained by mapping any balanced fundamental frame to its orthomodal companion and mapping any skeletal, regular \textbf{OS4}-frame to its fundamental reduct.
\end{proposition}

\begin{proof}
    If $(X,\op)$ is a balanced fundamental frame, then its orthomodal companion is skeletal and regular by \cref{reglma}. Moreover, since $(X, \op)$ is prefactoring, it is also the fundamental reduct of its orthomodal companion. Conversely, if $(X, \top, \leq)$ is a skeletal and regular \textbf{OS4}-frame, then its fundamental reduct $(X, \op)$ is prefactoring by definition, and its orthomodal companion is precisely $(X, \top, \leq)$ by \cref{backlma}. Hence we only need to verify that $(X, \op)$ is balanced. But this follows immediately from the fact that $(X, \top, \leq)$ is the orthomodal companion of $(X, \op)$ together with \cref{orthmodcomp}.
\end{proof}

\subsection{Translating Fundamental Logic into Modal Orthologic}

Let us now use the results from the previous subsections to prove results about the following   translation $\mu$ from the language $\mathcal{L}(\wedge,\vee,\neg)$ of fundamental logic to the language $\mathcal{L}(\wedge,\vee,\neg,\Box)$ of orthomodal logic:

\begin{itemize}
    \item $\mu(\top)=\top$ and $\mu(p) = \Box p$;
    \item $\mu(\phi \me \psi) = \mu(\phi) \me \mu(\psi)$;
    \item $\mu(\phi \jo \psi) = \mu(\phi) \jo \mu(\psi)$;
    \item $\mu(\neg \phi) = \Box \neg \mu(\phi)$.
\end{itemize}

Note that $\mu$ is exactly the reduct of the G\"odel-McKinsey-Tarski translation to the implication-free fragment of intuitionistic logic. The proof of the following is routine.

\begin{lemma}
The map $\mu$ is a faithful translation of fundamental logic into $\mathsf{OS4}$: if $\phi \vdash_\mathsf{F} \psi$, then ${\mu(\phi) \vdash_\mathsf{OS4} \mu(\psi)}$.
\end{lemma}

 We now prove that $\mu$ is also a full translation, meaning that if $\mu(\phi) \vdash_{\mb{OS4}} \mu(\psi)$, then $\phi \vdash_\mathsf{F} \psi$. The proof relies on the following ``McKinsey-Tarski-style'' result.

\begin{theorem} \label{embthm}
    For any fundamental lattice $(L, \neg_L)$, there is an \textbf{OS4}-lattice $(M,\neg_M,\Box_M)$ and a lattice embedding $e: L \to M$ such that:
    \begin{enumerate}
        \item for any $a \in L$, $\Box_M e(a) = e(a)$;
        \item $e(\neg_L a) = \Box_M \neg_M e(a)$.
    \end{enumerate}
\end{theorem}

The proof of this theorem proceeds as follows. Fix a fundamental lattice $(L,\neg)$ and consider its reflexive canonical frame $(X, \op)$. Recall that $L$ embeds into $\chi_{\op}(X)$ via the map $a \mapsto \widehat{a} = \{(F,I) \in X \mid a \in F\}$ (Theorem~\ref{RefCanEmbed}).

\begin{lemma} \label{canballma}
    For any fundamental lattice $(L, \neg)$, its reflexive canonical frame $(X, \op)$ is balanced.
\end{lemma}

\begin{proof}
We must show that $(X, \op)$ is both prefactoring and postfactoring. For clarity, we will write $x, y,...$ for the elements of $X$,  $x_F,y_F,...$ for their filter components and $x_I,y_I,...$ for their ideal components. Let us start by showing that $(X, \op)$ is postfactoring. Given $x,y \in X$ such that $y \op x$, let $y' = (\{1\},y_I)$. Clearly, $y' \in X$ and $y'\leq_{\po} y$. Since $x \op x$, we have that $1 \notin x_I$, so $x_I \cap y'_F = \varnothing$ and hence $y' \opo x$.

Next, let us show that prefactoring also holds. Fix some $x = (x_F, x_I) \in X$. We let $x' = (x_F,I(x_F))$, where $I(x_F)$ is the ideal $\{b \in L \mid \exists a_1,...,a_n \in x_F: b \leq \neg a_1 \jo ... \jo \neg a_n\}$. Note first that $x' \in X$. Indeed, we clearly have that $a \in x_F$ implies $\neg a \in I(x_F)$. Moreover, if $b \in I(x_F)$, then there are $a_1,...,a_n \in x_F$ such that $b \leq \neg a_1 \jo ... \jo \neg a_n$. We claim that $\neg \neg a_1 \land ... \land \neg \neg a_n \leq \neg b$. Clearly, this follows from $\neg \neg a_1 \land ... \land \neg \neg a_n \leq \neg (\neg a_1 \lor ... \lor \neg a_n)$. To show this, by dual self-adjointness, it is enough to show that $\neg a_i \leq \neg (\neg\neg a_1 \land...\land \neg\neg a_n)$ for any $i \leq n$. But we have $\neg \neg a_1 \land ... \land \neg \neg a_n \leq \neg \neg a_i$, so this follows from the antitonicity of $\neg$ together with the fact that $\neg \neg \neg c = \neg c$ for any $c \in L$, which is a direct consequence of dual self-adjointness. Hence $\neg \neg a_1 \me ... \me \neg \neg a_n \leq \neg b$, and $\neg b \in x_F$. From this it follows that $x_F \cap I(x_F) = \varnothing$. Otherwise, we would have that $a \me \neg a \in x_F$ for some $a \in L$, so $0 \in x_F \cap x_I$, contradicting the fact that $\op$ is reflexive. Now we claim that $y \opo x'$ for any $y \op x$. Suppose that $y \op x$. Clearly, since $x_F = x'_F$, we also have that $y \op x'$, so we only need to check that $I(x_F) \cap y_F = \varnothing$. Suppose $a \in I(x_F)$. Then, by the argument above, it follows that $\neg a \in x_F$. Moreover, $a \in y_F$ implies $\neg a \in y_I$. But this means that $y_I \cap x_F = \varnothing$ implies $I(x_F) \cap y_F = \varnothing$, so $x' \op y$. Since $x'_F=x_F$, we also have $x'\leq_{\op}x$, so $x'$ is the required witness for prefactoring.
\end{proof}

We can now complete the proof of \cref{embthm}. Fix a fundamental lattice $(L, \neg_L)$ with reflexive canonical frame $(X, \op)$. By \cref{canballma}, this fundamental frame is balanced, so by \cref{orthmodcomp}, it has an orthomodal companion $(X, \opo, \leq_{\op})$. By \cref{keylma}, the map $a \mapsto \widehat{a}$ is a lattice embedding of $L$ into the \textbf{OS4}-lattice $\chi_{\opo}(X)$. Moreover, for any $a \in L$, we have that $\Box_{\leq_{\op}} \widehat{a} = \widehat{a}$, and $\Box_{\leq_{\op}} \neg_{\opo} \widehat{a} = \neg_{\op} \widehat{a} = \widehat{\neg_L a}$.

\begin{theorem} The map $\mu$ is a full translation of fundamental logic into $\mathsf{OS4}$: if ${\mu(\phi) \vdash_\mathsf{OS4} \mu(\psi)}$, then $\phi \vdash_\mathsf{F} \psi$.
\end{theorem}

\begin{proof} Suppose that there are formulas $\phi$ and $\psi$ such that $\phi \nvdash_\mathsf{F} \psi$. Then there is a fundamental lattice $(L, \neg_L)$ and an $L$-valuation $\theta$ such that $\tilde{\theta}(\phi) \nleq \tilde{\theta}(\psi)$. By \cref{embthm}, we have an \textbf{OS4}-lattice $(M, \neg_M, \Box_M)$ and an embedding $e : L \to M$ with the properties that $\Box_M e(a) = e(a)$ and $\Box_M \neg_M e(a) = e(\neg_L a)$ for any $a \in L$. Now define an $M$-valuation $\theta'$ by letting $\theta'(p) = e(\theta(p))$ for any propositional variable $p$. We can now argue by induction that $\tilde{\theta'}(\mu(\chi)) = e(\tilde{\theta}(\chi))$ for any fundamental formula $\chi$. For the base case, we have:
\[\tilde{\theta'}(\mu(p)) = \tilde{\theta'}(\Box p) = \Box_M \theta'(p) = \Box_M e(\theta(p)) = e(\theta(p)).\] Assuming the claim holds for $\chi_1$ and $\chi_2$, we have:
\begin{align*}
     \tilde{\theta'}(\mu(\chi_1 \star \chi_2)) = \tilde{\theta'}(\mu(\chi_1) \star \mu(\chi_2)) &= \tilde{\theta'}(\mu(\chi_1)) \star_M \tilde{\theta'}(\mu(\chi_2)) \\ &= e(\tilde{\theta}(\chi_1)) \star_M e(\tilde{\theta}(\chi_2)) \\ &= e(\tilde{\theta}(\chi_1) \star_L \tilde{\theta}(\chi_2)) \\ &= e(\tilde{\theta}(\chi_1 \star \chi_2))
\end{align*} for $\star \in \{\me, \jo\}$. Finally, assuming that the claim holds for $\chi$, we have:
\[\tilde{\theta'}(\mu(\neg \chi)) = \tilde{\theta'}(\Box \neg \mu(\chi)) = \Box_M \neg_M \tilde{\theta'}(\mu(\chi)) = \Box_M \neg_M e(\tilde{\theta}(\chi)) = e(\neg_L \tilde{\theta}(\chi)) = e(\tilde{\theta}(\neg \chi)).\]
\noindent In particular, we have that $\tilde{\theta'}(\mu(\phi)) = e(\tilde{\theta}(\phi))$ and $\tilde{\theta'}(\mu(\psi)) = e(\tilde{\theta}(\psi))$. Since $e$ is an embedding, it follows that $\tilde{\theta'}(\mu(\phi)) \nleq \tilde{\theta'}(\mu(\psi))$ and hence that $\mu(\phi) \nvdash_{\mathsf{OS4}} \mu(\psi)$.\end{proof}

\subsection{First-Order Extension}\label{FirstOrderExt1}

We conclude this section by indicating how to extend the translation $\mu$ to the first-order case. Here as well, the standard method from the G\"odel-McKinsey-Tarski translation works smoothly.

The language of first-order fundamental logic is the usual first-order language with an infinite stock of individual variables, but without $\to$. First-order fundamental logic has all of the introduction and elimination rules of propositional fundamental logic plus the introduction and elimination rules for the quantifiers shown in Figure~\ref{QuantRules} (see Appendix~\ref{AppendixQuant} for a rigorous inductive definition). The first-order extension $\mathsf{QOS4}$ of $\mathsf{OS4}$ requires a bit more care to define, since the introduction and elimination rules can also be applied within $\Box$-subproofs, in which case the $\forall$I rule needs to stipulate that the quantified variable cannot occur free in any formula $\Box\varphi$ outside the $\Box$-subproof for which $\varphi$ could be brought into the subproof by $\Box$E; see Appendix~\ref{AppendixModal} for a rigorous definition.

 The translation $\mu$ from fundamental logic to $\mathsf{OS4}$ is extended to a translation from the language of first-order fundamental logic to $\mathsf{QOS4}$ via the following clauses:

\begin{itemize}
    \item $\mu(\forall x \phi(x)) = \Box \forall x \mu(\phi(x))$;
    \item $\mu(\exists x \phi(x)) = \exists x \mu(\phi(x))$.
\end{itemize}

\begin{figure}[H]
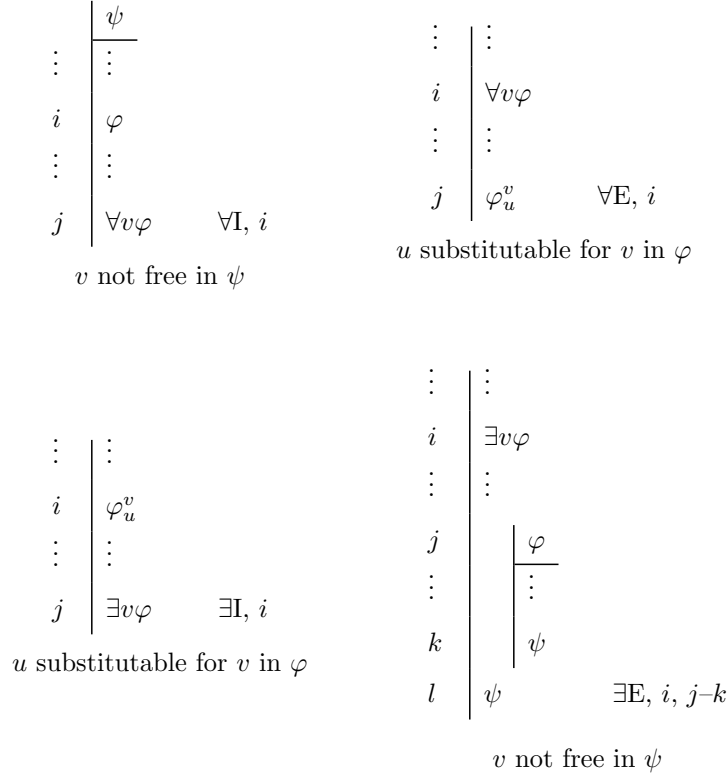

\begin{center}
\begin{minipage}{2in}
\[\begin{nd}
\hypo [\,] {} {\psi}
\have [\vdots] {0} {\vdots}
\have [i] {3}   {\varphi}
\have [\vdots] {4}   {\vdots}
\have [j]{5} {\forall v\varphi} \Ai{3}
\end{nd}
\]\[v\mbox{ not free in }\psi\]\end{minipage}\begin{minipage}{2in}
\[\begin{nd}
\have [\vdots] {0} {\vdots}
\have [i] {3}   {\forall v\varphi}
\have [\vdots] {4}   {\vdots}
\have [j]{5} {\varphi^v_u} \Ae{3}
\end{nd}
\]\[u\mbox{ substitutable for }v\mbox{ in }\varphi\]\end{minipage}\vspace{.2in}

\begin{minipage}{2.35in}
\[\begin{nd}
\have [\vdots] {0} {\vdots}
\have [i] {3}   {\varphi^v_u}
\have [\vdots] {4}   {\vdots}
\have [j]{5} {\exists v\varphi} \Ei{3}
\end{nd}
\]\[u\mbox{ substitutable for }v\mbox{ in }\varphi\]\end{minipage}\begin{minipage}{2in}
\[\begin{nd}
\have [\vdots] {0} {\vdots}
\have [i] {1} {\exists v\varphi}
\have [\vdots] {2} {\vdots}
\open
\hypo [j] {3}   {\varphi}
\have [\vdots] {4}   {\vdots}
\have [k] {6}   {\psi}
\close
\have [l]{7} {\psi} \Ee{1,3-6}
\end{nd}
\]\[v\mbox{ not free in }\psi\]\end{minipage}
\end{center}
\caption{Fitch-style rules for the logic with quantifiers.}\label{QuantRules}
\end{figure}

\begin{lemma}
    The translation $\mu$ from first-order fundamental logic to $\mathsf{QOS4}$ is faithful: for any two formulas $\phi,\psi$ of the language of first-order fundamental logic,   $\phi \vdash_{\mathsf{QF}} \psi$ implies $\mu(\phi) \vdash_{\mathsf{QOS4}} \mu(\psi)$.
\end{lemma}

\begin{proof}
    Notice first that for any formula $\phi$ of first-order fundamental logic, $\Box\mu(\phi)$ is equivalent to $\mu(\phi)$ over $\mathsf{QOS4}$. This is proved by an easy induction on the complexity of $\phi$. We write out only the existential case, where we must show that $\exists x \mu(\phi(x)) \vdash_{\mathsf{QOS4}} \Box \exists x\mu (\phi(x))$. Let $u$ be substitutable for $x$ in $\mu(\phi(x))$. Clearly, $\mu(\phi(u)) \vdash_{\mathsf{QOS4}} \exists x \mu(\phi(x))$, so $\Box \mu(\phi(u)) \vdash_{\mathsf{QOS4}} \Box \exists x \mu(\phi(x))$. As $\mu(\phi(u)) \vdash_{\mathsf{QOS4}} \Box \mu (\phi(u))$ by the induction hypothesis, we get that $\mu(\phi(u)) \vdash_{\mathsf{QOS4}} \Box \exists x \mu(\phi(x))$, and therefore $\exists x \mu(\phi(x)) \vdash_{\mathsf{QOS4}} \Box \exists x \mu(\phi(x))$, as desired.

    We can now show that $\phi \vdash_\mathsf{QF} \psi$ implies $\mu(\phi) \vdash_{\mathsf{QOS4}} \mu (\psi)$ for any formulas $\phi$ and $\psi$ of the language of first-order fundamental logic. We do this by induction on the length of $\mathsf{QF}$-proofs. The only non-immediate case is when the last rule applied is universal introduction. Suppose that $\phi \vdash_{\mathsf{QF}} \psi(x)$ for $x$ not free in $\phi$. By the induction hypothesis, we have that $\mu(\phi) \vdash_{\mathsf{QOS4}} \mu(\psi(x))$, so $\mu(\phi) \vdash_{\mathsf{QOS4}} \forall x\mu(\psi(x))$. This implies that $\Box \mu(\phi) \vdash_{\mathsf{QOS4}} \Box \forall x \mu(\psi(x))$. As established above, we have that $\mu(\phi) \vdash_{\mathsf{QOS4}} \Box \mu(\phi)$ and hence $\mu(\phi) \vdash_{\mathsf{QOS4}} \Box \forall x \mu(\psi(x))$. This shows that $\mu(\phi) \vdash_{\mathsf{QOS4}} \mu(\psi)$, as desired. 
\end{proof}

In order to prove that the translation is also full, we need the following strengthening of \cref{embthm}.

\begin{theorem} \label{qembthm}
    For any fundamental lattice $(L, \neg_L)$, there is a complete \textbf{OS4}-lattice $(M,\neg_M,\Box_M)$ and a lattice embedding $e: L \to M$ such that:
    \begin{enumerate}
        \item for any $a \in L$, $\Box_M e(a) = e(a)$;
        \item $e(\neg_L a) = \Box_M \neg_M e(a)$;
        \item for any $A \sset L$, if $\bigme_L A$ exists in $L$, then $e(\bigme_L A) = \Box_M\bigme_M e[A]$;
        \item for any $A \sset L$, if $\bigjo_L A$ exists in $L$, then $e(\bigjo_L A) = \bigjo_M e[A]$.
    \end{enumerate}
\end{theorem}

\begin{proof}
    Let $(X, \op)$ be the fundamental frame described in \cref{prinfrm} with $\mathrm{V} = \Lambda = L$. We claim that this frame is balanced. Suppose that $y \op x$, with $x = (x_a,x_b)$ and $y = (y_a, y_b)$. Note first that if $y_a \nleq x_b$, then $x \op y$, so $x$ and $y$ are the required witnesses for pre- and postfactoring,  respectively. So assume $y_a \leq x_b$, which means that $y_a \neq 1$ and therefore that $y = (y_a, \neg y_a)$. We let $z = (x_a, \neg x_a)$. It is easy to see that $z \leq_{\op} x$ and that $y \op z$. Moreover, if $y_a \leq \neg x_a$, then, by dual self-adjointness, $x_a \leq \neg y_a$, contradicting the fact that $y \op x$. Hence $z \op y$, and $z$ is the required witness for prefactoring. For postfactoring, let $z = (1, \neg y_a)$. It is easy to see that $z \leq_{\po} y$. Moreover, we have that $x \op z$, since $x_b \neq 1$. Hence $z$ is the required witness, which completes the proof of the claim.

    Now we let $(M,\neg_M,\Box_M)$ be the dual algebra of the orthomodal companion of $(X, \op)$, which is a complete \textbf{OS4}-lattice, and we let $e$ be the composition of the map $a \mapsto a^\downarrow$ with the inclusion map $ \iota: \chi_{\op}(X) \to M$. Since the map $a \mapsto a^\downarrow$ is an embedding, we know that $e$ is an embedding. Moreover, by the same reasoning as in the proof of \cref{embthm}, we know that conditions $1$ and $2$ hold. For condition $3$, fix $A \sset L$ such that $\bigme_L A$ exists in $L$. Note first that the map $a \mapsto a^\downarrow$ preserves arbitrary meets, so $(\bigme_L A)^\downarrow = \bigcap \{a^\downarrow \mid a \in A\}$. Moreover, by \cref{keylma}, $\Box_M: M \to \chi_{\op}(X)$ is right-adjoint to $\iota$, so it preserves limits. This means that \[\Box_M (\bigcap \{a^\downarrow \mid a \in A\}) = \bigcap \{\Box_M a^\downarrow \mid a \in A\} = \bigcap \{a^\downarrow \mid a \in A\}.\] This shows that $e(\bigme_L A) = \bigcap \{a^\downarrow \mid a \in A\} = \Box_M\bigme_M e[A]$, as desired.
    Finally, for condition $4$, fix $A \sset L$ such that $\bigjo_L A$ exists in $L$. Since $\iota: \chi_{\op}(X) \to M$ has a right-adjoint, it preserves arbitrary joins. Because the map $a \mapsto a^\downarrow$ also preserves arbitrary joins, it follows that $e: L \to M$ preserves arbitrary joins in $L$, so $e(\bigjo_L A) = \bigjo_M e[A]$. This completes the proof.
\end{proof}

\begin{corollary} \label{qfullgmt}
    The translation $\mu$ is also full: for any two formulas $\phi,\psi$ of the language of first-order fundamental logic, $\mu(\phi) \vdash_{\mathsf{QOS4}} \mu(\psi)$ implies $\phi \vdash_{\mathsf{QF}} \psi$.
\end{corollary}

\begin{proof}
    Suppose that $\phi \nvdash_\mathsf{QF} \psi$. The Lindenbaum algebra $L$ of first-order fundamental logic is a fundamental lattice in which the (equivalence class of the) formula $\forall x \phi(x)$ and the (equivalence class of the) formula $\exists x \phi(x)$ are, respectively, the meet and join of the set of (equivalence classes of) formulas in $\{\phi^x_y \mid y\mbox{ substitutable for }x\mbox{ in }\phi\}$. Let $\theta$ be the canonical valuation on $L$ mapping each atomic formula to its equivalence class. By assumption,  $\tilde{\theta}(\phi) \nleq \tilde{\theta}(\psi)$. By \cref{qembthm}, there is a complete \textbf{OS4}-lattice $(M, \neg_M, \Box_M)$ and an embedding $e: L \to M$ with the properties stated in the theorem. Letting $\theta'$ be a valuation on $M$ for the language of $\mathsf{QOS4}$ given by $\theta'(\phi) = e(\theta(\phi))$ for any atomic formula $\phi$, a straightforward induction establishes that $\tilde{\theta'}(\mu(\phi)) = e(\tilde{\theta}(\phi))$ for any first-order fundamental formula $\phi$. Since $e$ is an embedding, it follows that $\tilde{\theta'}(\mu(\phi)) \nleq \tilde{\theta'}(\mu(\psi))$, so $\mu(\phi) \nvdash_{\mathsf{QOS4}} \mu(\psi)$, as desired.
\end{proof}

\begin{remark}
    In the proof of \cref{qembthm}, the $\Box_M$ modal operator in the \textbf{OS4}-lattice $M$ constructed is completely multiplicative, meaning that
    \[\Box_M \bigme A = \bigme \{ \Box_M a \mid a \in A\}\] for any $A \sset M$. On such modal ortholattices, the Barcan axiom
    \begin{itemize}
        \item[($\forall$B)] $\forall x \Box \phi(x) \vdash \Box \forall x \phi(x)$
    \end{itemize}
    is valid. This means that the $\mu$ translation is also a full and faithful translation of first-order fundamental logic into the logic $\mathsf{QOS4}$ + ($\forall $B).
\end{remark}

\section{The Goldblatt Translation into Intuitionistic \textsf{KTB}}\label{KTBSection}

In the previous section, we saw that the GMT translation of intuitionistic logic into $\mathsf{S4}$ can be generalized beyond the distributive setting to a translation of fundamental logic into $\mathsf{OS4}$. We now swap the roles of intuitionistic logic and orthologic, and generalize the Goldblatt translation of orthologic into $\mathsf{KTB}$ (\citealt{Goldblatt1974}) to a translation of fundamental logic into an intuitionistic version of $\mathsf{KTB}$. The motivation is similar to the one in the previous section. Intuitively, the Goldblatt translation can be understood as a way for the classical logician to understand the orthologician: in order to determine whether a formula $\psi$ follows from $\phi$ according to the latter, it is enough to check whether the Goldblatt translation of $\psi$ follows from the Goldblatt translation of $\phi$ over the modal logic $\mathsf{KTB}$. In this section, we generalize this state of affairs to the intuitionistic background. In what follows, we assume some basic familiarity with relational semantics for intuitionistic modal logic, particularly for Fischer Servi's intuitionistic modal logic $\mathsf{FS}$ (\citealt{FischerServi1984}).

\subsection{\textbf{FSTB}-Frames}
We start with the following definition.
\begin{definition}\label{fstblat}
    An \textit{\textbf{FSTB}-lattice} is a triple $(H, \Box, \diamondsuit)$ such that $H$ is a Heyting algebra and $\Box$ and $\diamondsuit$ are unary operations on $H$ satisfying the following axioms:
    \begin{enumerate}
        \item $\Box 1 = 1$, $\diamondsuit 0 = 0$;
        \item $\Box (a \me b) = \Box a \me \Box b$, $\diamondsuit (a \jo b) = \diamondsuit a \jo \diamondsuit b$;
        \item $\diamondsuit a \to \Box b \leq \Box (a \to b)$;
        \item $\diamondsuit (a \to b) \leq \Box a \to \diamondsuit b$;
        \item $\Box a \leq a \leq \diamondsuit a$;
        \item $\diamondsuit \Box a \leq a \leq \Box \diamondsuit a$.
        \end{enumerate}
\end{definition}

Axioms $1$ and $2$ in the definition above are standard principles of algebraic modal logic. Axioms $3$ and $4$ characterize Fischer Servi's logic $\mathsf{FS}$, while axioms $5$ and $6$ are intuitionistic versions of axioms \textsf{T} ($\Box a \leq a$) and \textsf{B}  ($a \leq \Box \diamondsuit a)$,  respectively. Recall that whenever $(B, \Box)$ is a Boolean algebra with a normal $\Box$ satisfying axiom $T$ and axiom $B$, the map $a \mapsto \Box \diamondsuit a$ is a closure operator on $B$. This is the key algebraic property used by the Goldblatt translation, which maps formulas of the language of orthologic to fixpoints of the $\Box\diamondsuit$-operator on the Lindenbaum algebra of the modal logic $\mathsf{KTB}$. The following lemma establishes that a similar situation arises in the case of \textbf{FSTB}-lattices.

\begin{lemma} For any \textbf{FSTB}-lattice $(H, \Box, \diamondsuit)$, the operation $a \mapsto \Box\diamondsuit a$ is a closure operator.
\end{lemma}

\begin{proof}
   Recall that a closure operator on a poset $P$ is a monotone, increasing, and idempotent map from $P$ to $P$. Since $\Box \diamondsuit$ is the composition of two monotone operators, it is also monotone. Moreover, it is clearly increasing by axiom $6$ in \cref{fstblat}. Hence we only have to show idempotence, i.e., that $\Box \diamondsuit \Box \diamondsuit a \leq \Box \diamondsuit a$ for any $a \in H$. Indeed, by axiom $6$, we have $\diamondsuit \Box \diamondsuit a \leq \diamondsuit a$, and then monotonicity of $\Box$ immediately yields $\Box \diamondsuit \Box \diamondsuit a \leq \Box \diamondsuit a$, as desired.\end{proof}

Let us now define the Fischer Servi frames that correspond to \textbf{FSTB}-lattices. Following Dragalin \citeyearpar{Dragalin1988}, we think in terms of downsets rather than upsets, so $x\leq y$ means that $x$ is an intuitionistic successor of $y$. 

\begin{definition} \label{fstbdef}
    An \textit{\textbf{FSTB}-frame} is a triple $(X, \leq, \top)$ such that:
    \begin{enumerate}
        \item $X$ is a nonempty set, and $\leq$ is a reflexive and transitive relation on $X$;
        \item $\top$ is a reflexive and symmetric relation on $X$;
        \item\label{fstbdef3} for any $x,y,z \in X$, $y \leq x \top z$ implies that there is a $w \in X$ such that $y \top w \leq z$;
        \item\label{fstbdef4} for any $x,y,z \in X$, $x \top y \geq z$ implies that there is a $w \in X$ such that $x \geq w \top z$.
    \end{enumerate}
\end{definition}
\noindent Note that the list of conditions in Definition~\ref{fstbdef} is somewhat redundant. Conditions $1$, $3$, and $4$ are the standard conditions on Fischer Servi frames. However, when the accessibility relation $\top$ is symmetric, conditions $3$ and $4$ are easily seen to be equivalent to one another.

Given an \textbf{FSTB}-frame $(X, \leq, \top)$, the \textit{complex algebra} of $(X, \leq, \top)$ is the Heyting algebra $Dn(X)$ of downsets of $(X,\leq)$, together with the two operations $\Box_\top$ and $\diamondsuit_\top$ defined by $\Box_\top A = \{x \in X \mid {\forall y, z \, (x \geq y \top z \Rightarrow z \in A)}\}$ and $\diamondsuit_\top A = \{x \in X \mid \exists y \in A : x \top y\}$.

\begin{lemma}
    For any \textbf{FSTB}-frame $(X, \leq, \top)$, its complex algebra is an \textbf{FSTB}-lattice. 
\end{lemma}

\begin{proof}
    It is already known in the literature (\citealt{amati1994uniform}, \citealt{Ch23}) that the complex algebra of an \textbf{FSTB}-frame $(X, \leq, \top)$ satisfies axioms $1$-$5$ in Definition~\ref{fstblat}, so we only need to verify that $\diamondsuit_\top \Box_\top A \sset A \sset \Box_\top \diamondsuit_\top A$ for any $A \in Dn(X)$. For the first inclusion, suppose that $x \in \diamondsuit_\top \Box_\top A$. Then there is a $y \in \Box_\top A$ such that $x \top y$. Since $\top$ is symmetric, we have that $y \top x$, so $x \in A$. For the second inclusion, suppose $x \in A$ and $x \geq y \top z$. Since $A$ is a downset, $y \in A$. Since $\top$ is symmetric, $z \top y$. Hence $z \in \diamondsuit_\top A$. This shows that $x \in \Box_\top \diamondsuit_\top A$. 
\end{proof}

The previous lemma establishes that any \textbf{FSTB}-frame determines an \textbf{FSTB}-lattice. Moreover, every \textbf{FSTB}-lattice can be represented as a sublattice of the complex algebra of some \textbf{FSTB}-frame.

\begin{lemma} \label{repfstblma}
    Every \textbf{FSTB}-lattice embeds into the complex algebra of some \textbf{FSTB}-frame.
\end{lemma}

\begin{proof}
    The proof builds on standard representation theory for Heyting algebras (see, e.g., \citealt{Esakia2019}). Fix an \textbf{FSTB}-lattice $(H, \Box, \diamondsuit)$ and let $(X, \leq, \top)$ be the birelational frame such that $X$ is the set of all prime filters on $H$, $\leq$ is the converse inclusion ordering, and $\top$ is a binary relation such that for any $p, q \in X$, $p \top q$ iff (i) $\Box a \in p$ implies $a \in q$ and (ii) $a \in q$ implies $\diamondsuit a \in p$. It is already known that $(X, \leq, \top)$ satisfies conditions $1$, $3$, and $4$ of the definition of an \textbf{FSTB}-frame and that the map $a \mapsto \{p \in X \mid a \in p\}$ is a Heyting embedding of $(H, \Box, \diamondsuit)$ into $(Dn(X), \Box_\top, \diamondsuit_\top)$ that also preserves the modal operators. Hence we only need to verify that $\top$ is reflexive and symmetric. Reflexivity is immediate given axiom $5$ from Definition~\ref{fstblat}. For symmetry, suppose that $p \top q$, and let $\Box a \in q$. Then $\diamondsuit \Box a \in p$, so $a \in p$ since $\diamondsuit \Box a \leq a$. Conversely, if $a \in p$, then $\Box \diamondsuit a \in p$ since $a \leq \Box \diamondsuit a$, so $\diamondsuit a \in q$. This shows that $q \top p$. Thus, $(X, \leq, \top)$ is an \textbf{FSTB}-frame.
\end{proof}

We conclude this introduction to \textbf{FSTB}-frames with an intuitive interpretation of the semantics for intuitionistic modal logic that they provide. As in a standard explanation of relational semantics for intuitionistic logic, given an \textbf{FSTB} frame $(X, \leq, \top)$, we may think about points in $X$ as states of information, ordered by informativeness by the relation $\leq$, where $x \leq y$ means that $x$ contains at least as much information as $y$ does. The relation $\top$, which is reflexive and symmetric, can be understood as a relation of \emph{indistinguishability} between states. The Fischer Servi conditions on the interplay between $\leq$ and $\top$ now receive a natural interpretation: if two states $x$ and $z$ are indistinguishable, then their possible futures are also indistinguishable: to any possible information increase $y$ accessible from $x$ corresponds an information increase $w$ that is accessible from $z$ and indistinguishable from $y$. 

Let us now move on to the relationship between fundamental frames and \textbf{FSTB}-frames.

\subsection{From \textbf{FSTB}-frames to Fundamental Frames}

In this section, we show how to obtain a fundamental frame from an \textbf{FSTB}-frame in a canonical way. We start with the following definition, illustrated by an example in Figure~\ref{FSTBReductFig}.

\begin{definition}
    Let $(X, \leq, \top)$ be an \textbf{FSTB}-frame. The \textit{fundamental reduct of} $(X, \leq, \top)$ is the relational frame $(X, \op)$, where $y \op x$ iff there is $z \in X$ such that $y \top z \leq x$.
\end{definition}

\noindent Thus, $y$ is open to $x$ in the fundamental reduct iff there is a possible knowledge state in the future of $x$ that is indistinguishable from $y$. Clearly, this relation need not be symmetric: just because a possible increase in information from state $x$ would make it indistinguishable from $y$ does not mean that an increase in information from $y$ would make it indistinguishable from $x$. The next lemma establishes a tight connection between the complex algebra of an \textbf{FSTB}-frame and the dual algebra of its fundamental reduct.

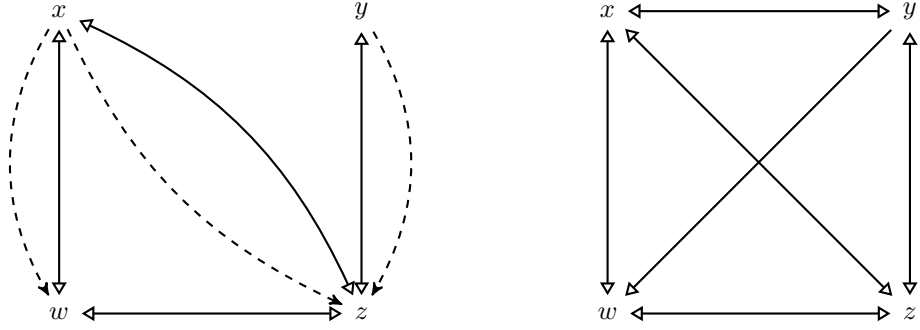
\begin{figure}[h]
  \begin{center}
  \begin{tikzpicture}[->,>={stealth'[length=6pt]},shorten >=1pt,shorten <=1pt, auto,node
  distance=2cm,thick,every loop/.style={<-,shorten <=1pt}]
  \tikzstyle{every state}=[fill=gray!20,draw=none,text=black]

  \node (0) at (-2,0) {{$x$}};
  \node (1) at (2,0) {{$y$}};
  \node (2) at (-2,-4) {{$w$}};
  \node (3) at (2,-4) {{$z$}};

  \path[{Triangle[open]}-{Triangle[open]},draw,thick] (0) to node {{}} (2);
  \path[{Triangle[open]}-{Triangle[open]},draw,thick, bend left=20] (0) to node {{}} (3);
  \path[{Triangle[open]}-{Triangle[open]},draw,thick] (1) to node {{}} (3);
  \path[{Triangle[open]}-{Triangle[open]},draw,thick] (2) to node {{}} (3);

  \path (0) edge[dashed,->, bend right=30] node {{}} (2);
  \path (0) edge[dashed,->, bend right=20] node {{}} (3);
  \path (1) edge[dashed,->, bend left=30] node {{}} (3);
  \end{tikzpicture}\qquad\qquad\qquad \begin{tikzpicture}[->,>={stealth'[length=6pt]},shorten >=1pt,shorten <=1pt,
  auto,node
  distance=2cm,thick,every loop/.style={<-,shorten <=1pt}]

    \tikzstyle{directed_edge}=[-{Triangle[open]},draw,thick]
    \tikzstyle{edge}=[{Triangle[open]}-{Triangle[open]},draw,thick]

  \node (0) at (-2,0) {{$x$}};
  \node (1) at (2,0) {{$y$}};
  \node (2) at (-2,-4) {{$w$}};
  \node (3) at (2,-4) {{$z$}};

    \draw[edge] (0) to (1);
    \draw[edge] (0) to (2);
    \draw[edge] (0) to (3);
    \draw[directed_edge] (1) to (2);
    \draw[edge] (1) to (3);
    \draw[edge] (2) to (3);
  \end{tikzpicture}
  \end{center}
  \caption{An \textbf{FSTB}-frame (left) and its fundamental reduct (right). On the left, a dashed (resp.~solid) line from $v$ to  $u$ indicates that $u\leq v$ (resp.~$u\top v$). On the right, a solid line from $v$ to $u$ indicates that $u\op v$.}\label{FSTBReductFig}
  \end{figure}

\begin{lemma} \label{reductlma}
    For any \textbf{FSTB}-frame $(X,\leq,\top)$, its fundamental reduct $(X, \op)$ is a fundamental frame. Moreover, the dual algebra of $(X, \op)$ coincides with the fixpoints of the $\Box_\top \diamondsuit_\top$ operation on $Dn(X)$, and $\neg_{\op} A = \Box_\top \neg A$ for any $A \in Dn(X)$.
\end{lemma}

\begin{proof}
    Fix an \textbf{FSTB}-frame $(X, \leq, \top)$ with fundamental reduct $(X, \op)$. We first claim that $(X, \op)$ is a fundamental frame. Note that, in the proof of \cref{funredlma}, we only use the fact that the relation $\top$ is reflexive and symmetric and that the relation $\leq$ is a preorder. Since the relation $\op$ is defined exactly as the openness relation in an $\mathbf{OS4}$-frame, the same proof therefore implies that $(X, \op)$ is a fundamental frame. To see that the dual algebra of $(X, \op)$ coincides with the fixpoints of the $\Box_\top \diamondsuit_\top$ operation, note first that $\leq$ is a subrelation of the prerefinement relation $\leq_{\op}$. Consequently, since any element in the dual algebra of $(X, \op)$ is downward closed under prerefinement, it is also an element in $Dn(X)$. 
    
    Now let us show that for any $A \in Dn(X)$, $\neg_{\op} \neg_{\po} A = \Box_\top \diamondsuit_\top A$. Suppose first that $x \in \Box_\top \diamondsuit_\top A$, and let $y \op x$. Then there is a $z \in X$ such that $y \top z \leq x$. Since $\Box_\top \diamondsuit_\top A$ is a downset, $z \in \Box_\top \diamondsuit_\top A$, and hence, since $\top$ is symmetric, $y \in \diamondsuit_\top A$. But this means that there is a $z' \in A$ such that $y \top z'$. Since $y \top z' \leq z'$, it follows that $y \op z'$. Hence $x \in \neg_{\op} \neg_{\po} A$. Conversely, suppose that $x \in \neg_{\op} \neg_{\po} A$, and let $y,z \in X$ be such that $x \geq y \top z$. Then $z \op x$, so there is a $z' \in A$ such that $z \op z'$. But this means there is a $y'$ such that $z \top y' \leq z'$. Since $A$ is a downset, we have that $y' \in A$ and thus $z \in \diamondsuit_\top A$. This shows that $x \in \Box_\top \diamondsuit_\top A$. Hence for any downset $A$, $A = \neg_{\op} \neg_{\po} A$ iff $A = \Box_\top \diamondsuit_\top A$, which shows that the dual algebra of $(X, \op)$ coincides with the fixpoints of $\Box_\top \diamondsuit_\top$ in $Dn(X)$. 
    
    Finally, we show that $\Box_\top \neg A = \neg_{\op} A$ for any downset $A$. Suppose first that $x \in \Box_\top \neg A$, and let $y \op x$. Then there is a $z \in X$ such that $y \top z \leq x$. But this immediately implies that $y \in \neg A$, since $x \in \Box_\top \neg A$, and hence $y \notin A$. This shows that $x \in \neg_{\op} A$. For the converse direction, suppose that $x \in \neg_{\op} A$, and let $y$ be such that $y \top z \leq x$ for some $z \in X$. We claim that $y \in \neg A$. Suppose $y' \leq y$. Since we have $z \top y \geq y'$, by condition $4$ of \textbf{FSTB}-frames, there is a $w \in X$ such that $z \geq w \top y'$. Since $x \geq z \geq w \top y'$, this means that $y' \op x$, so $y' \notin A$. Hence $y \in \neg A$, which completes the proof.
\end{proof}

This result shows that the complex algebra of any \textbf{FSTB}-frame contains, so to speak, the dual algebra of a fundamental frame inside it as the fixpoints of the $\Box \diamondsuit$ operator. This motivates a translation of fundamental logic into the logic of \textbf{FSTB}-frames, which we will define in \cref{goltran}. For now, let us identify a particular class of \textbf{FSTB}-frames. For this, we first need the following definitions.

\begin{definition}
    Let $(X, \op)$ be a fundamental frame. For any $x, y \in X$, $x$ \textit{strongly refines} y, denoted $x \preccurlyeq_{\op} y$, if $x \leq_{\op} y$ and $y \leq_{\po} x$.
\end{definition}

Whenever $(X,\op)$ is the fundamental reduct of an \textbf{FSTB}-frame $(X, \leq, \top)$, the strong refinement relation has a natural interpretation. Indeed, $x \leq_{\op} y$ means that any state indistinguishable from a possible future state of $x$ is also indistinguishable from a possible future state of $y$. Dually, $y \leq_{\po} x$ means that any state ``in the past'' of a state indistinguishable from $y$ is also ``in the past'' of a state indistinguishable from $x$. Loosely speaking, this means that $x$ strongly refines $y$ whenever $y$ is ``always in the past'' of $x$ modulo indistinguishability, and $x$ is ``always in the future'' of $y$ modulo indistinguishability. With that interpretation in mind, the following result is not surprising.

\begin{lemma}
    Let $(X, \leq, \top)$ be an \textbf{FSTB}-frame and $(X, \op)$ its fundamental reduct. Then $\top $ is a subrelation of $\opo$, and $\leq$ is a subrelation of $\preccurlyeq_{\op}$.
\end{lemma}

\begin{proof}
Fix $(X, \leq,\top)$ and its fundamental reduct $(X, \op)$. Since $\leq$ is reflexive and $\top$ is symmetric, it is clear that $\top$ is a subrelation of $\opo$. Moreover, since $\leq$ is transitive, we also have that $\leq$ is a subrelation of $\leq_{\op}$. To show that $\leq$ is a subrelation of $\preccurlyeq_{\op}$, we therefore only have to show that $x \leq y$ implies $y \leq_{\po} x$. So suppose $x \leq y$ and $y \op z$. By definition, this means that there is a $w \in X$ such that $y \top w \leq z$. Since $x \leq y \top w$, by the interaction condition between $\top$ and $\leq$, we have some $w' \in X$ such that $x \top w' \leq w \leq z$, from which it follows that $x \op z$. This shows that $\leq$ is a subrelation of $\preccurlyeq_{\op}$.
\end{proof}

Let us conclude by isolating those \textbf{FSTB}-frames $(X, \leq, \top)$ for which $\leq$ and $\top$ coincide with the relations $\preccurlyeq_{\op}$ and $\opo$, respectively, induced by their fundamental reduct $(X, \op)$. In the definitions below, given an \textbf{FSTB}-frame $(X, \leq, \top)$ and a set $U \sset X$, we let $\upset U := \{y \in X \mid \exists x \in U: x \leq y\}$ and $\dnset U :=\{ y \in X \mid \exists x \in U : y \leq x\}$. Whenever $U$ is a singleton set of the form $\{x\}$, we write $\upset x$ and $\dnset x$, respectively, for the set $\upset \{x \}$ and $\dnset \{x\}$. We also consider the modality $\Box_{\po}$ given by $\Box_{\po} U = \{x\in X\mid \mbox{for all }y\po x,\, y\in U\}$.

\begin{definition}
    An \textbf{FSTB}-frame $(X, \leq, \top)$ is \textit{pseudo-regular} if it satisfies the following condition for any $x, y \in X$:
    \[\big(x \in \Box_\top \diamondsuit_\top \dnset{y} \text{ and } y \in \Box_{\po}\upset{\diamondsuit_\top \{x\}}\big) \Rightarrow x \in \dnset{{y}}.\]
\end{definition}

For the following, we say that an \textbf{FSTB}-frame $(X, \leq, \top)$ is \textit{skeletal} if $(X, \top, \leq)$ is skeletal as defined in Definition~\ref{SkelReg}.

\begin{lemma} \label{skelpsreglma}
    Let $(X, \leq, \top)$ be an \textbf{FSTB}-frame and $(X, \op)$ its fundamental reduct. Then $(X, \leq, \top)$ is skeletal iff $\top = \, \opo$, and it is pseudo-regular iff $\preccurlyeq_{\op} \,=\, \leq$.
\end{lemma}

\begin{proof}
    That $(X, \leq, \top)$ is skeletal iff $\top \! = \, \opo$ is clear. To show that it is pseudo-regular iff $\preccurlyeq_{\op} \,= \, \leq$, note first that it is enough to show that pseudo-regularity is equivalent to $\preccurlyeq_{\op}$ being a subrelation of $\leq$. This, in turn, amounts to showing that for any $x, y \in X$, (i) $x \leq_{\op} y$ iff $x \in \Box_\top \diamondsuit_\top \dnset{y}$ and (ii) $y \leq_{\po} x$ iff $y \in \Box_{\po}\upset\diamondsuit_\top \{x\}$. For the first equivalence, we compute:
    \begin{align*}
        x \in \Box_\top \diamondsuit_\top\dnset{y} &\Leftrightarrow \forall z \big((\exists w :z \top w \leq x) \Rightarrow (\exists w' : z \top w' \leq y)\big)\\
        &\Leftrightarrow \forall z(z \op x \Rightarrow z \op y)\\
        &\Leftrightarrow x \leq_{\op} y.
    \end{align*}
\noindent For the second equivalence, we compute:
    \begin{align*}
        y \in \Box_{\po} \upset{\diamondsuit_\top \{x\}} &\Leftrightarrow \forall z \big( (\exists w :y \top w \leq z) \Rightarrow (\exists w': x \top w' \leq z)\big )\\
        &\Leftrightarrow \forall z(y \op z \Rightarrow x \op z)\\
        &\Leftrightarrow y \leq_{\po} x.
    \end{align*}    
    This completes the proof.
\end{proof}

Let us now see how to go from fundamental frames to \textbf{FSTB}-frames. 

\subsection{From Fundamental Frames to \textbf{FSTB}-frames}

The consideration of the strong refinement relation above motivates the following definition.

\begin{definition}
    A fundamental frame $(X, \op)$ is \textit{strongly factoring} if it is reflexive and for any $x, y \in X$, $y\op x$ implies that there is a $z \in X$ such that $y \opo z \preccurlyeq_{\op} x$.
\end{definition}

\noindent Intuitively, a strongly factoring frame is a frame where the relation $\op$ can always be decomposed into the symmetric ``compatibility'' relation induced by $\op$ and the strong refinement relation. In other words, $y \op x$ if and only if $y$ is fully compatible with a state $z$ that is ``generically'' in the future of $x$. The following establishes a connection between the strong factoring property and fundamental reducts of \textbf{FSTB}-frames.

\begin{lemma} \label{fstofunlma}
    Let $(X, \leq, \top)$ be an \textbf{FSTB}-frame. Then its fundamental reduct $(X, \op)$ is strongly factoring.
\end{lemma}

\begin{proof}
    For reflexivity, $x\op x$ follows from $x\top x\leq x$. For the rest of strong factoring, suppose $y \op x$. Then there is a $w \in X$ such that $y \top w \leq x$. We claim that $w \preccurlyeq_{\op} x$. To show this, we must establish that $w \leq_{\op} x$ and $x \leq_{\po} w$. Suppose first that $z \op w$. Then there is a $w'$ such that $z \top w' \leq w \leq x$, so $z \op x$. This shows that $w \leq_{\op} x$. Now suppose $x \op z$. Then there is a $w'$ such that $x \top w' \leq z$. Since we have $w \leq x \top w'$, there is a $z'$ such that $w \top z' \leq w' \leq z$. But this shows that $w \op z$, so $x \leq_{\po} w$. Therefore, $w \preccurlyeq_{\op} x$, which shows that $(X, \op)$ is strongly factoring.
\end{proof}

In fact, we can also go from strongly factoring fundamental frames to \textbf{FSTB}-frames, via the following construction.

\begin{lemma}
    Let $(X, \op)$ be a strongly factoring fundamental frame. Then $(X, \preccurlyeq_{\op}, \opo)$ is an \textbf{FSTB}-frame.
\end{lemma}

\begin{proof}
    Note first that since $\preccurlyeq_{\op}$ is the intersection of two preorders, it is also a preorder. Similarly, $\opo$ is the symmetric kernel of $\op$. Hence we only need to check conditions \ref{fstbdef3} and \ref{fstbdef4} in \cref{fstbdef}. Note that because $\opo$ is symmetric, it is in fact enough to check one of the two. So let us check that if $x \top y \succcurlyeq_{\op} z$, then there is a $w \in X$ such that $x \succcurlyeq_{\op} w \top z$. First, note that $x \top y \succcurlyeq_{\op} z$ implies that $y \op x$ and $y \leq_{\po} z$ and hence that $z \op x$. By strong factoring, this means that there is a $w$ such that $z \top w \preccurlyeq_{\op} x$, which completes the proof.
\end{proof}

The previous lemma motivates the following definition.

\begin{definition}
    For any strongly factoring fundamental frame $(X, \op)$, its \textit{\textbf{FSTB}-companion} is the frame $(X, \preccurlyeq_{\op}, \opo)$.
\end{definition}

    \begin{figure}[H]
  \begin{center}                                                                                          
  \begin{tikzpicture}[->,>={stealth'[length=6pt]},shorten >=1pt,shorten <=1pt, auto,node                                
  distance=2cm,thick,every loop/.style={<-,shorten <=1pt}]

    \tikzstyle{directed_edge}=[-{Triangle[open]},draw,thick]
    \tikzstyle{edge}=[{Triangle[open]}-{Triangle[open]},draw,thick]

  \node (0) at (-2,-4) {{$x$}};
  \node (1) at (2,-4) {{$y$}};
  \node (2) at (-2,0) {{$w$}};
  \node (3) at (2,0) {{$z$}};

    \draw[edge] (0) to (1);
    \draw[edge] (0) to (2);
    \draw[edge] (1) to (3);
    \draw[directed_edge] (3) to (0);
    \draw[directed_edge] (2) to (1);
    \draw[directed_edge] (3) to (2);
  \end{tikzpicture}\qquad\qquad\qquad \begin{tikzpicture}[->,>={stealth'[length=6pt]},shorten >=1pt,shorten <=1pt,
  auto,node
  distance=2cm,thick,every loop/.style={<-,shorten <=1pt}]
  \tikzstyle{every state}=[fill=gray!20,draw=none,text=black]

  \node (0) at (-2,-4) {{$x$}};
  \node (1) at (2,-4) {{$y$}};
  \node (2) at (-2,0) {{$w$}};
  \node (3) at (2,0) {{$z$}};

  \path[{Triangle[open]}-{Triangle[open]},draw,thick] (0) to node {{}} (1);
  \path[{Triangle[open]}-{Triangle[open]},draw,thick] (0) to node {{}} (2);
  \path[{Triangle[open]}-{Triangle[open]},draw,thick] (1) to node {{}} (3);

  \path (2) edge[dashed,->, bend right=30] node {{}} (0);
  \path (3) edge[dashed,->] node {{}} (0);
  \path (3) edge[dashed,->, bend left=30] node {{}} (1);
  \end{tikzpicture}
  \end{center}
  \caption{A strongly factoring fundamental frame (left) and its \textbf{FSTB}-companion (right). On the
   right, a dashed (resp.~solid) line from $v$ to $u$ indicates that $u\leq v$ (resp.~$u\top
  v$).}\label{FSTBCompFig}
  \end{figure}
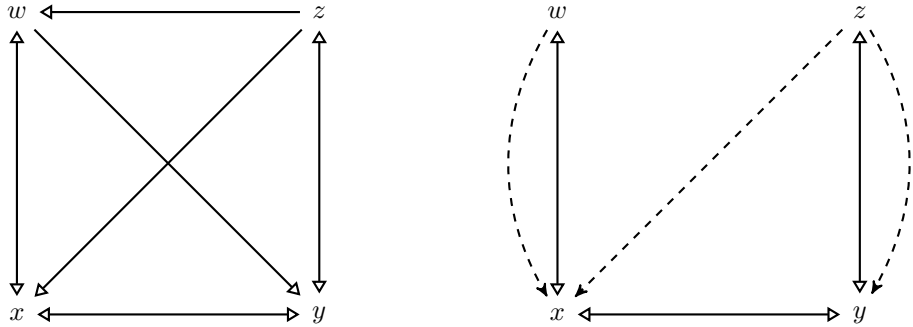

As expected, there is a tight relationship between the dual algebra of a fundamental frame and the complex algebra of its \textbf{FSTB}-companion.

\begin{lemma} \label{fstbkeylma}
    Given a strongly factoring fundamental frame $(X, \op)$, its dual algebra is isomorphic to the $\Box _{\opo} \diamondsuit _{\opo}$ fixpoints of the complex algebra of its \textbf{FSTB}-companion endowed with the negation $\Box _{\opo} \neg$.
\end{lemma}

\begin{proof}
    By \cref{reductlma}, it is enough to show that $(X, \op)$ is the fundamental reduct of its \textbf{FSTB}-companion. But this is immediate, since $y \op x$ iff there is a $z$ such that $y \opo z \preccurlyeq_{\op} x$. Indeed, the right-to-left direction follows from the fact that $y \opo z \preccurlyeq_{\op} x$ implies that $z \leq_{\op} x$ and $y \op z$ and hence that $y \op x$, while the converse is simply the definition of strong factoring.
\end{proof}

We are now in a position to prove the main result of this section, which will play a key role in establishing that the translation of fundamental logic into \textsf{FSTB} defined in the next section is full.

\begin{theorem} \label{embthm2}
    For any fundamental lattice $(L, \neg_L)$, there is an \textbf{FSTB}-lattice $(H, \Box, \diamondsuit)$ and a bounded meet-semilattice embedding $\nu : L \to H$ such that for any $a, b \in L$:
    \begin{itemize}
        \item $\Box \diamondsuit \nu (a) = \nu(a)$;
        \item $\nu(a \jo b) = \Box \diamondsuit (\nu(a) \jo \nu(b))$;
        \item $\nu (\neg_L a) = \Box \neg \nu (a)$.
    \end{itemize}
\end{theorem}

\begin{proof}
   Fix a fundamental lattice $(L, \neg_L)$. Since $(L, \neg_L)$ embeds into the dual algebra of its reflexive canonical frame, the statement of the theorem follows from \cref{fstbkeylma} if we can show that this frame is strongly factoring. So suppose $y \op x$. Note  that we already established in the proof of \cref{canballma} that $y \opo (x_F, I(x_F))$. Hence we need only show that $(x_F, I(x_F)) \preccurlyeq_{\op} x$, i.e., $(x_F, I(x_F))\leq_{\op} x$ and $x\leq_{\po} (x_F, I(x_F))$, which is in turn equivalent to $x_F\subseteq x_F$ and $I(x_F)\sset x_I$. To prove $I(x_F) \sset x_I$, suppose $a \in I(x_F)$. Then $a \leq \neg b_1 \jo ... \jo \neg b_n$ for some $b_1,...,b_n \in x_F$. But then $\neg b_1,...,\neg b_n \in x_I$, so $a \in x_I$. This completes the proof that the reflexive canonical frame of $(L, \neg_L)$ is strongly factoring.
\end{proof}

Finally, let us establish a one-to-one correspondence between strongly factoring fundamental frames and skeletal, pseudo-regular \textbf{FSTB}-frames. By \cref{skelpsreglma,fstofunlma}, we already know that the fundamental reduct $(X, \op)$ of a skeletal and pseudo-regular \textbf{FSTB}-frame $(X, \leq, \top)$ is strongly factoring, and moreover, that the \textbf{FSTB}-companion of $(X, \op)$ is $(X, \leq, \top)$. Hence we only need to show the following.

\begin{lemma}
    For any strongly factoring fundamental frame $(X, \op)$, its \textbf{FSTB}-companion $(X, \preccurlyeq_{\op}, \opo)$ is skeletal and pseudo-regular.
\end{lemma}

\begin{proof}
    By \cref{fstbkeylma}, $(X, \op)$ is the fundamental reduct of its \textbf{FSTB}-companion. Clearly, this also implies that $(X, \preccurlyeq_{\op}, \opo)$ is skeletal and pseudo-regular.
\end{proof}

\begin{corollary}
    There is a one-to-one correspondence between strongly factoring fundamental frames and skeletal, pseudo-regular \textbf{FSTB}-frames, obtained by mapping any strongly factoring fundamental frame to its \textbf{FSTB}-companion and mapping any skeletal, pseudo-regular \textbf{FSTB}-frame to its fundamental reduct.
\end{corollary}

\subsection{Translating Fundamental Logic into ``Intuitionistic \textsf{KTB}''} \label{goltran}

The results obtained in this section so far motivate defining a translation of fundamental logic into the following intuitionistic version of $\mathsf{KTB}$.

\begin{definition}
The logic $\mathsf{FSTB}$ is the intuitionistic modal logic obtained by adding to Fischer Servi's~\citeyearpar{FischerServi1984} logic $\mathsf{FS}$ the following axioms:
\begin{enumerate}
    \item $\Box p \to p$, $p \to \diamondsuit p$;
    \item $p \to \Box \diamondsuit p$, $\diamondsuit \Box p \to p$.
\end{enumerate}
\end{definition}

\begin{lemma}\label{FSTBcomp}
    The logic $\mathsf{FSTB}$ is sound and complete with respect to \textbf{FSTB}-frames.
\end{lemma}

The proof of Lemma~\ref{FSTBcomp} is obtained by standard algebraic techniques, using the representation in \cref{repfstblma}. Now we recursively define the translation $\gamma$ from the language of fundamental logic to the language of intuitionistic modal logic as follows.

\begin{itemize}
    \item $\gamma(\top)=\top$ and $\gamma(p) =  \Box \diamondsuit p$;
    \item $\gamma(\phi \me \psi) = \gamma(\phi) \me \gamma(\psi)$;
    \item $\gamma(\phi \jo \psi) = \Box\diamondsuit(\gamma(\phi) \jo \gamma(\psi))$;
    \item $\gamma(\neg \phi) = \Box \neg \gamma(\phi)$.
\end{itemize}

\begin{theorem}
    The translation $\gamma$ is faithful: for any $\phi,\psi\in\mathcal{L}(\wedge,\vee,\neg)$, if $\phi \vdash_\mathsf{F} \psi$, then $\gamma(\phi) \vdash_\mathsf{FSTB} \gamma(\psi)$.
\end{theorem}

\begin{proof}
    Suppose $\phi \vdash_\mathsf{F} \psi$. Fix an \textbf{FSTB}-frame $(X, \leq, \top)$ and a valuation $\theta$. Consider the fundamental reduct $(X, \op)$ of this frame, and the valuation $\theta'$ given by $\theta'(p) = \tilde{\theta}(\Box \diamondsuit p)$. Note that this is a well-defined fundamental valuation by \cref{reductlma}. Moreover, a simple induction on the complexity of formulas establishes that $\tilde{\theta'}(\chi) = \tilde{\theta}(\gamma(\chi))$ for any fundamental formula $\chi$. Since $\phi \vdash_\mathsf{F} \psi$, it follows that $\tilde{\theta}(\gamma(\phi)) = \tilde{\theta'}(\phi) \sset \tilde{\theta'}(\psi) = \tilde{\theta}(\gamma(\psi))$. By the completeness of $\mathsf{FSTB}$ with respect to \textbf{FSTB}-frames, this shows that $\gamma(\phi) \vdash_{\mathsf{FSTB}} \gamma(\psi)$, which completes the proof.
\end{proof}

Let us now show that the translation is also full.

\begin{theorem} \label{fstbtrthm}
    The translation $\gamma$ is full: for any $\phi,\psi\in\mathcal{L}(\wedge,\vee,\neg)$, if $\gamma(\phi) \vdash_\mathsf{FSTB} \gamma(\psi)$, then $\phi \vdash_\mathsf{F} \psi$.
\end{theorem}

\begin{proof}
    We argue by contraposition. Suppose $\phi \nvdash_\mathsf{F} \psi$. Then there is a fundamental lattice $(L, \neg_L)$ and a valuation $\theta$ on $(L, \neg_L)$ such that $\tilde{\theta}(\phi) \nleq \tilde{\theta}(\psi)$. By \cref{embthm2}, there is an \textbf{FSTB}-lattice $(H, \Box, \diamondsuit)$ and an embedding $\nu: L \to H$ with the properties mentioned in the theorem. Now define a valuation $\theta'$ on $(H, \Box, \diamondsuit)$ by letting $\theta'(p) = \nu(\tilde{\theta}(p))$. By induction on the complexity of formulas together with the properties of $\nu$, we have that $\tilde{\theta'}(\gamma(\chi)) = \nu(\tilde{\theta}(\chi))$ for any fundamental formula $\chi$. Since $\nu$ is an order-embedding, it follows that $\tilde{\theta'}(\gamma(\phi)) \nleq \tilde{\theta'}(\gamma(\psi))$ and hence $\gamma(\phi) \nvdash_{\mathsf{FSTB}} \gamma(\psi)$.
\end{proof}

\subsection{First-Order Extension}

We conclude this section by extending our results for $\mathsf{FSTB}$ to the first-order case. We show that the translation $\gamma$ can be extended to a translation of first-order fundamental logic into the quantified version of $\mathsf{FSTB}$, called $\mathsf{QFSTB}$, which extends $\mathsf{FSTB}$ with the introduction and elimination rules for the quantifiers as in first-order intuitionistic logic (see Appendix~\ref{AppendixQuant}). The translation $\gamma$ is extended to a translation from the language of first-order fundamental logic to $\mathsf{QFSTB}$ via the following clauses:

\begin{itemize}
    \item $\gamma(\forall x \phi(x)) = \forall x \gamma(\phi(x))$;
    \item $\gamma(\exists x \phi(x)) = \Box \diamondsuit \exists x \gamma(\phi(x))$.
\end{itemize}

\begin{lemma}
    The translation $\gamma$ from first-order fundamental logic to $\mathsf{QFSTB}$ is faithful: for any formulas $\varphi,\psi$ of the language of first-order fundamental logic, $\varphi \vdash_\mathsf{QF} \psi$ implies $\gamma (\varphi) \vdash_\mathsf{QFSTB} \gamma(\psi)$.
\end{lemma}

\begin{proof}
    Note first that a simple induction on the complexity of formulas establishes that $\gamma(\phi)$ is $\mathsf{QFSTB}$-equivalent to $\Box\diamondsuit \gamma(\phi)$ for any first-order fundamental formula $\phi$. The only nontrivial case is the universal case. Since $\Box\diamondsuit$ is a closure operator on $\mathsf{QFSTB}$-formulas, we immediately have that $\forall x \gamma(\phi(x)) \vdash_{\mathsf{QFSTB}} \Box\diamondsuit \forall x \gamma(\phi(x))$. For the converse, note that by the induction hypothesis, $\Box \diamondsuit \gamma(\phi(x)) \vdash_{\mathsf{QFSTB}} \gamma(\phi(x))$. By monotonicity of $\Box\diamondsuit$, we have that $\Box\diamondsuit \forall x \gamma(\phi(x)) \vdash_{\mathsf{QFSTB}} \Box \diamondsuit \gamma (\phi(x))$. Hence $\Box \diamondsuit \forall x \gamma(\phi(x)) \vdash_{\mathsf{QFSTB}} \gamma (\phi(x))$. Since $x$ is not free in $\Box \diamondsuit \forall x \gamma(\phi(x))$ this shows that $\Box\diamondsuit \gamma(\forall x \phi(x)) \vdash_{\mathsf{QFSTB}} \gamma(\forall x \phi(x))$, as desired.

    We can now show that $\phi \vdash_{\mathsf{QF}} \psi$ implies $\gamma(\phi) \vdash_{\mathsf{QFSTB}} \gamma(\psi)$, by induction on the length of $\mathsf{QF}$-proofs. We only write out the case when the last rule applied is existential elimination. Suppose that for some $\psi$ in which $x$ does not appear freely, we have that $\phi(x) \vdash_{\mathsf{QF}} \psi$. By the induction hypothesis, this means that $\gamma(\phi(x)) \vdash_{\mathsf{QFSTB}} \gamma(\psi)$ and hence $\exists x \gamma(\phi(x)) \vdash_{\mathsf{QFSTB}} \gamma(\psi)$. By the result established above together with the monotonicity of $\Box\diamondsuit$, we have that
    \[\Box \diamondsuit \exists x \gamma(\phi(x)) \vdash_{\mathsf{QFSTB}} \Box\diamondsuit \gamma(\psi) \vdash_{\mathsf{QFSTB}} \gamma(\psi).\] This shows that $\gamma(\exists x \phi(x)) \vdash_{\mathsf{QFSTB}} \gamma(\psi)$, as desired.
\end{proof}

To prove that the translation is also full, we first establish the following strengthening of \cref{embthm2}.

\begin{theorem} \label{qembthm2}
    For any fundamental lattice $(L, \neg_L)$, there is a complete \textbf{FSTB}-lattice $(H, \Box, \diamondsuit)$ and a bounded meet-semilattice embedding $\nu : L \to H$ such that for any $a, b \in L$:
    \begin{itemize}
        \item $\Box \diamondsuit \nu (a) = \nu(a)$;
        \item $\nu(a \jo b) = \Box \diamondsuit (\nu(a) \jo \nu(b))$;
        \item $\nu (\neg_L a) = \Box \neg \nu (a)$;
        \item for any $A \sset L$, if $\bigme_L A$ exists in $L$, then $\nu(\bigme_L A) = \bigme_H \nu[A]$;
        \item for any $A \sset L$, if $\bigjo_L A$ exists in $L$, then $\nu(\bigjo_L A) = \Box\diamondsuit(\bigjo_H \nu[A])$.
    \end{itemize}
\end{theorem}

\begin{proof}
    Let $(X, \op)$ be the fundamental frame described in \cref{prinfrm} with $\mathrm{V} = \Lambda = L$. We claim this frame is strongly factoring. Note first that it is reflexive. Moreover, suppose $y \op x$, with $x = (x_a,x_b)$ and $y = (y_a,y_b)$. If $x \op y$, then $x$ is the required witness $z$ for strong factoring. Otherwise, we have  $y_a \leq x_b$, so $y_a \neq 1$ and $y = (y_a, \neg y_a)$. Moreover, if $x = (x_a,\neg x_a)$, then $y_a \leq \neg x_a$ implies, by dual self-adjointness, that $x_a \leq \neg y_a$, contradicting the assumption that $y \op x$. Hence $x = (1,x_b)$. We must find $z = (z_a,z_b)$ such that $y \opo z$, $z_a \leq x_a$, and $z_b \leq x_b$. Clearly, $z = (1,0)$ is such a witness. This completes the proof of the claim. 

    Let $(H, \Box, \diamondsuit)$ be the complex algebra of the \textbf{FSTB}-companion of $(X,\op)$. By \cref{fstbkeylma}, we know that the dual algebra of $(X,\op)$ embeds into $(H,\Box,\diamondsuit)$ as the fixpoints of the closure operator $\Box\diamondsuit$ via the inclusion map $\iota$. Since $\Box\diamondsuit$ is a closure operator, its set of fixpoints is closed under arbitrary meets, and for any set $A \sset H$ of fixpoints of $\Box\diamondsuit$, their join in the algebra of $\Box\diamondsuit$ fixpoints is $\Box\diamondsuit(\bigjo_H A)$. Moreover, the map $a \mapsto a^\downarrow$ is a fundamental embedding that preserves arbitrary meets and joins. Putting these facts together, we get that the composition of the map $a \mapsto a^\downarrow$ with the inclusion $\iota: \chi_{\op}(X) \to H$ is the required embedding~$\nu$.
\end{proof}
 
\begin{corollary}
    The translation $\gamma$ is also full: for any two formulas $\phi,\psi$ of the language of first-order fundamental logic, $\gamma (\phi) \vdash_\mathsf{QFSTB} \gamma(\psi)$ implies $\phi \vdash_\mathsf{QF} \psi$.
\end{corollary}

\begin{proof}
    The proof is the same as in \cref{qfullgmt}, using \cref{qembthm2} instead of \cref{qembthm}.
\end{proof}

\begin{remark}
    It is worth mentioning that in the proof of \cref{qembthm2}, the \textbf{FSTB}-lattice $H$ constructed is actually a complete bi-Heyting algebra, i.e., a complete Heyting algebra satisfying the Meet Infinite Distributive Law:
    \[\bigme A \lor b = \bigme\{a \lor b \mid a \in A\}\] for any $A \sset H$ and $b\in H$.
On such complete lattices, the Constant Domain axiom is valid:
\begin{itemize}
    \item[(CD)] $\forall x(\phi(x) \lor \psi) \vdash \forall x \phi(x) \lor \psi$, if $x$ is not free in $\psi$.
\end{itemize}
Moreover, the modal operators $\Box$ and $\diamondsuit$ in that \textbf{FSTB}-lattice are also completely multiplicative and completely additive, respectively, meaning that the two Barcan formulas are also valid:

\begin{itemize}
    \item[($\forall$B)] $\forall x \Box \phi(x) \vdash \Box \forall x \phi(x)$;
    \item[($\exists$B)] $\diamondsuit \exists x \phi(x) \vdash \exists x \diamondsuit \phi(x)$.
\end{itemize}
This means that first-order fundamental logic can also be fully and faithfully translated into the logic $\mathsf{QFSTB}$ + (CD) + ($\forall$B) + ($\exists$B).
\end{remark}

\section{The GMT Translation into \textsf{FN4}}\label{FN4Section}
In the previous sections, we have explored translations of fundamental logic into modal logics over a stronger base logic, either orthologic (in \cref{OS4Section}) or intuitionistic logic (in \cref{KTBSection}). In this section, we explore the converse problem of translating stronger logics into modal extensions of fundamental logic. \citet[Prop.~2.3]{Holliday2023} already established that orthologic can be embedded into fundamental logic via the G\"odel-Gentzen ``double negation'' translation. Here, we tackle the more involved task of embedding intuitionistic logic into fundamental modal logic via a version of the GMT translation, before studying the interaction between this translation and the G\"odel-Gentzen translation.

\subsection{The Logic \textsf{FN4}}
In this section, we show that intuitionistic logic can be embedded into a modal extension of fundamental logic with a necessity operator. This extension, which we call $\mathsf{FN4}$, strengthens the fundamental version of $\mathsf{S4}$, $\mathsf{FS4}$, with one additional rule. The key idea behind $\mathsf{FN4}$ can be found in John Stuart Mill's \textit{A System of Logic} (\citealt{Mill1843}): 
\begin{quote}If there be any meaning which confessedly belongs to the term necessity, it is \textit{unconditionalness}. That which is necessary, that which \textit{must} be, means that which will be, whatever supposition we may make in regard to all other things. (p.~222)\end{quote}
Translated into our setting, we take this to mean that when we regard $\Box$ as true \textit{necessity}, we should accept the rule of $\Box$-Reiteration shown in Figure~\ref{BoxReitFig}, since we should hold on to $\Box\psi$ ``whatever supposition we may make in regard to all other things,'' such as $\varphi$. For a rigorous inductive definition of $\mathsf{FN4}$ \textit{proofs given $R$}, where $R$ is a set of boxed formulas that may be reiterated, see Appendix~\ref{AppendixFN4}.

\begin{figure}[H]
\begin{center}
\begin{minipage}{2.5in}
\[\begin{nd}
\have [\vdots] {0} {\vdots}
\have [i] {3}   {\Box\psi}
\have [\vdots] {4}   {\vdots}
\open
\hypo[ j] {1} {\varphi}
\have [\vdots] {}  {\vdots}
\have [k] {5}   {\Box\psi} \rone{3}
\end{nd}
\]\end{minipage}
\end{center}
\caption{The $\Box$-Reiteration rule of $\mathsf{FN4}$.}\label{BoxReitFig}
\end{figure}

The corresponding class of algebras is the following.

\begin{definition}\label{FN4AlgDef} An $\mathbf{FN4}$-algebra is an algebra $(L,\neg,\Box)$ where $(L,\neg)$ is a fundamental lattice and for all $a,b,c\in L$:
\begin{enumerate}
\item\label{FN4AlgDef1} $\Box 1=1$, $\Box(a\wedge b)=\Box a\wedge\Box b$, $\Box a\leq a$, and $\Box a\leq \Box\Box a$;
\item\label{FN4AlgDef2} $(a\vee b)\wedge \Box c\leq (a\wedge \Box c)\vee (b\wedge\Box c)$;
\item\label{FN4AlgDef3} if $a\wedge \Box c\leq \neg b$, then $b\wedge \Box c\leq\neg a$.
\end{enumerate}
\end{definition}
\noindent The second item corresponds to the use of $\Box$-Reiteration into subproofs for $\vee$E, while the third item corresponds to the use of $\Box$-Reiteration into subproofs for $\neg$I, as shown in Figure~\ref{DistPseudoFN4}.

\begin{figure}[H]
\begin{center}
\begin{minipage}{3in}
\[\begin{nd}
\hypo [1] {1} {(p\vee q)\wedge\Box r}
\have [2] {2}  {p\vee q} \ae{1}
\have [3] {3} {\Box r} \ae{1}
\open
\hypo [4] {4} {p}
\have [5] {5} {\Box r} \rone{3}
\have [6] {6} {p\wedge\Box r} \ai{4,5}
\have [7] {7} {(p\wedge\Box r)\vee (q \wedge\Box r)} \oi{6}
\close 
\open
\hypo [8] {8} {q}
\have [9] {9} {\Box r} \rone{3}
\have [10] {10} {q\wedge\Box r} \ai{8,9}
\have [11] {11} {(p\wedge\Box r)\vee (q \wedge\Box r)} \oi{10}
\close
\have [12] {12} {(p\wedge\Box r)\vee (q \wedge\Box r)} \oe{2, 4-7, 8-11}
\end{nd}
\]\end{minipage}\begin{minipage}{3in}
\[\begin{nd}
\hypo [1] {1} {q\wedge \Box r}
\have [2] {2} {q} \ae{1}
\have [3] {3} {\Box r} \ae{1}
\open
\hypo [4] {4} {p}
\have [5] {5} {\Box r} \rone{3}
\have [6] {6} {p\wedge\Box r} \ai{4,5}
\have [7] {7} {\vdots}
\have [8] {8} {\neg q} 
\close 
\have [9] {9} {\neg p} \ni{4-8, 2}
\end{nd}
\]\end{minipage}
\end{center}
\caption{$\mathsf{FN4}$ proofs of principles corresponding to Definition~\ref{FN4AlgDef}.\ref{FN4AlgDef2}-\ref{FN4AlgDef3}.}\label{DistPseudoFN4}
\end{figure}

We define a consequence relation $\vDash_{\mathbb{FN}4}$ over the class $\mathbb{FN}4$ of $\mathbf{FN4}$-algebras in the usual way. It is a straightforward exercise to establish the following soundness and completeness theorem.

\begin{proposition} For any $\varphi,\psi\in\mathcal{L}(\wedge,\vee,\neg,\Box)$, $\varphi\vdash_\mathsf{FN4}\psi$ iff $\varphi\vDash_{\mathbb{FN}4}\psi$.
\end{proposition}

We can use this algebraic semantics to show $\mathsf{FN4}$ is conservative over fundamental logic.

\begin{proposition}\label{FN4Conserve} For all $\varphi,\psi\in\mathcal{L}(\wedge,\vee,\neg)$, if $\varphi\vdash_{\mathsf{FN4}}\psi$, then $\varphi\vdash_{\mathsf{F}}\psi$.
\end{proposition}
\begin{proof} If $\varphi\nvdash_{\mathsf{F}}\psi$, then by the completeness of $\mathsf{F}$ with respect to the class of fundamental lattices, there is a fundamental lattice $(L,\neg)$ and valuation $\theta$ for $(L,\neg)$ such that $\tilde{\theta}(\varphi)\not\leq \tilde{\theta}(\psi)$. Expand $(L,\neg)$ with the $\Box$ operation defined by 
\[\Box a= \begin{cases} 1&\mbox{if }a=1 \\ 0 &\mbox{ otherwise}.\end{cases}\]
Then it is easy to check that $(L,\neg,\Box)$ is an $\mathbf{FN4}$ algebra. Hence by the soundness of $\mathsf{FN4}$ with respect to the class of $\mathbf{FN4}$ algebras, we have $\varphi\nvdash_{\mathsf{FN4}}\psi$. \end{proof}

\noindent We can also use the algebraic semantics to prove that the GMT translation from $\mathsf{IPC}^{-}$ to $\mathsf{FN4}$ is full.

\begin{proposition}\label{FN4toInt} For any $\varphi,\psi\in\mathcal{L}(\wedge,\vee,\neg)$, if $\mu(\varphi)\vdash_{\mathsf{FN4}}\mu(\psi)$, then $\varphi\vdash_{\mathsf{IPC}^{-}}\psi$.
\end{proposition}
\begin{proof} Suppose $\varphi\nvdash_{\mathsf{IPC}^{-}}\psi$. Then there is a pseudocomplemented distributive lattice $(L,\neg)$ and valuation $\theta$ on $(L,\neg)$ such that $\tilde{\theta}(\varphi)\not\leq_{L} \tilde{\theta}(\psi)$. Equip $L$ with a unary operation $\Box$ that is the identity operation. Then it is easy to see that $(L,\neg,\Box)$ is an $\mathbf{FN4}$-algebra, and a trivial induction yields
$\tilde{\theta}(\chi)=\tilde{\theta}(\mu(\chi))$. Thus, $\tilde{\theta}(\mu(\varphi))\not\leq_{L} \tilde{\theta}(\mu(\psi))$, so $\mu(\varphi)\nvdash_{\mathsf{FN4}}\mu(\psi)$ by the soundness of $\mathsf{FN4}$ with respect to $\mathbf{FN4}$-algebras.\end{proof}

\begin{remark} In the setting of classical logic, the logic of necessity is standardly taken to be $\mathsf{S5}$, not $\mathsf{S4}$. However, in the setting of fundamental logic, there are several systems one could mean by `$\mathsf{S5}$'. In fundamental logics of necessity \textit{and possibility} (\citealt{Holliday2024}), as in intuitionistic logics of necessity and possibility (\citealt{FischerServi1984}), possibility is not assumed to be the dual of necessity; one does not assume that $\diamondsuit p$ is equivalent to $\neg\Box\neg p$. One way of defining `$\mathsf{S5}$', algebraically, in such a fundamental setting is to introduce an additive $\diamondsuit$ operation satisfying $\diamondsuit \Box a\leq \Box a$ or $\diamondsuit a \leq \Box\diamondsuit a$ (which are not equivalent over fundamental modal algebras) or both. If one were to define $\mathsf{FN5}$ as the logic for such algebras, then Proposition~\ref{FN4toInt} would still hold, since in the proof, the algebra $(L,\neg,\Box,\diamondsuit)$ where both $\Box$ and $\diamondsuit$ are the identity operation would be an $\mathbf{FN5}$-algebra. However, there are other ways defining `$\mathsf{S5}$' that would break the proof of Proposition \ref{FN4toInt}, e.g., if one were to require that $\neg\Box\neg \Box a\leq \Box a$, since this becomes the intuitionistically unacceptable double negation elimination rule when $\Box$ is the identity, so $(L,\neg,\Box)$ may not be an algebra for the logic. We will return to related issues when we discuss classical logic in Section~\ref{ClassicalSect}.\end{remark}

The key to proving the converse of Proposition~\ref{FN4toInt} is the following.

\begin{proposition}\label{BoxedElementsProp1} For any $\mathbf{FN4}$-algebra $(L,\neg,\Box)$, the set $\{\Box a\mid a\in L\}$ of boxed elements becomes a distributive lattice under the meet and join operations of $L$ and pseudocomplemented under the operation $\neg_\Box$ defined by $\neg_\Box(a):=\Box\neg a$.
\end{proposition}

\begin{proof} That the boxed elements form a lattice under the meet and join operations of $L$ is a standard exercise using Definition~\ref{FN4AlgDef}.\ref{FN4AlgDef1}. That they form a \textit{distributive} lattice is immediate from Definition~\ref{FN4AlgDef}.\ref{FN4AlgDef2} together with the fact that $\Box a=\Box\Box a$ from Definition~\ref{FN4AlgDef}.\ref{FN4AlgDef1}. For pseudocomplementation under $\neg_\Box$, suppose $\Box a\wedge\Box c=0=\neg 1$. Then by Definition~\ref{FN4AlgDef}.\ref{FN4AlgDef3}, we have $1\wedge \Box c\leq \neg\Box a$ and hence $\Box c\leq \Box (1\wedge\Box c) \leq\Box\neg\Box a=\neg_\Box \Box a$. Note that we also have $\Box a\wedge \neg_\Box \Box a=\Box a\wedge \Box\neg\Box a= \Box\Box a \wedge\Box\neg\Box a=\Box(\Box a\wedge\neg\Box a)=\Box 0=0$.\end{proof}

\begin{corollary}\label{IPC-toFN4} For any $\varphi,\psi\in\mathcal{L}(\wedge,\vee,\neg)$, if $\varphi\vdash_{\mathsf{IPC}^{-}}\psi$, then $\mu(\varphi)\vdash_{\mathsf{FN4}}\mu(\psi)$.
\end{corollary}

\begin{proof} Suppose $\mu(\varphi)\nvdash_{\mathsf{FN4}}\mu(\psi)$, so there is an $\mathbf{FN4}$-algebra $(L,\neg,\Box)$ and valuation $\theta$ on $L$ such that $\tilde{\theta}(\mu(\varphi))\not\leq_L \tilde{\theta}(\mu(\psi))$. Let $D$ be the pseudocomplemented distributive lattice given by Proposition~\ref{BoxedElementsProp1}, and let $\theta'$ be the valuation on $D$ defined by $\theta'(p)= \tilde{\theta}(\mu(p))$. Then we can prove by induction that for any formula $\chi$, $\tilde{\theta'}(\chi)= \tilde{\theta}(\mu(\chi))$. Then $\tilde{\theta}(\mu(\varphi))\not\leq_L \tilde{\theta}(\mu(\psi))$ implies $\tilde{\theta'}(\varphi)\not\leq_D \tilde{\theta'}(\psi)$ and hence $\varphi\nvdash_{\mathsf{IPC}^{-}}\psi$.\end{proof}

In Appendix~\ref{ProofTransAppendix}, we give a more constructive, proof-theoretic proof of Corollary \ref{IPC-toFN4}, showing how proofs in $\mathbf{IPC}^-$ can be systematically translated into proofs in $\mathbf{FN4}$. 

Combining Proposition~\ref{FN4toInt} and Corollary~\ref{IPC-toFN4}, we have the desired full and faithful embedding.

\begin{theorem}\label{IntFN4FullFaithfulv1} For any $\varphi,\psi\in\mathcal{L}(\wedge,\vee,\neg)$, $\varphi\vdash_{\mathsf{IPC}^{-}}\psi$ iff $\mu(\varphi)\vdash_{\mathsf{FN4}}\mu(\psi)$.
\end{theorem}

\begin{remark}\label{WeakerBoxReit} Inspection of the proofs leading up to Theorem \ref{IntFN4FullFaithfulv1} shows that it also holds if we weaken $\Box$-Reiteration in such a way that $\Box\psi$ can only be $\Box$-reiterated into a subproof \textit{whose assumption is of the form $\Box\varphi$}. This corresponds to weaker forms of distributivity and pseudocomplementation: $(\Box a\vee \Box b)\wedge \Box c\leq (\Box a\wedge\Box c)\vee (\Box b\wedge\Box c)$, and if $\Box a\wedge\Box c\leq \neg b$, then $b\wedge\Box c\leq \neg\Box a$. These weaker forms are sufficient for the proof in Proposition~\ref{BoxedElementsProp1} that boxed elements form a pseudocomplemented distributive lattice.\end{remark}

\subsection{Conditional Extension}\label{CondExt}

In this section, we extend Theorem~\ref{IntFN4FullFaithfulv1} to full intuitionistic propositional logic including $\to$ and $\mathsf{FN4}$ extended with a conditional.\footnote{As discussed in \citealt[\S~6]{Holliday2023}, there are various ways to add a conditional to fundamental logic (also see \citealt{Holliday2024b}); here we add a weak conditional that is strong enough, given $\Box$-Reiteration, for the embedding of $\mathsf{IPC}$.} The introduction and elimination rules for $\to$ are shown in Figure~\ref{ToRules}. Let $\mathsf{IPC}$ and $\mathsf{FC4}$ be the extensions of $\mathsf{IPC}^-$ and $\mathsf{FN4}$, respectively, with these rules. As before, $\mathsf{IPC}$ allows Reiteration into subproofs for $\to$I, whereas $\mathsf{FC4}$ only allows $\Box$-Reiteration into subproofs for $\to$I. 

Let $\mathsf{FC4}^{\cong}$ be the extension of $\mathsf{FC4}$ with the congruence rules in Figure~\ref{CongRules}. These are not necessary for our translation results, but they allow $\mathsf{FC4}^{\cong}$ to prove that $\to$ is congruential (i.e., if $\varphi\dashv\vdash_{\mathsf{FC4}^{\cong}}\varphi'$ and $\psi\dashv\vdash_{\mathsf{FC4}^{\cong}}\psi'$, then $\varphi\to\psi\dashv\vdash_{\mathsf{FC4}^{\cong}}\varphi'\to\psi'$), which allows us to use algebraic-semantic arguments for $\mathsf{FC4}^{\cong}$; more involved proof-theoretic arguments for $\mathsf{FC4}$ are provided in Appendix~\ref{ProofTransAppendix}. Note that fundamental logic without $\to$ is already strong enough to prove that $\wedge$, $\vee$, and $\neg$ are congruential.

\begin{figure}[H]
\begin{center}
\begin{minipage}{2in}
\[\begin{nd}
\have [\vdots] {0} {\vdots}
\open
\hypo [i] {3}   {\hspace{-.24in}\to\;\,\varphi}
\have [\vdots] {4}   {\vdots}
\have [j] {6}   {\psi}
\close
\have [k]{7} {\varphi\to\psi} \ii{3-6}
\end{nd}
\]\end{minipage}
\begin{minipage}{2in}\[\begin{nd}
\have [\vdots] {} {\vdots}
\have [i] {0} {\varphi\to\psi}
\have [\vdots] {1} {\vdots}
\have [j] {3}   {\varphi}
\have [\vdots] {4}   {\vdots}
\have [k] {6}   {\psi} \ie{0,3}
\end{nd}
\]\end{minipage}\end{center}
\caption{Introduction and elimination rules for $\to$.}\label{ToRules}
\end{figure}

\begin{figure}[h]
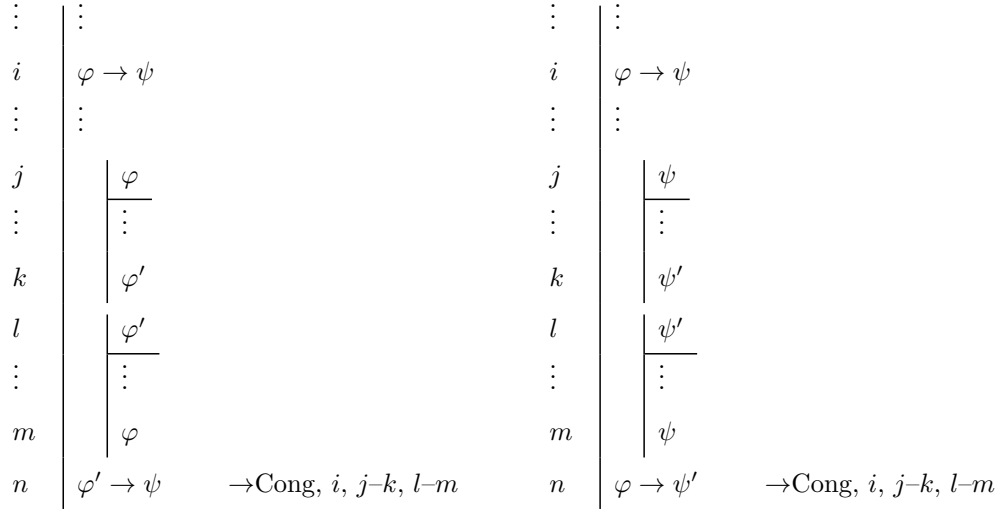

\begin{center}
\begin{minipage}{2.75in}\[\begin{nd}
\have [\vdots] {} {\vdots}
\have [i] {0} {\varphi\to\psi}
\have [\vdots] {1} {\vdots}
\open
\hypo [j] {2}   {\varphi}
\have [\vdots] {3}   {\vdots}
\have [k] {4}   {\varphi'} 
\close
\open
\hypo [l] {5}   {\varphi'}
\have [\vdots] {6}   {\vdots}
\have [m] {7}   {\varphi}
\close
\have [n] {8} {\varphi'\to\psi} \cong{0,2-4,5-7}
\end{nd}
\]\end{minipage}
\begin{minipage}{2.75in}\[\begin{nd}
\have [\vdots] {} {\vdots}
\have [i] {0} {\varphi\to\psi}
\have [\vdots] {1} {\vdots}
\open
\hypo [j] {2}   {\psi}
\have [\vdots] {3}   {\vdots}
\have [k] {4}   {\psi'} 
\close
\open
\hypo [l] {5}   {\psi'}
\have [\vdots] {6}   {\vdots}
\have [m] {7}   {\psi}
\close
\have [n] {8} {\varphi\to\psi'}  \cong{0,2-4,5-7}
\end{nd}
\]\end{minipage}\end{center}
\caption{Congruence rules for $\to$ in $\mathsf{FC4}^{\cong}$. The subproofs must be proofs given $\varnothing$ (see Appendix~\ref{AppendixFN4}), so $\Box$-Reiteration cannot be used to bring formulas outside of these subproofs into the subproofs.}\label{CongRules}
\end{figure}

The class of algebras corresponding to $\mathsf{FC4}^{\cong}$ is the following.

\begin{definition} An \textit{$\mathbf{FC4}$-algebra} is an algebra $(L,\neg,\Box,\to)$ where $(L,\neg,\Box)$ is an $\mathbf{FN4}$-algebra and for all $a,b,c\in L$:
\begin{enumerate}
\item if $\Box c\wedge a\leq b$, then $\Box c\leq a\to b$; 
\item $a\wedge (a\to b)\leq b$.
\end{enumerate}
\end{definition}
\noindent The first principle corresponds to $\to$I, allowing $\Box$-Reiteration into $\to$I-subproofs, while the second corresponds to $\to$E. The soundness of $\mathsf{FC4}$ and the soundness and completeness of $\mathsf{FC4}^{\cong}$ with respect to the class $\mathbb{FC}4$ of $\mathbf{FC4}$-algebras are straightforward. 

\begin{proposition} For any $\varphi,\psi\in\mathcal{L}(\wedge,\vee,\neg,\Box,\to)$, if $\varphi\vdash_\mathsf{FC4}\psi$, then  $\varphi\vDash_{\mathbb{FC}4}\psi$; and $\varphi\vdash_{\mathsf{FC4}^{\cong}}\psi$ iff $\varphi\vDash_{\mathbb{FC}4}\psi$.
\end{proposition}

Again we can use the algebraic semantics to prove a conservativity result.

\begin{proposition}\label{FC4Conserve} For all $\varphi,\psi\in\mathcal{L}(\wedge,\vee,\neg)$, if $\varphi\vdash_{\mathsf{FC4}^{\cong}}\psi$, then $\varphi\vdash_{\mathsf{F}}\psi$.
\end{proposition}
\begin{proof} The proof is the same as that of Proposition \ref{FN4Conserve} except that we further expand $(L,\neg,\Box)$ to the algebra $(L,\neg, \Box,\to)$ where 
\[a\to b= \begin{cases} 1&\mbox{if }a\leq b \\ 0 &\mbox{ otherwise}.\end{cases}\]
Then it is easy to check that $(L,\neg,\Box, \to)$ is an $\mathbf{FC4}$ algebra. \end{proof}

The GMT translation extends to $\mathcal{L}(\wedge,\vee,\neg,\to)$ with the clause \[\mu(\varphi\to\psi)=\Box (\mu(\varphi)\to\mu(\psi)).\] 

\begin{proposition}\label{IntFN4FullFaithfulv2} For all $\varphi,\psi\in\mathcal{L}(\wedge,\vee,\neg,\to)$, if $\mu(\varphi)\vdash_{\mathsf{FC4}^{\cong}}\mu(\psi)$, then $\varphi\vdash_{\mathsf{IPC}}\psi$.
\end{proposition}

\begin{proof} The proof is essentially the same as the proof of Proposition~\ref{FN4toInt} except that given $\varphi\nvdash_{\mathsf{IPC}}\psi$, we get a Heyting algebra $H$ and valuation $\theta$ on $H$ such that $\tilde{\theta}(\varphi)\not\leq_H \tilde{\theta}(\psi)$. The proof then continues as in the proof of Proposition~\ref{FN4toInt}.\end{proof}

The proof of the following is very similar to that of Proposition~\ref{BoxedElementsProp1}.

\begin{proposition}\label{BoxedElementsProp2} For any $\mathbf{FC4}$-algebra $(L,\neg,\Box,\to)$, the set $\{\Box a\mid a\in L\}$ becomes a Heyting algebra under the meet and join operations of $L$, the negation operation $\neg_\Box$ defined by $\neg_\Box(a):=\Box\neg a$, and the implication operation $\to_\Box$ defined by $a\to_\Box b:= \Box (a\to b)$.
\end{proposition}

\begin{corollary}\label{IPCtoFC4Cong} For all $\varphi,\psi\in\mathcal{L}(\wedge,\vee,\neg,\to)$, if $\varphi\vdash_{\mathsf{IPC}}\psi$, then  $\mu(\varphi)\vdash_{\mathsf{FC4}^{\cong}}\mu(\psi)$.
\end{corollary}
In Appendix~\ref{ProofTransAppendix}, we prove the analogue of Corollary~\ref{IPCtoFC4Cong} for $\mathsf{FC4}$.

\subsection{First-Order Extension}

Let us now extend the foregoing results to first-order logic. For intuitionistic first-order logic, $\mathsf{IQC}$, we add to $\mathsf{IPC}$ introduction and elimination rules for the quantifiers; but first we must make the free variable constraints in $\forall$I and $\exists$E in Figure~\ref{QuantRules} sensitive to the set $R$ of reiterable formulas, as explained in Appendix~\ref{AppendixIQC}. Let $\mathsf{QFC4}$ be the first-order extension of $\mathsf{FC4}$ with the same introduction and elimination rules as for $\mathsf{IQC}$, also allowing these rules to apply within $\Box$-subproofs, where for $\forall$I we must add a restriction on free occurrences of the quantified variable as we did for $\mathsf{QOS4}$; see Appendix~\ref{AppendixFN4} for a rigorous definition.

We now obtain a full and faithful embedding of $\mathsf{IQC}$ into $\mathsf{QFC4}$. 

\begin{theorem}\label{QFC4Thm} For any $\varphi,\psi\in\mathcal{L}(\wedge,\vee,\neg,\to,\forall,\exists)$, $\varphi \vdash_{\mathsf{IQC}}\psi$ iff $\mu(\varphi)\vdash_{\mathsf{QFC4}} \mu(\psi)$.
\end{theorem}

\begin{proof} For the right-to-left direction, if $\varphi \nvdash_{\mathsf{IQC}}\psi$, then by a standard argument (see, e.g., \citealt{Scott2008}) there is a \textit{complete} Heyting algebra $H$ and predicate valuation $\theta$ on $H$ such that $\tilde{\theta}(\varphi)\not\leq_H \tilde{\theta}(\psi)$, interpreting $\forall$ and $\exists$ using meets and joins in the standard way. Equipping $H$ with a unary operation $\Box$ that is the identity operation, it is easy to see that $(H,\Box)$ is a complete $\mathbf{FC4}$-algebra, appropriate for interpreting first-order fundamental logic as in \citealt[\S5]{Holliday2023}, and satisfying the infinitary law needed for the soundness of the $\exists$E rule when $\Box$-Reiteration brings boxed formulas into the $\exists$E subproofs:
\[\big(\underset{i\in I}{\bigvee}a_i\big)\wedge\Box c\leq \underset{i\in I}{\bigvee}(a_i\wedge \Box c).\]
A trivial induction yields $\tilde{\theta}(\mu(\chi))=\tilde{\theta}(\chi)$ for all formulas $\chi$. Thus, $\tilde{\theta}(\mu(\varphi))\not\leq_{H}\tilde{\theta}(\mu(\psi))$, so $\mu(\varphi)\nvdash_{\mathsf{QFC4}}\mu(\psi)$ by the soundness of $\mathsf{QFC4}$ with respect to complete $\mathbf{FC4}$-algebras satisfying the displayed infinitary law.

The left-to-right direction is proved as Corollary~\ref{IQCtoFun} in Appendix~\ref{AppendixIQC}.\end{proof}

\subsection{Interpreting Classical Logic}\label{ClassicalSect}

In this section, we contemplate the relation of the foregoing results to the interpretation of \textit{classical} logic into fundamental logic with necessity. Recall the G\"{o}del-Gentzen double negation translation of classical logic into intuitionistic logic (\citealt{Godel1933b}, \citealt{Gentzen1936}), here restricted to the $\{\wedge,\vee,\neg\}$ fragment:
\begin{itemize}
\item $g(\top)=\top$ and  $g(p)=\neg\neg p$;
\item $g(\neg\varphi)=\neg g(\varphi)$;
\item $g(\varphi\wedge\psi)=(g(\varphi)\wedge g(\psi))$;
\item $g(\varphi\vee\psi)=\neg(\neg g(\varphi)\wedge\neg g(\psi))$.
\end{itemize}

\begin{theorem}[\citealt{Godel1933b}, \citealt{Gentzen1936}]\label{DoubleNegThm} For any $\varphi,\psi\in\mathcal{L}(\wedge,\vee,\neg)$, $\varphi\vdash_\mathsf{CPC}\psi$ iff $g(\varphi)\vdash_{\mathsf{IPC}^-}g(\psi)$.
\end{theorem}
\noindent Combining this with Theorem~\ref{IntFN4FullFaithfulv1}, we obtain a full and faithful translation of classical logic into $\mathsf{FN4}$.

\begin{corollary}\label{CPCtoFN4v1} For any $\varphi,\psi\in \mathcal{L}(\wedge,\vee,\neg)$, $\varphi\vdash_{\mathsf{CPC}}\psi$ iff $\mu(g(\varphi))\vdash_{\mathsf{FN4}}\mu(g(\psi))$.
\end{corollary}
\noindent The extension of Theorem~\ref{DoubleNegThm} to first-order logic (with $g(\forall v\varphi)=\forall vg(\varphi)$ and $g(\exists v\varphi)=\neg\forall v\neg g(\varphi)$) combined with Theorem~\ref{QFC4Thm} yields a full and faithful embedding of classical first-order logic into $\mathsf{QFC4}$.\footnote{For this embedding we do not need conditionals in the fundamental language, so we can give a full and faithful embedding of classical first-order logic into $\mathsf{QFN4}$, defined in the obvious way.}

A natural question is whether the order of the G\"{o}del-Gentzen and GMT translations in Corollary~\ref{CPCtoFN4v1} can be reversed, so that we first apply GMT and then G\"{o}del-Gentzen. This requires extending the G\"{o}del-Gentzen translation to $\mathcal{L}(\wedge,\vee,\neg,\Box)$. If we do so with the clause $g(\Box\varphi)=\Box g(\varphi)$, then we cannot reverse the order of the translations; for example, $\neg\neg p\vdash_{\mathsf{CPC}}p$ while $g(\mu(\neg\neg p)) =\Box\neg\Box\neg\Box\neg\neg p\not\vdash_{\mathsf{FN4}} \Box\neg\neg p = g(\mu(p))$. Instead, we will use the extension $g^*$ of the G\"{o}del-Gentzen translation to the modal language defined just like $g$ but with the following $\Box$ clause, where $\diamondsuit \varphi :=\neg\Box\neg\varphi$
\[g^*(\Box\varphi)=\Box \diamondsuit\Box g^*(\varphi).\]
We will also use a variant $\mu^+$ of the GMT translation $\mu$ that changes the $\wedge$ and $\vee$ clauses to  
\begin{itemize}
\item $\mu^+(\phi \me \psi) = \Box(\mu^+(\phi) \me \mu^+(\psi))$ and 
\item $\mu^+(\phi \jo \psi) = \Box(\mu^+(\phi) \jo \mu^+(\psi))$,
\end{itemize} respectively (see Definition~\ref{mu+Def} for the full definition). It is easy to see that for all formulas $\varphi\in\mathcal{L}(\wedge,\vee,\neg)$, $\mu(\varphi)\dashv \vdash_{\mathsf{FN4}}\mu^+(\varphi)$. Yet it is not the case that for all $\varphi\in\mathcal{L}(\wedge,\vee,\neg)$, $g^*(\mu(\varphi))\dashv \vdash_{\mathsf{FN4}}g^*(\mu^+(\varphi))$. A reason we must use $\mu^+$ is that while $\top\vdash_{\mathsf{CPC}} p\vee\neg p$, even classical $\mathsf{S4}$ cannot carry out the translated inference under $g^*\circ \mu$: \[g^*(\mu(\top))= \top  \nvdash_{\mathsf{S4}} \neg (\neg \Box \diamondsuit\Box \neg\neg p \wedge \neg \Box \diamondsuit\Box \neg \Box \diamondsuit\Box\neg\neg p)= g^*(\mu(p\vee\neg p)),\] as shown by a three-world Kripke model with $v_2 \shortdashleftarrow w\shortdashrightarrow v_1$, neither $v_1 \shortdashrightarrow v_2$ nor $v_2 \shortdashrightarrow v_1$, and $p$ true only at $v_1$. By contrast, we do have $g^*(\mu^+(\top))\vdash_{\mathsf{FN4}} g^*(\mu^+(p\vee\neg p))$. Indeed, we have the following general result.

\begin{theorem}\label{CPCtoFN4v2} For any $\varphi,\psi\in\mathcal{L}(\wedge,\vee,\neg)$, $\varphi\vdash_{\mathsf{CPC}}\psi$ iff $g^*(\mu^+(\varphi))\vdash_{\mathsf{FN4}}g^*(\mu^+(\psi))$.\end{theorem}

\begin{proof} From right to left, if $\varphi\nvdash_{\mathsf{CPC}}\psi$, then we have a Boolean algebra $B$ and valuation $\theta$ on $B$ such that $\tilde{\theta}(\varphi)\not\leq_B \tilde{\theta}(\psi)$. Then equipping $B$ with $\Box$ as the identity operation, $(B,\Box)$ is an algebra for $\mathsf{FN4}$, and a trivial induction shows that $\tilde{\theta}(\chi)=\tilde{\theta}(g^*(\mu^+(\chi)))$. Then the rest of the proof is analogous to that of Proposition~\ref{FN4toInt}.

From left to right, suppose $g^*(\mu^+(\varphi))\nvdash_{\mathsf{FN4}}g^*(\mu^+(\psi))$, so there is an $\mathsf{FN4}$-algebra $(L,\neg,\Box)$ and valuation $\theta$ on $L$ such that $\tilde{\theta}(g^*(\mu^+(\varphi)))\not\leq_L \tilde{\theta}(g^*(\mu^+(\psi)))$. By Proposition~\ref{BoxedElementsProp1}, the boxed elements of $(L,\neg,\Box)$ form a distributive lattice with the pseudocomplementation given by $\neg_\Box a:=\Box\neg a$. Now let $B$ be the Boolean algebra of $\neg_\Box\neg_\Box$-fixpoints in that pseudocomplemented distributive lattice. Note that for any boxed element $a$, we have $\neg_\Box \neg_\Box a = \Box\neg\Box\neg a =\Box\neg\Box\neg\Box a=\Box\diamondsuit\Box a$. Also note that in $B$, the join is computed as $a\vee_B b= \neg_\Box(\neg_\Box a\wedge\neg_\Box b)$. Let $\theta'$ be the valuation on $B$ defined by $\theta'(p)=\tilde{\theta}(g^*(\mu^+(p)))$, and let $\tilde{\theta'}$ be the extension of $\theta'$ computed inside $B$. Then we can prove by induction that for all formulas $\chi$, $\tilde{\theta'}(\chi)=\tilde{\theta}(g^*(\mu^+(\chi)))$. The $\vee$ case relies on the lemma that for all $a,b\in B$,
\[\Box\diamondsuit\Box \neg (\neg a\wedge\neg b)=a\vee_B b.\]
First, $a\leq \neg (\neg a\wedge\neg b)$, so $\Box a\leq \Box \neg (\neg a\wedge\neg b)\leq \neg_\Box\neg_\Box \Box \neg (\neg a\wedge\neg b) = \Box\diamondsuit\Box \neg (\neg a\wedge\neg b)$, and similarly $b\leq \Box\diamondsuit\Box \neg (\neg a\wedge\neg b)$, so $a\vee_B b \leq \Box\diamondsuit\Box \neg (\neg a\wedge\neg b)$. Conversely, $\neg_\Box a\wedge\neg_\Box b\leq \neg a\wedge\neg b$, so $\neg (\neg a\wedge\neg b)\leq \neg (\neg_\Box a\wedge\neg_\Box b)$ and hence $\Box\neg (\neg a\wedge\neg b)\leq \Box\neg (\neg_\Box a\wedge\neg_\Box b)$, which implies \[\neg_\Box\neg_\Box \Box\neg (\neg a\wedge\neg b)\leq \neg_\Box\neg_\Box\Box\neg (\neg_\Box a\wedge\neg_\Box b) = \Box\neg (\neg_\Box a\wedge\neg_\Box b) = a\vee_B b,\]
where for the penultimate equation we use the fact that $\Box\neg (\neg_\Box a\wedge\neg_\Box b)\in B$, so it is a fixpoint of $\neg_\Box\neg_\Box$. Then for the $\vee$ case, we have 
\begin{align*}
\tilde{\theta}(g^*(\mu^+(\alpha\vee\beta))) & = \tilde{\theta}(g^*(\Box(\mu^+(\alpha)\vee \mu^+(\beta)))) \\
& = \tilde{\theta}(\Box\diamondsuit\Box \neg (\neg g^*(\mu^+(\alpha))\wedge\neg g^*(\mu^+(\beta)))) \\
&= \Box\diamondsuit\Box \neg (\neg \tilde{\theta}(g^*(\mu^+(\alpha)))\wedge\neg \tilde{\theta}(g^*(\mu^+(\beta)))) \\ 
& = \Box\diamondsuit\Box \neg (\neg \tilde{\theta'}(\alpha)\wedge\neg \tilde{\theta'}(\beta)) \\
& = \tilde{\theta'}(\alpha)\vee_B \tilde{\theta'}(\beta)\\
&= \tilde{\theta'}(\alpha\vee\beta).
\end{align*}
Thus, $\tilde{\theta}(g^*(\mu^+(\varphi)))\not\leq_L \tilde{\theta}(g^*(\mu^+(\psi)))$ implies $\tilde{\theta'}(\varphi)\not\leq_B \tilde{\theta'}(\psi)$ and hence $\varphi\nvdash_{\mathsf{CPC}}\psi$.\end{proof}

It is natural to want even more than Theorem~\ref{CPCtoFN4v2}, namely that $\mu^+$ is a full and faithful embedding of $\mathsf{CPC}$ into some orthomodal logic $\mathsf{L}$ and then $g^*$ is a full and faithful embedding of $\mathsf{L}$ into $\mathsf{FN4}$ or some natural variant thereof. The second step would be a modal analogue of the following analogue of Theorem~\ref{DoubleNegThm}.

\begin{proposition}[\citealt{Holliday2023}] For any $\varphi,\psi\in\mathcal{L}(\wedge,\vee,\neg)$, $\varphi\vdash_\mathsf{O}\psi$ iff $g(\varphi)\vdash_{\mathsf{F}}g(\psi)$.
\end{proposition}

Now to extend this proposition to an orthomodal logic $\mathsf{L}$ and $g^*$, an immediate observation is that $\mathsf{L}$ cannot contain the \textsf{T} Rule, because $g^*(\Box p)=\Box\diamondsuit\Box \neg\neg p \nvdash_{\mathsf{FN4}}\neg\neg p = g^*(p)$. Thus, we might seek an orthomodal logic of \textit{certainty} rather than necessity, since one can be certain of a falsehood. Accordingly, let us weaken $\mathsf{FN4}$ to a logic $\mathsf{FM4}$ by replacing the \textsf{T} Rule with the rule that lets us infer $\Box\varphi$ from $\Box\Box\varphi$ plus the rule that lets us infer $\psi$ from $\Box \bot$. Similarly, let an $\mathbf{FM4}$ algebra be defined like an $\mathbf{FN4}$ algebra but with $\Box a\leq a$ replaced by $\Box\Box a\leq\Box a$ and $\Box 0 = 0$. Then we have the following analogue of Proposition~\ref{BoxedElementsProp1}.

\begin{proposition}\label{BoxedElementsPropM} For any $\mathbf{FM4}$-algebra $(L,\neg,\Box)$, the set $\{\Box a\mid a\in L\}$ becomes a bounded distributive lattice under the meet operation of $L$, the join given by $a\vee_\Box b = \Box (a\vee b)$, and the bounds of $L$, and it is pseudocomplemented under the operation $\neg_\Box$ defined by $\neg_\Box(a):=\Box\neg a$.
\end{proposition}

\begin{proof} That the boxed elements in an $\mathbf{FM4}$ algebra form a bounded lattice under the meet of $L$ and the join $\vee_\Box$ is an easy exercise. That the bounds are the same as those of $L$ follows from $\Box 0=0$. For distributivity, we must show
\[\Box (\Box a\vee\Box b)\wedge\Box c\leq \Box ((\Box a\wedge\Box c)\vee (\Box b\wedge \Box c)).\]
Now $(\Box a\vee\Box b)\wedge\Box c\leq  (\Box a\wedge\Box c)\vee (\Box b\wedge \Box c)$ holds in any $\mathbf{FM4}$ algebra (cf.~Definition~\ref{FN4AlgDef}.\ref{FN4AlgDef2}). Thus,
\begin{align*}
\Box (\Box a\vee\Box b)\wedge\Box c & \leq \Box (\Box a\vee\Box b)\wedge\Box \Box c\\
& \leq \Box ((\Box a\vee\Box b)\wedge \Box c) \\
& \leq \Box ((\Box a\wedge\Box c)\vee (\Box b\wedge \Box c)).
\end{align*}
 For pseudocomplementation under $\neg_\Box$, the proof is the same as in the proof of Proposition~\ref{BoxedElementsProp1}, where we did not use $\Box x\leq x$.
\end{proof}

Now the desired logic $\mathsf{L}$ can sit in the interval of logics given by the following theorem.

\begin{theorem}\label{CPCInterval}  Let $\mathsf{L}$  be the logic of a class of $\mathbf{FM4}$-algebras satisfying the law that $\Box\diamondsuit\Box a\leq \Box a$ and containing all Boolean $\mathbf{S5}$-algebras. Then for any $\varphi,\psi\in \mathcal{L}(\wedge,\vee,\neg)$, $\varphi\vdash_{\mathsf{CPC}}\psi$ iff $\mu^+(\varphi)\vdash_{\mathsf{L}}\mu^+(\psi)$.
\end{theorem}

\begin{proof} The right-to-left direction of the proof is similar to that of Theorem~\ref{CPCtoFN4v2} except we observe that $(B,\Box)$ is a Boolean $\mathbf{S5}$-algebra.

 From left to right, suppose $\mu^+(\varphi)\nvdash_{\mathsf{L}}\mu^+(\psi)$, so we have an $\mathbf{FM4}$-algebra $(L,\neg,\Box)$ that satisfies $\Box\diamondsuit\Box a\leq \Box a$ for all $a\in L$ and valuation $\theta$ on $L$ such that $\tilde{\theta}(\mu^+(\varphi)) \not\leq_L \tilde{\theta}(\mu^+(\psi))$. Now by Proposition~\ref{BoxedElementsPropM}, the boxed elements of $(L,\neg,\Box)$ form a distributive lattice that is pseudocomplemented under the operation $\neg_\Box$ defined by $\neg_\Box a:=\Box\neg a$. Now we simply observe that for any boxed element $a$, we have $\neg_\Box \neg_\Box a \leq \Box\diamondsuit \Box a\leq \Box a =a$. Hence the boxed elements now form a Boolean algebra $B$ under $\neg_\Box$. Let $\theta'$ be the valuation on $B$ defined by $\theta'(p)=\tilde{\theta}(\mu^+(p))$. Then we can prove by induction that for any formula $\chi$, $\tilde{\theta'}(\chi)=\tilde{\theta}(\mu^+(\chi))$. Thus, $\tilde{\theta}(\mu^+(\varphi)) \not\leq_L \tilde{\theta}(\mu^+(\psi))$ implies $\tilde{\theta'}(\varphi) \not\leq_B \tilde{\theta'}(\psi)$ and hence $\varphi\nvdash_{\mathsf{CPC}}\psi$.\end{proof}

Now we cannot faithfully embed orthomodal logics $\mathsf{L}$ into $\mathsf{FM4}$ itself via $g^*$, since $\neg\neg \Box p\vdash_{\mathsf{L}} \Box p$ while \[g^*(\neg\neg\Box p)=\neg\neg \Box\diamondsuit \Box \neg\neg p\nvdash_{\mathsf{FM4}} \Box\diamondsuit \Box \neg\neg p = g^*(\Box p).\]
To deal with this problem, let $\mathsf{FM4}^*$ be the extension of $\mathsf{FM4}$ with the $\diamondsuit$-RAA rule in Figure~\ref{BoxNegRAA}. Note that if we were to remove the boxes from the rule, then $\diamondsuit$-RAA would be equivalent to the intuitionistic principle that if we can prove $\neg\neg\neg\varphi$ (by deriving $\neg\psi$ from $\neg\neg\varphi$ for an established $\psi$), then we can prove $\neg\varphi$. The associated class of $\mathbf{FM4}^*$-algebras are $\mathbf{FM4}$-algebras also satisfying the principle that $\neg\neg \Box \neg a\leq \Box\neg a$. Note that the inference from $\neg\neg\Box\neg\varphi$ to $\Box\neg\varphi$ holds in Fischer Servi's minimal intuitionistic modal logic $\mathsf{FS}$ (\citealt{DasMarin2022}), and the inference from $\neg\neg\forall x \neg P(x)$ to $\forall x \neg P(x)$ holds in intuitionistic first-order logic.\footnote{Since $\forall x \neg P(x) \vdash_{\mathsf{IQC}} \neg P(x)$, we have $\neg\neg\forall x \neg P(x) \vdash_{\mathsf{IQC}} \neg\neg\neg P(x)\vdash \neg P(x)$ and hence $\neg\neg\forall x \neg P(x) \vdash_{\mathsf{IQC}} \forall x \neg P(x)$ by $\forall$I.}

\begin{figure}[H]
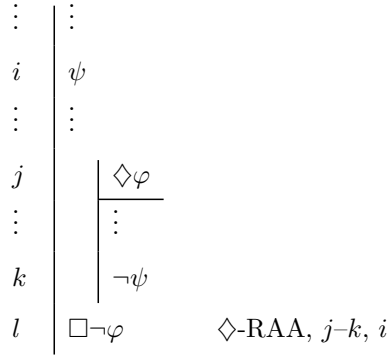

\begin{center}
\begin{minipage}{2.5in}
\[\begin{nd}
\have [\vdots] {0} {\vdots}
\have [i] {3}   {\psi}
\have [\vdots] {4}   {\vdots}
\open
\hypo[ j] {1} {\diamondsuit\varphi}
\have [\vdots] {}  {\vdots}
\have [k] {5}   {\neg\psi}
\close
\have [l]{6} {\Box\neg\varphi} \diRAA{1-5,3}
\end{nd}
\]\end{minipage}
\end{center}
\caption{The $\diamondsuit$-RAA rule of $\mathsf{FM4}^*$.}\label{BoxNegRAA}
\end{figure}

We can now define the desired orthomodal logic. Let $\mathsf{OMR}$ (where $\mathsf{R}$ suggests \textit{regular}, as in Proposition~\ref{RegInFM4*} below) be the orthomodal logic with $\Box$-Reiteration, the \textsf{4} Rule, and the following rules:
\begin{itemize}
    \item from $\Box\diamondsuit\Box\varphi$, infer $\Box \varphi$;
    \item from $\Box\Box\varphi$, infer $\Box\varphi$; 
    \item from $\Box\bot$, infer $\psi$.
\end{itemize}
Thus, by Theorem~\ref{CPCInterval}, $\mu^+$ is a full and faithful embedding of $\mathsf{CPC}$ into $\mathsf{OMR}$. Let $\mathbf{OMR}$ algebras be defined just like $\mathbf{FM4}$ algebras except that $(L,\neg)$ is an ortholattice and we add the requirement that $\Box\diamondsuit\Box a\leq \Box a$. As usual, one can show that $\mathsf{OMR}$ is sound and complete with respect to the class of $\mathbf{OMR}$ algebras.

\begin{samepage}
 \begin{fact}\label{OMRFM4*} $\,$
 \begin{enumerate}
     \item\label{OMRFM4*1} Every $\mathbf{OMR}$-algebra is an $\mathbf{FM4}^*$ algebra.
\item\label{OMRFM4*2} In any $\mathbf{OMR}$-algebra, $\Box a\leq \diamondsuit a$.
      \end{enumerate}
 \end{fact}
 \end{samepage}
 \begin{proof} For part \ref{OMRFM4*1}, every principle of $\mathbf{FM4}^*$-algebras is explicitly built into the definition of $\mathbf{OMR}$-algebras, except for $\neg\neg\Box\neg a\leq\Box\neg a$, but this holds in any orthomodal lattice.

 For part \ref{OMRFM4*2}, $\Box a\wedge\Box \neg a\leq \Box (a\wedge\neg a)=0$, so by the version of pseudocomplementation in $\mathbf{OMR}$-algebras (cf.~Definition~\ref{FN4AlgDef}.\ref{FN4AlgDef3}), $\Box a\leq \neg\Box\neg a$.\end{proof}

 Now the key algebraic fact for relating $\mathsf{OMR}$ and $\mathsf{FM4}^*$ is the following.

\begin{proposition}\label{RegInFM4*} For any $\mathbf{FM4}^*$-algebra $(L,\neg,\Box)$, the set $\{\neg\neg a \mid a\in L\}$ of regular elements becomes an $\mathbf{OMR}$ algebra under the meet and negation operation of $L$, the join given by $a\vee_O b:=\neg(\neg a\wedge\neg b)$, and the box given by $\Box_O a:=\Box\diamondsuit \Box a$.
\end{proposition}

\begin{proof} It is easy to see that taking the meet and negation operation of $L$ and the join $\vee_O$ gives an ortholattice on the set $\{\neg\neg a\mid a\in L\}$. Moreover, thanks to the principle $\neg\neg\Box\neg a = \Box\neg a$ of $\mathbf{FM4}^*$ algebras, $\{\neg\neg a\mid a\in L\}$ is closed under $\Box$ in addition to $\neg$ and hence closed under $\Box_O$. To see that the resulting orthomodal lattice is an $\mathbf{OMR}$ algebra, we must  check that each of the following holds in $(L,\neg,\Box)$:
\begin{itemize}
\item $\Box_O 1= 1$, $\Box_O(a\wedge b)=\Box_Oa\wedge\Box_Ob$, $\Box_O a =  \Box_O\Box_O a$, $\Box_O \diamondsuit_O \Box_O a\leq \Box_O a$, and $\Box_O 0=0$;
\item if $a\wedge\Box_O c\leq \neg b$, then $b\wedge\Box_O c\leq\neg a$; 
\item $(a\vee_O b)\wedge\Box_Oc\leq (a\wedge\Box_Oc)\vee_O (b\wedge\Box_Oc)$.
\end{itemize}

Note that $\Box_O a = \neg_\Box \neg_\Box \Box a$, where $\neg_\Box$ is the pseudocomplementation on boxed elements from Proposition~\ref{BoxedElementsPropM}. Thus, we have
\begin{align*}
\Box_O 1 &= \neg_\Box\neg_\Box \Box 1 = \neg_\Box\neg_\Box 1 = 1 \\  
\Box_O(a\wedge b) & = \neg_\Box\neg_\Box \Box ( a\wedge b) = \neg_\Box\neg_\Box ( \Box a\wedge \Box b) = (\neg_\Box\neg_\Box\Box a\wedge \neg_\Box\neg_\Box\Box b)= \Box_O a\wedge \Box_O b\\
\Box_O\Box_O a & =  \neg_\Box \neg_\Box \Box \neg_\Box \neg_\Box \Box a = \neg_\Box \neg_\Box \neg_\Box \neg_\Box \Box a = \neg_\Box \neg_\Box \Box a = \Box_O a \\ 
\Box_O \neg \Box_O \neg \Box_O a &= \neg_\Box \neg_\Box \Box \neg \neg_\Box \neg_\Box \Box \neg \neg_\Box \neg_\Box \Box a = \neg_\Box \neg_\Box \neg_\Box \neg_\Box \neg_\Box \neg_\Box \neg_\Box \neg_\Box \Box a =  \neg_\Box \neg_\Box \Box a =\Box_Oa \\
\Box_O 0 & = \neg_\Box\neg_\Box \Box 0 =  \neg_\Box\neg_\Box 0 = 0.
\end{align*}
The second bullet point is immediate from the form of pseudocomplementation in $\mathbf{FM4}^*$-algebras (cf.~Definition \ref{FN4AlgDef}.\ref{FN4AlgDef3}) given that $\Box_O$ begins with a $\Box$. For the third bullet point, consider
\[(a\wedge\Box_Oc)\vee_O (b\wedge\Box_Oc) = \neg \big(\neg (a\wedge\Box_Oc)\wedge \neg (b\wedge\Box_Oc)\big).\]
Let $x= \big(\neg (a\wedge\Box_Oc)\wedge \neg (b\wedge\Box_Oc)\big)$. We want to prove that $(a\vee_O b)\wedge\Box_O c\leq \neg x$. Now $x\leq \neg (a\wedge\Box_O c)$, so $a\wedge\Box_O c\leq \neg x$ by dual self-adjointness. Then by the form of pseudocomplementation in $\mathbf{FM4}^*$ algebras, we have $x \wedge \Box_O c\leq\neg a$. By analogous reasoning, we have $x \wedge \Box_O c\leq\neg b$. Thus, $x \wedge \Box_O c\leq \neg a\wedge\neg b = \neg\neg (\neg a\wedge\neg b)=\neg  (a\vee_O b) $. Then by the form of pseudocomplementation in $\mathbf{FM4}^*$ again, we have $(a\vee_O b)\wedge\Box_O c\leq \neg x$.\end{proof}

We can now finish this story by showing that $g^*$ is a full and faithful embedding of $\mathsf{OMR}$ into $\mathsf{FM4}^*$ and $\mu^+$ is a full and faithful embedding of $\mathsf{IPC}^-$ into $\mathsf{FM4}^*$. Thus, the translations allow us to reverse the arrows in the ``fundamental diamond'' in Figure~\ref{fundiamond}, obtaining the ``reversed fundamental diamond'' in Figure~\ref{reversediamond}.

\begin{theorem}\label{FM4*} $\,$
\begin{enumerate}
    \item\label{FM4*1} For all $\varphi,\psi\in\mathcal{L}(\wedge,\vee,\neg,\Box)$, $\varphi\vdash_{\mathsf{OMR}}\psi$ iff $g^*(\varphi)\vdash_{\mathsf{FM4}^*}g^*(\psi)$.
    \item\label{FM4*2} For all $\varphi,\psi\in\mathcal{L}(\wedge,\vee,\neg)$, $\varphi\vdash_{\mathsf{IPC}^-}\psi$ iff $\mu^+(\varphi)\vdash_{\mathsf{FM4}^*}\mu^+(\psi)$.
\end{enumerate}
\end{theorem}

\begin{proof} For part~\ref{FM4*1}, suppose $\varphi\nvdash_{\mathsf{OMR}}\psi$, so there is an $\mathbf{OMR}$-algebra $\mathcal{O}=(L,\neg,\Box)$ and valuation $\theta$ on $\mathcal{O}$ such that $\tilde{\theta}(\varphi)\not\leq_\mathcal{O} \tilde{\theta}(\psi)$. By Fact~\ref{OMRFM4*}.\ref{OMRFM4*1}, $\mathcal{O}$ is also an $\mathbf{FM4}^*$ algebra. Now we can prove by induction that for any $\chi$, $\tilde{\theta}(\chi)=\tilde{\theta}(g^*(\chi))$. Then $\tilde{\theta}(\varphi)\not\leq_\mathcal{O} \tilde{\theta}(\psi)$ implies $\tilde{\theta}(g^*(\varphi))\not\leq_\mathcal{O} \tilde{\theta}(g^*(\psi))$ and hence $g^*(\varphi)\nvdash_{\mathsf{FM4}^*}g^*(\psi)$.

For the left-to-right direction, if $g^*(\varphi)\nvdash_{\mathsf{FM4}^*}g^*(\psi)$, then there is an $\mathbf{FM4}^*$ algebra $\mathcal{A}=(L,\neg,\Box)$ and valuation $\theta$ such that $\tilde{\theta}(g^*(\varphi))\not\leq_\mathcal{A} \tilde{\theta}(g^*(\psi))$. Now an easy induction shows that for any formula $\chi$, $\tilde{\theta}(g^*(\chi))=\neg\neg \tilde{\theta}(g^*(\chi))$, so $\tilde{\theta}(g^*(\chi))\in \{\neg\neg a\mid a\in L\}$. Let $\mathcal{O}$ be the $\mathbf{OMR}$ algebra based on $\{\neg\neg a\mid a\in L\}$ given by Proposition~\ref{RegInFM4*}, and let $\theta'$ be the valuation on $\mathcal{O}$ defined by $\theta'(p)=\tilde{\theta}(g^*(p))$. Then we can prove by induction that for any formula $\chi$, $\tilde{\theta'}(\chi)= \tilde{\theta}(g^*(\chi))$. Thus, $\tilde{\theta}(g^*(\varphi))\not\leq_\mathcal{A} \tilde{\theta}(g^*(\psi))$ implies $\tilde{\theta'}(\varphi)\not\leq_\mathcal{O} \tilde{\theta'}(\psi)$ and hence $\varphi\nvdash_{\mathsf{OMR}}\psi$.

For part~\ref{FM4*2}, the right-to-left direction of the proof is essentially the same as that of Proposition~\ref{FN4toInt}, except we use $\mu^+$ and note that adding an identity $\Box$ to a pseudocomplemented distributive lattice gives us not only an $\mathbf{FN4}$-algebra but also an $\mathbf{FM4}^*$ algebra. The proof of the left-to-right direction is essentially the same as that of Corollary~\ref{IPC-toFN4} except we use Proposition~\ref{BoxedElementsPropM} instead of Proposition~\ref{BoxedElementsProp1}.\end{proof}

\begin{figure}[H]
    \centering
    \begin{tikzpicture}[
        ->,>={stealth'[length=6pt]},
        shorten >=1pt,shorten <=1pt,
        semithick,dashed,
        desc/.style={fill=white,inner sep=2pt}
    ]
        \node (CPC) at (2,4) {$\mathsf{CPC}$};
        \node (OMR) at (0,2) {$\mathsf{OMR}$};
        \node (IPC) at (4,2) {$\mathsf{IPC}^-$};
        \node (FM4) at (2,0) {$\mathsf{FM4}^*$};

        \path (CPC) edge (OMR);
        \path (CPC) edge (IPC);
        \path (OMR) edge (FM4);
        \path (IPC) edge (FM4);

        \node[desc] at ($(CPC)!0.5!(OMR)!3pt!(CPC)$) {$\mu^+$};
        \node[desc] at ($(CPC)!0.5!(IPC)!3pt!(CPC)$) {$g$};
        \node[desc] at ($(OMR)!0.5!(FM4)!3pt!(OMR)$) {$g^*$};
        \node[desc] at ($(IPC)!0.5!(FM4)!3pt!(IPC)$) {$\mu^+$};
    \end{tikzpicture}
    \caption{The ``reversed fundamental diamond'', reversing the arrows from Figure~\ref{fundiamond}. Dashed arrows represent full and faithful translations from the source logic to the target logic.}
    \label{reversediamond}
\end{figure}
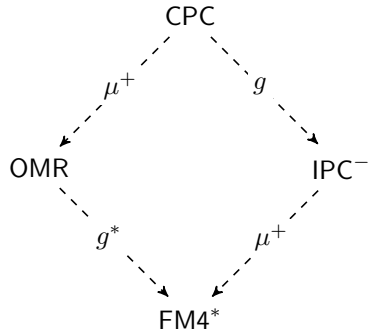

\begin{remark} The two paths from classical logic to fundamental modal logic in Figure~\ref{reversediamond} do not coincide: at the atomic level, $g^*(\mu^+(p))=g^*(\Box p) = \Box\diamondsuit\Box g^*(p)=\Box\diamondsuit\Box \neg \neg p$ is not equivalent to $\mu^+(g^*(p))= \mu^+(\neg\neg p) = \Box\neg\Box\neg \mu^+(p)=\Box\neg\Box\neg \Box p = \Box\diamondsuit\Box p$ over $\mathsf{FM4}^*$. Thus, there is a genuine difference between the fundamental logician's understanding of the classical logician as filtered through someone who accepts RAA (the orthologician) or through someone who accepts Reiteration (the intuitionist).\end{remark}

\section{Further Directions}\label{FurtherDirections}

In this section, we highlight some directions for future research suggested by the results presented here.

\subsection{Superfundamental Logics}\label{Superfun}

In \cref{OS4Section,KTBSection}, we showed that the GMT translation and the Goldblatt translation can be generalized to obtain translations of fundamental logic into orthomodal logic and intuitionistic modal logic, respectively. In their standard setting, these translations also apply to the lattice of logics extending the source logic of the translation. In our setting, this means that we can explore how the two translations of fundamental logic we have discussed behave with respect to superfundamental logics, i.e., propositional logics in the signature $\{\land, \lor, \neg\}$ that extend fundamental logic. We start with the following definitions.

\begin{definition}
    Let $\mathcal{SF}$ be the lattice of all propositional logics extending fundamental logic. Similarly, let $\mathcal{SOM}$ and $\mathcal{SIM}$ be the lattices of modal logics extending $\mathsf{OS4}$ and $\mathsf{FSTB}$, respectively.
\end{definition}

\begin{definition}
    We define the following maps:
    \begin{itemize}
        \item $\mu^*: \mathcal{SF} \to \mathcal{SOM}$, given by $\mu^*(\mathsf{L}) = \mathsf{OS4} \oplus \{\mu(\phi) \vdash \mu(\psi) \mid \phi \vdash \psi \in \mathsf{L}\}$;
        \item $\gamma^*: \mathcal{SF} \to  \mathcal{SIM}$, given by $\gamma^*(\mathsf{L}) = \mathsf{FSTB} \oplus \{\gamma(\phi) \vdash \gamma(\psi) \mid \phi \vdash \psi \in \mathsf{L}\}$.
    \end{itemize}
    
\end{definition}

Recall that a propositional logic $\mathsf{L}$ is \emph{canonical} if a variety of algebras $\mathbb{V}$ providing a sound and complete semantics for $\mathsf{L}$ is closed under canonical extensions. In the case of a superfundamental logic $\mathsf{L}$, this is equivalent to $\mathsf{L}$ being valid on the dual algebra of the reflexive canonical frame $(X, \op)$ induced by any fundamental lattice on which $\mathsf{L}$ is valid (\citealt[Lemma~4.2]{massas2024goldblatt}). Because the results presented above rely on such reflexive canonical frames, we can easily provide full and faithful translations for canonical superfundamental logics, as established by the following result.

\begin{theorem} \label{compthm}
    For any superfundamental logic $\mathsf{L} \in \mathcal{SF}$, $\mu$ and $\gamma$ are faithful translations of $\mathsf{L}$ into $\mu^*(\mathsf{L})$ and $\gamma^*(\mathsf{L})$, respectively. Moreover, if $\mathsf{L}$ is canonical, then both translations are also full.
\end{theorem}

\begin{proof}
    Fix a superfundamental logic $\mathsf{L}$. That $\mu$ and $\gamma$ are faithful translations into $\mu^*(\mathsf{L})$ and $\gamma^*(\mathsf{L})$, respectively, is clear. Suppose that $\mathsf{L}$ is canonical. To show that the translation $\mu$ (resp.~$\gamma$) is full, it is enough to show that for any $\mathsf{L}$-fundamental lattice $L$, there is an embedding of $L$ into a $\mu^*(\mathsf{L})$-$\mathbf{OS4}$ lattice with the properties of \cref{embthm} (resp.~an embedding of $L$ into a $\gamma^*(\mathsf{L})$-\textbf{FSTB}-lattice with the properties of \cref{embthm2}). Fix an $\mathsf{L}$-fundamental lattice $L$, and consider the dual algebra $L^\delta$ of its dual reflexive frame $(X, \op)$. Since $\mathsf{L}$ is canonical and $L^\delta$ is the canonical extension of $L$, it follows that $L^\delta$ is also an $\mathsf{L}$-fundamental lattice. Now let $(O, \Box_\leq)$ be the dual algebra of the orthomodal companion of $(X, \op)$. By the proof of \cref{embthm}, the identity is an embedding of $L^\delta$ into $(O, \Box_\leq)$ with the required properties. Moreover, by \cref{keylma}, it also maps $L^\delta$ onto the $\Box_\leq$ fixpoints of $(O, \Box_\leq)$. But from this plus the fact that $L^\delta$ is an $\mathsf{L}$-fundamental lattice, it follows that $(O, \Box_\leq)$ is a $\mu^*(\mathsf{L})$-$\mathbf{OS4}$ lattice. Similarly, letting $(A, \Box_\top, \diamondsuit_\top)$ be the \textbf{FSTB}-lattice of downsets of the \textbf{FSTB}-companion of $(X, \op)$, by \cref{embthm2} we have that the identity is an embedding of $L^\delta$ into $(A, \Box_\top, \diamondsuit_\top)$ with the required properties. Moreover, by \cref{fstbkeylma}, this embedding maps $L^\delta$ onto the $\Box_\top\diamondsuit_\top$ fixpoints of $A$. From this, together with the fact that $L^\delta$ is a $\mathsf{L}$-fundamental lattice, it follows that $(A, \Box_\top, \diamondsuit_\top)$ is a $\gamma^*(\mathsf{L})$-\textbf{FSTB}-lattice.
\end{proof}

As an example, we consider the logic $\mathsf{FEM}$, obtained by adding to fundamental logic the excluded middle $\top \vdash p \lor \neg p$. Call a modal logic $\mathsf{M}$ extending $\mathsf{OS4}$ (resp.~$\mathsf{FSTB}$) an
\textit{orthomodal} (resp.~$\mathsf{FSTB}$-) \textit{companion} of a superfundamental logic $\mathsf{L}$
if $\mu$ (resp.~$\gamma$) is a full and faithful translation of $\mathsf{L}$ into $\mathsf{M}$.

\begin{corollary}\label{FEMCor}
 The logic $\mathsf{OS4} \oplus \top \vdash \Box p \lor \Box \neg \Box p$ is an orthomodal companion of $\mathsf{FEM}$, and the logic $\mathsf{FSTB} \oplus \top \vdash \Box \diamondsuit (\Box \diamondsuit p \lor \Box \neg \Box\diamondsuit p)$ is an $\mathsf{FSTB}$-companion of $\mathsf{FEM}$.
\end{corollary}

\noindent A routine induction on $\mathsf{FEM}$-proofs shows that the two logics in Corollary \ref{FEMCor} coincide with $\mu^*(\mathsf{FEM})$ and $\gamma^*(\mathsf{FEM})$, respectively. Thus, Corollary \ref{FEMCor} follows from \cref{compthm} once we establish that $\mathsf{FEM}$ is canonical. To do this, it is enough to show that for any fundamental lattice $L$ satisfying the equation $a \lor \neg a = 1$, the dual algebra $\chi_{\op}(X)$ of its dual reflexive frame $(X,\op)$ also satisfies excluded middle. We start with the following lemma.

\begin{lemma}\label{FEMframe}
    Let $(L,\neg)$ be a fundamental lattice satisfying excluded middle and $(X,\op)$ its reflexive canonical frame. Then the following holds:
    \[\forall x \exists y \po x \, \forall z(y \po z \Rightarrow z \po x).\]
\end{lemma}

\begin{proof}
    Fix some $x \in X$ of the form $(x_F,x_I)$, and let $y = (\upset\{\neg a \mid a \in x_I\},\dnset\{\neg \neg a \mid a \in x_I\})$. To check that $y \in X$, we first show that $\upset\{\neg a \mid a \in x_I\}$ is a filter. If $\neg a \leq b$ and $\neg a' \leq b'$ for $a, a' \in x_I$, so $a\vee a'\in x_I$, then $\neg(a \lor a') = \neg a \land \neg a' \leq b \land b'\in \upset\{\neg a \mid a \in x_I\}$. To see that it is also proper, note that if $\neg a \leq 0$ for some $a \in x_I$, then $1 = a \lor \neg a = a \lor 0 = a$, which implies that $a \notin x_I$. To see that  $\dnset\{\neg \neg a \mid a \in x_I\}$ is an ideal, note that if $b\leq \neg\neg a$ and $b'\leq\neg \neg a'$ for $a,a'\in x_I$, so $a\vee a'\in x_I$, then $b\vee b'\leq \neg\neg a\vee \neg\neg a'\leq \neg\neg (a\vee a')$, so $b\vee b'\in \dnset\{\neg \neg a \mid a \in x_I\}$. To see that it is proper, note that if $1\leq \neg\neg a$, then $\neg a\leq 0$, which we ruled out above. Next, observe that $\upset\{\neg a \mid a \in x_I\}$ and  $\dnset\{\neg \neg a \mid a \in x_I\}$ are disjoint, since if there are $a,b\in x_I$ such that $\neg a\leq c \leq \neg\neg b$, then $a\vee b\in x_I$, and $\neg (a\vee b)\leq \neg a\wedge\neg b\leq \neg\neg b\wedge\neg b=0$, contradicting the properness of $\upset\{\neg a \mid a \in x_I\}$. Finally, observe that if $b\in \upset\{\neg a \mid a \in x_I\}$, so $\neg a\leq b$, then $\neg b\leq \neg\neg a$, so $\neg b\in \dnset\{\neg \neg a \mid a \in x_I\}$. Thus, we have shown that $y \in X$. 
    
    We now claim that $y \po x$. Suppose, toward a contradiction, that there is $b \in x_I$ such that $\neg a \leq b$ for some $a \in x_I$. Then $1 = a \lor \neg a \leq a \lor b$, contradicting the fact that $a \lor b \in x_I$. This shows that $y \po x$. Finally, let $z = (z_F,z_I)$ be such that $y \po z$. This means that $y_F \cap z_I = \varnothing$. Now for any $a \in x_I$, we have that $\neg a \in y_F$, which implies  $\neg a \notin z_I$. But this in turn means that $a \notin z_F$. Hence $z \po x$. This completes the proof.
\end{proof}

We can now show that the dual algebra $\chi_{\op}(X)$ of the reflexive canonical frame $(X,\op)$ of a fundamental lattice satisfying excluded middle also satisfies  excluded middle. This will establish that the logic $\mathsf{FEM}$ is canonical, as desired.

\begin{lemma}
    Let $(L,\neg)$ be a fundamental lattice satisfying excluded middle and $(X, \op)$ its reflexive canonical frame. Then $\chi_{\op}(X)$ also satisfies excluded middle.
\end{lemma}

\begin{proof}
    We claim that for any $x \in X$ and any $A \sset X$, we have \[x \in \neg_{\op}\neg_{\po}(\neg_{\op}\neg_{\po} A \cup \neg_{\op} A).\]
    In fact, we show that for any $x \in X$, there is $y \po x$ such that $y \in A$ or $y \in \neg_{\op} A$. Clearly, this is enough to prove the claim. Fix some $x \in X$. Suppose that for all $y \po x$, $y \notin A$. By the previous lemma, there is a $y \po x$ such that for any $z \in X$, $y \po z$ implies $z \po x$. By assumption, we have that $z \po x$ implies $z \notin A$. But this means that $y\po z$ implies $z \notin A$, so $y \in \neg_{\op} A$. This completes the proof.
\end{proof}

As a second application of \cref{compthm}, we consider the intersection of orthologic and the $\to$-free fragment of intuitionistic logic. This logic was recently axiomatized by \cite{aguilera2025constructivequantumlogics} as follows.

\begin{definition}
    Let $\mathsf{Ex}$ be the extension of fundamental logic with the following axioms:
    \begin{itemize}
        \item $\neg \neg p \land \neg \neg q \vdash \neg\neg (p \land q)$;
        \item $p \land (q \lor r) \land \neg \neg s \vdash (p \land q) \lor (p \land r) \lor s$;
        \item $\neg \big( p \land ((q \land r) \lor (q \land s))\big) \land p \vdash (q \land (r\lor s)) \lor \neg(q \land (r \lor s))$.
    \end{itemize}
\end{definition}

Since both the $\to$-free fragment of intuitionistic logic (i.e., the logic of pseudocomplemented distributive lattices) and orthologic are canonical, standard facts about canonical extensions \citep[Thm.~2.8]{DunnGehrkePalmigiano} imply that $\mathsf{Ex}$ is also canonical. Therefore, we obtain the following corollary of \cref{compthm}.

\begin{corollary}$\,$
    \begin{enumerate}
        \item $\mu^*(\mathsf{Ex})$ is axiomatized over $\mathsf{OS4}$ by the following axioms:
        \begin{itemize}
            \item $\Box (\diamondsuit \Box p \land \diamondsuit \Box q) \vdash \Box \diamondsuit \Box(p \land q)$;
            \item $\Box p \land (\Box q \lor \Box r) \land \Box \diamondsuit \Box  s \vdash \Box(p \land q) \lor \Box(p \land r) \lor \Box s$;
            \item $\Box\big[ \neg \big( \Box p \land (\Box (q \land r) \lor \Box(q \land s))\big) \land p\big] \vdash (\Box q \land (\Box r\lor \Box s)) \lor \Box\neg (\Box q \land (\Box r \lor \Box s))$.
        \end{itemize}
        \item $\gamma^*(\mathsf{Ex})$ is axiomatized over $\mathsf{FSTB}$ by the following axioms:
        \begin{itemize}
            \item $\Box \neg (\Box \neg \Box \diamondsuit p \lor \Box \neg \Box \diamondsuit q) \vdash \Box \neg \Box \neg \Box (\diamondsuit p \land \diamondsuit q)$;
            \item $\Box \big(\diamondsuit p \land \diamondsuit (q \lor r) \land \neg \Box \neg \Box \diamondsuit s\big) \vdash \Box \diamondsuit \big(\Box(\diamondsuit p \land \diamondsuit q) \lor \Box(\diamondsuit p \land \diamondsuit r) \lor s \big)$;
            \item $\Box \big[\neg \Box \big( \diamondsuit p \land \diamondsuit(\Box (\diamondsuit q \land  \diamondsuit r) \lor \Box (\diamondsuit q \land  \diamondsuit s))\big) \land\diamondsuit p \big] \vdash \Box \diamondsuit\big[\Box (\diamondsuit q \land \diamondsuit( r\lor  s)) \lor \Box\neg\Box(\diamondsuit q \land \diamondsuit(r \lor  s))\big]$.
        \end{itemize}
    \end{enumerate}
\end{corollary}

\begin{proof}
This follows from \cref{compthm}, together with a routine induction on $\mathsf{Ex}$-proofs and some straightforward syntactic simplifications.
\end{proof}

We do not know whether these results extend beyond the realm of canonical superfundamental logics. Therefore, we leave open for now the following problem.

\begin{question} Given an arbitrary (e.g., non-canonical) superfundamental logic $\mathsf{L}$, is $\mu$ (resp. $\gamma$) a full translation of $\mathsf{L}$ into $\mu^*(\mathsf{L})$ (resp. $\gamma^*(\mathsf{L})$)?\end{question} 

\subsection{Blok-Esakia-style Results}\label{BlokEsakia}

Our extension of the G\"odel-McKinsey-Tarski translation to fundamental logic suggests another area of research. In the case of intuitionistic logic, it is well known that Grzegorczyk's logic $\mathsf{Grz}$ is the largest modal companion of $\mathsf{IPC}$, and the function mapping any superintuitionistic logic $\mathsf{L}$ to the logic \[G(\mathsf{L})=\mathsf{Grz} \oplus \{\mu(\phi) \vdash \mu(\psi) \mid \phi \vdash_\mathsf{L} \psi\}\] is an isomorphism between the lattice of superintuitionistic logics and the lattice of modal logics extending $\mathsf{Grz}$. This result, known as the Blok-Esakia theorem (\citealt{blok1976varieties,esakia1976modal}), highlights the tight connection between intuitionistic and modal logic, and it motivates the following two questions.

\begin{question} Does fundamental logic have a greatest orthomodal companion? In other words, is there an extension $\mathsf{O}$ of $\mathsf{OS4}$ such that the $\mu$ translation maps fundamental logic fully and faithfully into $\mathsf{O}$, and, for any orthomodal companion $\mathsf{O'}$ of fundamental logic, $\mathsf{O'}$ is a sublogic of $\mathsf{O}$?\end{question}

\begin{question}Assuming that such a logic $\mathsf{O}$ exists, given a superfundamental logic $\mathsf{L}$, let $G(\mathsf{L})$ be the modal orthologic $\mathsf{O} \oplus \{\mu(\phi) \vdash \mu(\psi) \mid \phi \vdash_\mathsf{L} \psi\}$. Does this define an isomorphism between the lattice of superfundamental logics and the lattice of orthomodal logics extending $\mathsf{O}$?
\end{question}

We leave these questions open for now, and we limit ourselves to pointing out a disanalogy between the intuitionistic and fundamental case. The semantic proof that the G\"odel-McKinsey-Tarski translation is faithful relies on the fact that for any Heyting algebra $H$, there is a smallest \textbf{S4}-algebra $(B, \Box)$ such that $H$ embeds as a lattice into $B$ via a map $e$ such that $e(a \to_H b) = \Box(e(a) \to_B e(b))$. In fact, when $H$ is realized as the algebra of downsets of some poset $(P,\leq)$, then $B$ is the complex algebra of $(P, \leq)$, where $\leq$ is viewed as a reflexive and transitive accessibility relation. The fact that any Heyting algebra can be mapped to an \textbf{S4}-algebra in such a canonical way plays a significant role in the abstract machinery of the Blok-Esakia theorem (\citealt{almeida2024polyatomic}). In the fundamental setting, let us call an \emph{\textbf{OS4}-embedding} a lattice embedding $e: L \to O$ from a fundamental lattice $(L, \neg_L)$ into an \textbf{OS4}-lattice $(O, \neg_O, \Box_O)$ such that $\Box_O e(a) = e(a)$ and $e(\neg_L a) = \Box_O \neg_O e(a)$ for any $a \in L$. We will present an example of a fundamental lattice $(L, \neg_L)$ such that $(L,\neg_L)$ \textbf{OS4}-embeds minimally into two distinct, mutually incomparable \textbf{OS4}-lattices. This suggests that the correspondence between fundamental lattices and \textbf{OS4}-lattices may not be as well behaved as the correspondence between Heyting algebras and \textbf{S4}-algebras.

To see this, consider the lattice $\mathsf{F4}$ depicted in Figure~\ref{FunLat}. 
  
   \begin{figure}[H]
  \begin{center}
  \begin{tikzpicture}[->,>={stealth'[length=6pt]},shorten >=1pt,shorten <=1pt, auto,node
  distance=2cm,semithick,every loop/.style={<-,shorten <=1pt}]
  \tikzstyle{every state}=[fill=gray!20,draw=none,text=black]

  \node  (1)  at (1.5,3)   {{$1$}};
  \node  (a)  at (0,1.5)   {{$a$}};
  \node  (b)  at (3,1.5)   {{$b$}};
  \node  (0)  at (1.5,0)   {{$0$}};

  \path (1) edge[-] node {{}} (a);
  \path (1) edge[-] node {{}} (b);
  \path (a) edge[-] node {{}} (0);
  \path (b) edge[-] node {{}} (0);

  \path (a) edge[->, bend right=20, dashed] node {{}} (0);
  \path (b) edge[->, bend left=20,  dashed] node {{}} (0);
  \path (0) edge[->, bend right=12, dashed] node {{}} (1);
  \path (1) edge[->, bend right=12, dashed] node {{}} (0);
  \end{tikzpicture}
  \end{center}
  \caption{A fundamental lattice, where dashed arrows represent the $\neg$ operation.}\label{FunLat}
  \end{figure}
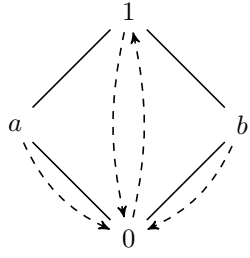

  \noindent Note first that if $(O,\neg_O,\Box_O)$ is an \textbf{OS4}-lattice such that there is an \textbf{OS4}-embedding $e: \mathsf{F4} \to O$, we must have that $\neg_O e(a)$ and $\neg_O e(b)$ are not fixpoints of $\Box_O$,  for otherwise we would have $\neg_Oe(a) = \neg_O e(b) = 0$, which is impossible in an ortholattice. This means that $O$ must have at least six elements. However, the two \textbf{OS4}-lattices depicted in Figure~\ref{2OS4Fig} are $6$-element ortholattices into which $\mathsf{F4}$ \textbf{OS4}-embeds. 

 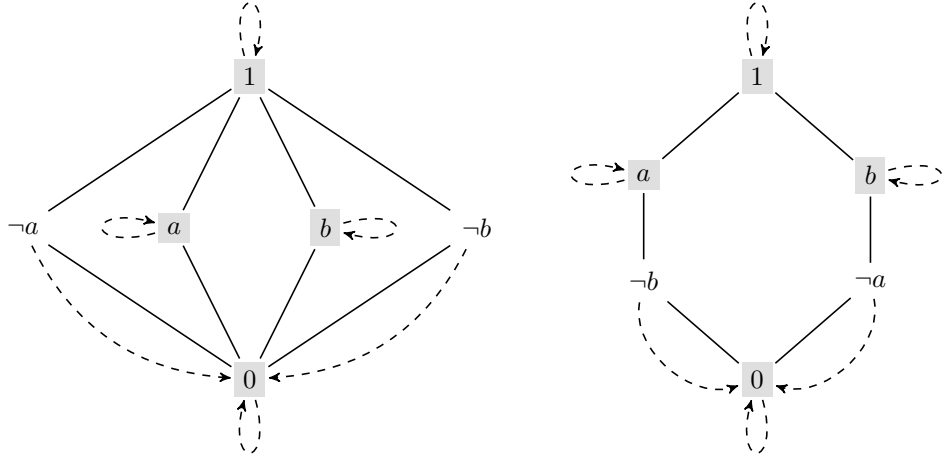
\begin{figure}[H]   
  \begin{center}
  \begin{tikzpicture}[->,>={stealth'[length=6pt]},shorten >=1pt,shorten <=1pt, auto,node
  distance=2cm,semithick,every loop/.style={<-,shorten <=1pt}]
  \tikzstyle{every state}=[fill=gray!20,draw=none,text=black]
  
  \node[fill=gray!25] (1)  at (3,4) {{$1$}};
  \node               (n3) at (0,2) {{$\neg a$}};
  \node[fill=gray!25] (a)  at (2,2) {{$a$}};
  \node[fill=gray!25] (b)  at (4,2) {{$b$}};
  \node               (n5) at (6,2) {{$\neg b$}};
  \node[fill=gray!25] (0)  at (3,0) {{$0$}};
  
  \path (1) edge[-] node {{}} (n3);
  \path (1) edge[-] node {{}} (a);
  \path (1) edge[-] node {{}} (b);
  \path (1) edge[-] node {{}} (n5);
  \path (n3) edge[-] node {{}} (0);
  \path (a)  edge[-] node {{}} (0);
  \path (b)  edge[-] node {{}} (0);
  \path (n5) edge[-] node {{}} (0);
  
  \path (a)  edge[loop left,  min distance=1.0cm, dashed] node {{}} (a);
  \path (b)  edge[loop right, min distance=1.0cm, dashed] node {{}} (b);
  \path (n3) edge[->, bend right, dashed] node {{}} (0);
  \path (n5) edge[->, bend left,  dashed] node {{}} (0);
  \path (0)  edge[loop below, min distance=1.0cm, dashed] node {{}} (0);
  \path (1)  edge[loop above, min distance=1.0cm, dashed] node {{}} (1);
  \end{tikzpicture}\qquad \begin{tikzpicture}[->,>={stealth'[length=6pt]},shorten >=1pt,shorten <=1pt, auto,node
  distance=2cm,semithick,every loop/.style={<-,shorten <=1pt}]
  \tikzstyle{every state}=[fill=gray!20,draw=none,text=black]
  
  \node[fill=gray!25] (1)  at (1.5,4)   {{$1$}};
  \node[fill=gray!25] (a)  at (0,2.7)   {{$a$}};
  \node[fill=gray!25] (b)  at (3,2.7)   {{$b$}};
  \node               (nb) at (0,1.3)   {{$\neg b$}};
  \node               (na) at (3,1.3)   {{$\neg a$}};
  \node[fill=gray!25] (0)  at (1.5,0)    {{$0$}};
  
  \path (1)  edge[-] node {{}} (a);
  \path (a)  edge[-] node {{}} (nb);
  \path (nb) edge[-] node {{}} (0);
  \path (1)  edge[-] node {{}} (b);
  \path (b)  edge[-] node {{}} (na);
  \path (na) edge[-] node {{}} (0);
  
  \path (a)  edge[loop left,  min distance=1.0cm, dashed] node {{}} (a);
  \path (b)  edge[loop right, min distance=1.0cm, dashed] node {{}} (b);
  \path (nb) edge[->, bend right=60, dashed] node {{}} (0);
  \path (na) edge[->, bend left=60,  dashed] node {{}} (0);
  \path (0)  edge[loop below, min distance=1.0cm, dashed] node {{}} (0);
  \path (1)  edge[loop above, min distance=1.0cm, dashed] node {{}} (1);
  \end{tikzpicture}
  \end{center}
  \caption{Two $\mathbf{OS4}$-lattices, where dashed arrows represent
  the $\Box$ operation, into which $\mathsf{F4}$ \textbf{OS4}-embeds. The shaded nodes are the $\Box$-fixpoints.}\label{2OS4Fig}
  \end{figure}

  \noindent Clearly, neither of the ortholattices in Figure~\ref{2OS4Fig} embeds into the other. This shows that there is no ``smallest'' \textbf{OS4}-lattice into which $\mathsf{F4}$ \textbf{OS4}-embeds.

\subsection{Relational Semantics for $\mathsf{FN4}$}\label{RelationalFN4}

In Section~\ref{FN4Section}, inspired by Mill's idea of necessity as \textit{unconditionalness}, we added $\Box$-Reiteration to fundamental modal logic, allowing $\Box$ formulas to be reiterated into subproofs used for $\vee$E and $\neg$I. Algebraically, this corresponds to allowing a special case of distributivity, namely $(a\vee b)\wedge \Box c\leq (a\wedge\Box c)\vee (b\wedge\Box c)$, and a special case of pseudocomplementation, namely that if $a\wedge\Box c\leq \neg b$, then $b \wedge\Box c\leq\neg a$. Although a more concrete relational semantics for $\mathsf{FN4}$ was not needed for our results in Section~\ref{FN4Section}, there is such a semantics based on the following first-order definable class of frames.

\begin{definition}\label{FNframes} An \textit{\textbf{FN4}-frame} is a fundamental modal frame $(X,\op, \shortdashleftarrow)$ in which $\shortdashleftarrow$ is reflexive and transitive and the following conditions hold:
\begin{enumerate}
\item direct enrichment: if $x\po y$, then there is a prerefinement $y^+$ of $y$ such that $x\po y^+\shortdashleftarrow x$;
\item indirect enrichment: if $x\po z$, then there is a postrefinement $z^-$ of $z$ such that $x\po z^-$ and for any $y\po z^-$, there is a prerefinement $y^+$ of $y$ such that $z^-\op y^+\shortdashleftarrow x$.\footnote{\label{WeakerIE}For the soundness of $\mathbf{FN4}$, we can weaken this condition so that $z^-$ is not required to postrefine $z$, and $z^-\op y^+$ is replaced by $z\op y^+$. But our completeness proof also allows us to impose the stronger condition, which is more informative about $z^-$.}
\end{enumerate}
\end{definition}
\noindent For illustrations of these conditions, see Figure \ref{FNFig}.

We can understand the relationship between these conditions and our proof system intuitively as follows, by thinking of opening a subproof as moving us between states in a frame:
\begin{itemize}
\item If we are in a state $x$ and open a subproof for $\neg$I, this makes us consider an arbitrary state $y\op x$ where the assumption of the subproof holds. Now \textit{direct enrichment} says that we should be able to prerefine $y$ to a state $y^+$ to which the necessity facts from $x$ transfer, in line with $\Box$-Reiteration into $\neg$I-subproofs. Then if we can deduce at $y^+$ the negation of something true at $x$, which contradicts $x\po y^+$ given the pseudosymmetry of $\op$, this shows that there was never really any state $y\op x$ where the assumption of the subproof holds, which establishes the negation of that assumption back at $x$.
\item If we are in state $x$ and aim to prove a conclusion via proof by cases, then an arbitrary $z\op x$ poses a challenge for us to find a $v\po z$ satisfying the conclusion of the proof by cases. \textit{Indirect enrichment} says that we can postrefine $z$ to a $z^-\op x$ such that when we open a subproof for one of the cases, which takes us to a $y\po z^-$ where the assumption of the subproof holds, we can prerefine $y$ to a state $y^+\po z^-$ to which the necessity facts from $x$ transfer, in line with $\Box$-Reiteration into $\vee$E-subproofs. Then since $z^-$ postrefines $z$, $y^+\po z^-$ implies $y^+\po z$, so we can take $v=y^+$ in each case to meet the challenge.\end{itemize}

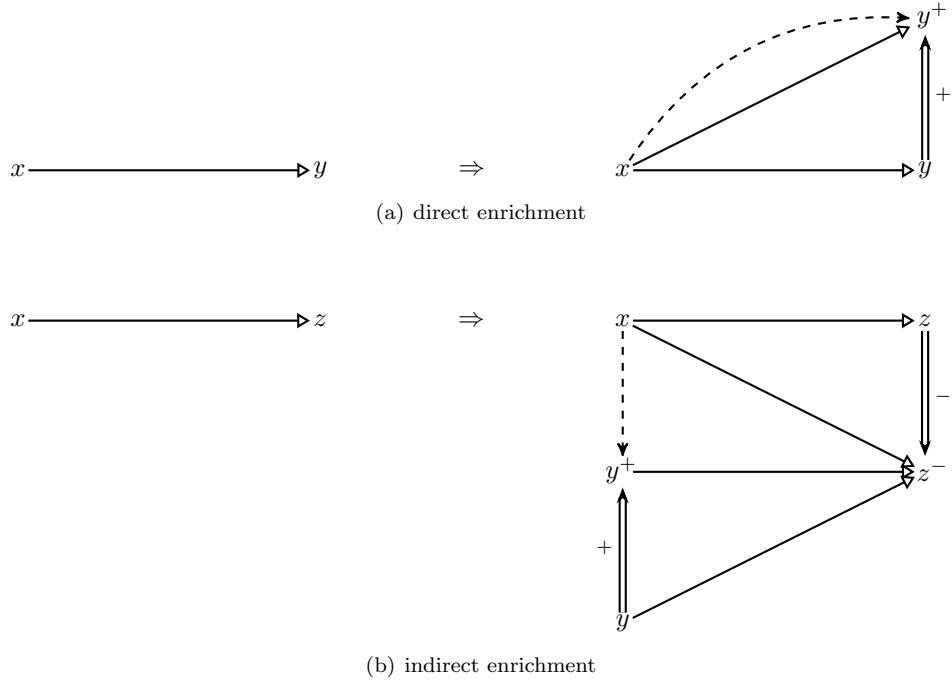
\begin{figure}[h]
\begin{center}
\subfigure[direct enrichment]{
\begin{tikzpicture}[->,>={stealth'[length=6pt]},shorten >=1pt,shorten <=1pt, auto,node
distance=2cm,thick,every loop/.style={<-,shorten <=1pt}]
\tikzstyle{every state}=[fill=gray!20,draw=none,text=black]

\node[label=center:$x$,inner sep=0pt,minimum size=.175cm] at (0,0) (x) {};
\node[label=center:$y$,inner sep=0pt,minimum size=.175cm] at (4,0) (y) {};
\path[-{Triangle[open]},draw,thick] (x) to node{{}} (y);

\node at (6,0) {{\textit{$\Rightarrow$}}};

\node[label=center:$x$,inner sep=0pt,minimum size=.175cm] at (8,0) (x2) {};
\node[label=center:$y$,inner sep=0pt,minimum size=.175cm] at (12,0) (y2) {};
\node[inner sep=0pt,minimum size=.175cm] at (12,2) (yplus) {};

\path[-{Triangle[open]},draw,thick] (x2) to node{{}} (y2);
\path[-{Triangle[open]},draw,thick,shorten >=3pt] (x2) to node{{}} (yplus);
\path (x2) edge[dashed,->,shorten >=3pt,bend left=30] node {{}} (yplus);
\path (y2) edge[double,double distance=1.5pt,-{Stealth[length=6pt]},shorten >=3pt]
  node[right] {{$\scriptstyle +$}} (yplus);

\node[anchor=center,inner sep=0pt] at (12.1,2.05) {$y^+$};

\end{tikzpicture}
}

\vspace{.25in}

\subfigure[indirect enrichment]{
\begin{tikzpicture}[->,>={stealth'[length=6pt]},shorten >=1pt,shorten <=1pt, auto,node
distance=2cm,thick,every loop/.style={<-,shorten <=1pt}]
\tikzstyle{every state}=[fill=gray!20,draw=none,text=black]

\node[label=center:$x$,inner sep=0pt,minimum size=.175cm] at (-2,4) (x) {};
\node[label=center:$z$,inner sep=0pt,minimum size=.175cm] at (2,4) (z) {};
\path[-{Triangle[open]},draw,thick] (x) to node{{}} (z);

\node at (4,4) {{\textit{$\Rightarrow$}}};

\node[label=center:$x$,inner sep=0pt,minimum size=.175cm] at (6,4) (x2) {};
\node[label=center:$z$,inner sep=0pt,minimum size=.175cm] at (10,4) (z2) {};
\node[inner sep=0pt,minimum size=.175cm] at (10,2) (zminus) {};
\node[inner sep=0pt,minimum size=.175cm] at (6,2) (yplus) {};
\node[label=center:$y$,inner sep=0pt,minimum size=.175cm] at (6,0) (y) {};

\node[anchor=center,inner sep=0pt] at (10.1,2.02) {$z^-$};
\node[anchor=center,inner sep=0pt] at (5.98,2.02) {$y^+$};

\path[-{Triangle[open]},draw,thick] (x2) to node{{}} (z2);
\path[-{Triangle[open]},draw,thick] (x2) to node{{}} (zminus);
\path[-{Triangle[open]},draw,thick] (y) to node{{}} (zminus);
\path[-{Triangle[open]},draw,thick] (yplus) to node{{}} (zminus);
\path (x2) edge[dashed,->,shorten >=3pt] node {{}} (yplus);

\path (y) edge[double,double distance=1.5pt,-{Stealth[length=6pt]},shorten >=3pt]
  node[left] {{$\scriptstyle +$}} (yplus);
\path (z2) edge[double,double distance=1.5pt,-{Stealth[length=6pt]},shorten >=3pt]
  node[right] {{$\scriptstyle -$}} (zminus);

\end{tikzpicture}
}

\end{center}
\caption{Illustration of the frame conditions in Definition \ref{FNframes}. As before, a solid line from $u$ to $v$ indicates $u\po v$ and a dashed line from $u$ to $v$ indicates $u\shortdashrightarrow v$. Here a double line from $u$ to $v$ marked with $+$ (resp.~$-$) indicates that $v$ prerefines (resp.~postrefines) $u$.}
\label{FNFig}
\end{figure}

In Appendix~\ref{FN4CompAppendix}, we prove the following theorem.

\begin{theorem} $\mathsf{FN4}$ is sound and complete with respect to the class of $\mathbf{FN4}$-frames.
\end{theorem}

\noindent It is possible to prove analogous theorems for extensions of $\mathsf{FN4}$ with stronger forms of distributivity, but we will leave the project of charting the semantic landscape of extensions of $\mathsf{FN4}$ for future work.

\section{Conclusion}\label{Conclusion}

We have attempted to shed light on fundamental logic and its semantics through the lens of modal logic. On the semantic side, our results provide different perspectives on the \textit{openness} relation in fundamental frames. We started with the characterization of openness in terms of the notions of \textit{accepting} or \textit{rejecting} a proposition: a state $y$ is \textit{open to} a state $x$ if $y$ does not \textit{reject} any proposition that $x$ \textit{accepts}. This characterization explains how openness can fail to be symmetric: just because we do not reject any proposition that you accept, it does not follow that you do not reject any proposition that we accept. But this may leave one wondering how non-symmetric openness relates to symmetric notions of \textit{compatibility}. As some of our results establish (\cref{embthm,embthm2}), any fundamental lattice can be represented by a reflexive fundamental frame whose openness relation can be factored as the composition of a compatibility relation, which is reflexive and symmetric, and a \emph{positive information increase} relation, which is reflexive and transitive. In the case of balanced frames, this means that $y$ is open to $x$ iff (i) $y$ is compatible with some state that accepts every proposition that $x$ accepts and (ii) $x$ is compatible with some state that rejects every proposition that $y$ rejects. In strongly factoring frames, this means that $y$ is open to $x$ iff $y$ is compatible with a single state $z$ that accepts every proposition that $x$ accepts and does not reject any proposition that $x$ does not reject. In both cases, this gives a clearer intuition regarding the interplay between openness and its ``symmetrization'' into a compatibility relation.

On the purely logical side, the GMT translation of fundamental logic into $\mathsf{OS4}$ shows that the orthologician can reinterpret the fundamental logician modally in precisely the same way as the classical logician can reinterpret the intuitionistic logician modally. Similarly, the Goldblatt translation of fundamental logic into intuitionistic $\mathsf{KTB}$ shows that the intuitionist can reinterpret the fundamental logician modally in precisely the same way as the classical logician can reinterpret the orthologician modally. This means that via the translations $\mu$ and $\gamma$, we can now extend the ``fundamental diamond'' of Figure~\ref{fundiamond} to the ``fundamental lotus'' of Figure~\ref{funlotus}, which connects fundamental logic, (modal) orthologic, intuitionistic (modal) logic, and classical (modal) logic via additional rules and translations.

\begin{figure}[h]
    \centering
\[
\begin{tikzcd}[
    cells={
        nodes={
            minimum width=2.9em,
            inner sep=2pt
        }
    },
    arrows={
        ->,
        >={stealth'[length=6pt]},
        semithick,
        shorten >=1pt,
        shorten <=1pt
    },
    every label/.append style={fill=white,inner sep=2pt}
]
	&&& {\mathsf{CPC}} &&& \\
	{\mathsf{KTB}} &&&&&& {\mathsf{S4}} \\
	&& {\mathsf{O}} && {\mathsf{IPC}^-} \\
	& {\mathsf{FSTB}} &&&& {\mathsf{OS4}} \\
	&&& {\mathsf{F}}
	\arrow["\text{Reit}"{description}, from=3-3, to=1-4]
	\arrow["\text{Goldblatt}"{description}, dashed, from=3-3, to=2-1]
	\arrow["\text{RAA}"{description}, from=3-5, to=1-4]
	\arrow["\text{GMT}"{description}, dashed, from=3-5, to=2-7]
	\arrow["\text{RAA}"{description}, from=4-2, to=2-1]
	\arrow["\text{Reit}"{description}, from=4-6, to=2-7]
	\arrow["\text{RAA}"{description}, from=5-4, to=3-3]
	\arrow["\text{Reit}"{description}, from=5-4, to=3-5]
\arrow["{\scalebox{1.35}{$\gamma$}}"{description}, dashed, from=5-4, to=4-2]
\arrow["{\scalebox{1.35}{$\mu$}}"{description}, dashed, from=5-4, to=4-6]
\end{tikzcd}
\]
    \caption{The ``fundamental lotus'', connecting fundamental logic to various propositional and modal logics. Dashed arrows are full and faithful translations, while solid arrows are additional proof-theoretic rules.}
    \label{funlotus}
\end{figure}
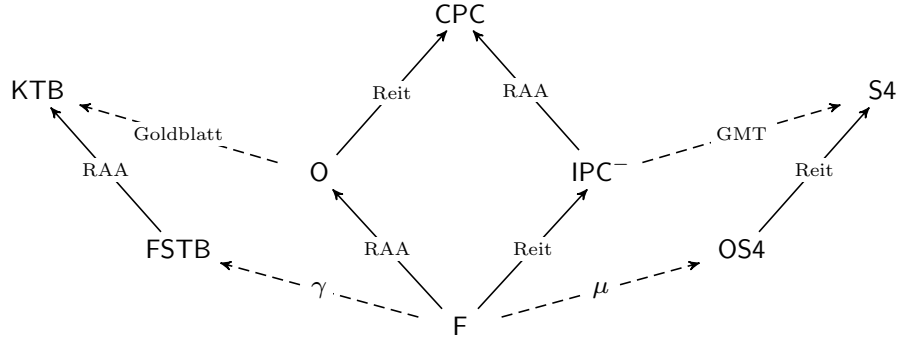

Finally, the GMT translation of intuitionistic logic into fundamental logic with \textit{necessity} shows that the fundamental logician can regard the intuitionist's reasoning as correct in the realm of purely necessary propositions. Since \textit{mathematics} is plausibly a realm in which propositions are necessarily true if true, the fundamental logician can regard intuitionistic mathematical reasoning as correct. As a consequence, the fundamental logician can also regard classical mathematical reasoning as correct via a double negation reinterpretation of its claims.

\subsection*{Acknowledgements}

We are grateful for feedback on this work from audiences at the Mathematics, Philosophy, and Physics Seminar at Chapman University in October 2024, the French Philosophy of Mathematics Workshop at Archive Henri Poincaré in Nancy, France in October 2024, the Nordic Online Logic Seminar in November 2024, the LLAMA Seminar of the ILLC in March 2025, the Logica Conference in May 2025, the KCL-SNS Workshop in Logic in June 2025, the Peking University Logic Seminar in September 2025, the Sixth Taiwan Philosophical Logic Colloquium (TPLC 2025) in October 2025, the Workshop on Truthmakers, Possibilities, and Information States at the Institute of Philosophy of the Czech Academy of Sciences, Prague in November 2025, and the Workshop on Logic, Categories and Decision-Making: Algebraic and Proof-Theoretic Methods (LoCAD 2025), Amsterdam in December 2025. We are also grateful to Rodrigo Almeida, Johan van Benthem, Nick Bezhanishvili, and Ahmee Christensen for comments on earlier versions of this paper. A previous version of the paper posted online posed a question in \S~\ref{RelationalFN4} about relational semantics for $\mathbf{FN4}$ that this version resolves. Yifeng Ding informed us that GPT 5.6 Sol Ultra had found a way to answer the question, without providing its answer to us, which motivated us to  spend more time on the question to answer it in the way we do here (in particular, the previous version stated \textit{indirect enrichment} in \S~\ref{RelationalFN4} in a stronger form, which validated a stronger logic than $\mathbf{FN4}$, whereas the current form is appropriate for $\mathbf{FN4}$ itself).

\appendix

\section{Appendix}

\subsection{Formal Proofs for Intuitionistic Logic}\label{AppendixIPC}

For each finite set $R$ of formulas, the set of \textit{proofs given $R$} for $\mathsf{IPC}^{-}$ is the smallest set containing for each formula $\varphi$ the sequence  $\langle \varphi\rangle$ and satisfying the following closure conditions for $1\leq i,j\leq n$:
\begin{itemize}
\item If $\langle \sigma_1,\dots,\sigma_n\rangle$ is a proof given $R$ and $\varphi\in R$, then $\langle \sigma_1,\dots,\sigma_n,\varphi\rangle$ is a proof given $R$ (Reiteration).
\item If $\langle \sigma_1,\dots,\sigma_n\rangle$ is a proof given $R$, then so is $\langle \sigma_1,\dots,\sigma_n,\top\rangle$ ($\top$I).
\item If $\langle \sigma_1,\dots,\sigma_n\rangle$ is a  proof given $R$ and $\sigma_i,\sigma_j$  are formulas, then $\langle \sigma_1,\dots,\sigma_n,\sigma_i\wedge\sigma_j\rangle$ is a proof given $R$ ($\wedge$I).
\item If $\langle \sigma_1,\dots,\sigma_n\rangle$ is a  proof given $R$ and $\sigma_i$ is a formula of the form $\varphi\wedge\psi$, then $\langle \sigma_1,\dots,\sigma_n,\varphi\rangle$ and $\langle \sigma_1,\dots,\sigma_n,\psi\rangle$ are proofs given $R$ ($\wedge$E).
\item If $\langle \sigma_1,\dots,\sigma_n\rangle$ is a  proof given $R$ and  $\sigma_i$ is a formula, then for any formula $\varphi$,  both $\langle \sigma_1,\dots,\sigma_n,\sigma_i\vee\varphi\rangle$ and $\langle \sigma_1,\dots,\sigma_n,\varphi\vee \sigma_i\rangle$ are proofs given $R$ ($\vee$I).
\item If $\langle \sigma_1,\dots,\sigma_n\rangle$ is a  proof given $R$, $\sigma_i$ is a formula of the form $\varphi\vee\psi$, $\tau_1$ is a proof given some $R_1\subseteq R\cup\{\sigma_j\mid \sigma_j\mbox{ a formula}\}$ beginning with $\varphi$ and ending with $\chi$, and $\tau_2$ is a proof given some $R_2\subseteq R\cup\{\sigma_j\mid \sigma_j\mbox{ a formula}\}$  beginning with $\psi$ and ending with $\chi$, then $\langle \sigma_1,\dots,\sigma_n, \tau_1,\tau_2,\chi\rangle$ is a proof given $R$ ($\vee $E).
\item If $\langle \sigma_1,\dots,\sigma_n\rangle$ is a  proof given $R$, $\sigma_i$ is a formula  $\psi$, and $\tau$ is a proof given some $R'\subseteq R\cup\{\sigma_j\mid \sigma_j\mbox{ a formula}\}$ beginning with $\varphi$ and ending with $\neg\psi$, then $\langle \sigma_1,\dots,\sigma_n,\tau,\neg\varphi\rangle$ is a proof given $R$ ($\neg$I).
\item If $\langle \sigma_1,\dots,\sigma_n\rangle$ is a  proof given $R$ and $\sigma_i$ and $\sigma_j$ are formulas of the form $\varphi$ and $\neg\varphi$, respectively, then for any formula $\psi$, $\langle \sigma_1,\dots,\sigma_n,\psi\rangle$ is a proof given $R$ ($\neg$E). 
\end{itemize}

$\mathsf{IPC}$ extends $\mathsf{IPC}^-$ with the following introduction and elimination clauses for $\to$:

\begin{itemize}
\item If $\langle \sigma_1,\dots,\sigma_n\rangle$ is a  proof given $R$ and $\tau$ is a proof given some $R'\subseteq R\cup \{\sigma_j\mid\sigma_j\mbox{ a formula}\}$ beginning with $\varphi$ and ending with $\psi$, then $\langle \sigma_1,\dots,\sigma_n,\varphi\to\psi\rangle$ is a proof given $R$ ($\to$I).
\item If $\langle \sigma_1,\dots,\sigma_n\rangle$ is a  proof given $R$ and $\sigma_i,\sigma_j$  are formulas of the form $\varphi$ and $\varphi\to\psi$, respectively, then $\langle \sigma_1,\dots,\sigma_n,\psi\rangle$ is a proof given $R$ ($\to$E).
\end{itemize}

\subsection{Formal Proofs for First-Order Logics}\label{AppendixQuant}

To define the set of proofs for first-order fundamental logic, $\mathsf{QF}$, we add the following clauses to those for propositional fundamental logic:

\begin{itemize}
\item if $\langle \sigma_1,\dots,\sigma_n\rangle$ is a proof, $\sigma_i$ is a formula $\varphi$, and $v$ does not occur free in $\sigma_1$, then $\langle \sigma_1,\dots,\sigma_n,\forall v\varphi\rangle$ is a proof ($\forall$I);
\item if $\langle \sigma_1,\dots,\sigma_n\rangle$ is a proof, $\sigma_i$ is a formula of the form $\forall v\varphi$, and $u$ is substitutable for $v$ in $\varphi$, then $\langle \sigma_1,\dots,\sigma_n,\varphi^v_u\rangle$ is a proof ($\forall$E);
\item if $\langle \sigma_1,\dots,\sigma_n\rangle$ is a proof, $\sigma_i$ is a formula of the form $\varphi^v_u$, and $u$ is substitutable for $v$ in $\varphi$, then $\langle \sigma_1,\dots,\sigma_n,\exists v\varphi\rangle$ is a proof ($\exists$I);
\item if $\langle \sigma_1,\dots,\sigma_n\rangle$ is a proof, $\sigma_i$ is a formula of the form $\exists v\varphi$, $\tau$ is a proof beginning with $\varphi$ and ending with $\psi$, and $v$ does not occur free in $\psi$, then $\langle \sigma_1,\dots,\sigma_n,\tau,\psi\rangle$ is a proof ($\exists$E).
\end{itemize}

For intuitionistic first-order logic, $\mathsf{IQC}$,  we adjust the $\forall$E and $\exists$I clauses above by replacing `proof' by `proof given $R$' and modifying the $\forall$I and $\exists$E clauses so that the relevant quantified variable does not occur free in any reiterable formula:
\begin{itemize}
\item if $\langle \sigma_1,\dots,\sigma_n\rangle$ is a proof given $R$, $\sigma_i$ is a formula $\varphi$, and $v$ does not occur free in $\sigma_1$ or in any formula in $R$, then $\langle \sigma_1,\dots,\sigma_n,\forall v\varphi\rangle$ is a proof given $R$ ($\forall$I);
\item if $\langle \sigma_1,\dots,\sigma_n\rangle$ is a proof given $R$, $\sigma_i$ is a formula of the form $\exists v\varphi$, $\tau$ is a proof given some $R'\subseteq R\cup\{\sigma_j\mid \sigma_j\mbox{ a formula}\}$ beginning with $\varphi$ and ending with $\psi$, and $v$ does not occur free in $\psi$ or in any formula in $R'$, then $\langle \sigma_1,\dots,\sigma_n,\tau,\psi\rangle$ is a proof given $R$ ($\exists$E).
\end{itemize}

\subsection{Formal Proofs for Modal Logic}\label{AppendixModal}

For the minimal modal logic $\mathsf{FK}$ extending fundamental logic, we define by a simultaneous induction the set of \textit{proofs} and the set of \textit{$\Box$-proofs given $B$}, for any finite set $B$ of formulas, using all the clauses in the definition of proofs for fundamental logic, plus the following (see Figure~\ref{BoxRules}):
\begin{itemize}
\item If $\langle\sigma_1,\dots,\sigma_n\rangle$ is a proof with $\sigma_1=\top$, then $\big\langle \Box, \langle \sigma_1,\dots ,\sigma_n\rangle \big\rangle$ is a $\Box$-proof given $B$.
\item If $\big\langle \Box,\langle \sigma_1,\dots,\sigma_n\rangle\big\rangle$ is a $\Box$-proof given $B$   and $\Box\varphi\in B$, then $\big\langle \Box, \langle \sigma_1,\dots,\sigma_n,\varphi\rangle\big\rangle$ is a $\Box$-proof given $B$ ($\Box$E). 
\item If $\langle \sigma_1,\dots,\sigma_n\rangle$ is a proof and $\tau=\big\langle \Box,\langle \tau_1,\dots,\tau_m\rangle\big\rangle$ is a $\Box$-proof given some $B\subseteq \{\sigma_j\mid \sigma_j\mbox{ a formula}\}$ where $\tau_m$ is a formula, then $\langle \sigma_1,\dots,\sigma_n,\tau, \Box \tau_m\rangle$ is a proof ($\Box$I).
\end{itemize}
We also extend the introduction and elimination rules in proofs to $\Box$-proofs given $B$:
\begin{itemize}
\item If $\big\langle \Box,\langle \sigma_1,\dots,\sigma_n\rangle\big\rangle$ is a $\Box$-proof given $B$ and $\sigma_i,\sigma_j$  are formulas, then $\big\langle \Box,\langle \sigma_1,\dots,\sigma_n,\sigma_i\wedge\sigma_j\rangle \big\rangle$ is a $\Box$-proof given $B$ ($\wedge$I).
\item If $\big\langle \Box,\langle \sigma_1,\dots,\sigma_n\rangle\big\rangle$ is a $\Box$-proof given $B$ and $\sigma_i$ is a formula of the form $\varphi\wedge\psi$, then $\big\langle \Box,\langle \sigma_1,\dots,\sigma_n,\varphi\rangle\big\rangle$ and $\big\langle \Box,\langle \sigma_1,\dots,\sigma_n,\psi\rangle\big\rangle$ are $\Box$-proofs given $B$ ($\wedge$E).
\item If $\big\langle \Box,\langle \sigma_1,\dots,\sigma_n\rangle\big\rangle$ is a  $\Box$-proof given $B$ and  $\sigma_i$ is a formula, then for any formula $\varphi$, $\big\langle \Box,\langle \sigma_1,\dots,\sigma_n,\sigma_i\vee\varphi\rangle\big\rangle$ and $\big\langle \Box,\langle \sigma_1,\dots,\sigma_n,\varphi\vee \sigma_i\rangle\big\rangle$ are both $\Box$-proofs given $B$ ($\vee$I).
\item If $\big\langle \Box,\langle \sigma_1,\dots,\sigma_n\rangle\big\rangle$ is a $\Box$-proof given $B$, $\sigma_i$ is a formula of the form $\varphi\vee\psi$, $\tau_1$ is a proof beginning with $\varphi$ and ending with $\chi$, and $\tau_2$ is a proof beginning with $\psi$ and ending with $\chi$, then $\big\langle \Box,\langle \sigma_1,\dots,\sigma_n, \tau_1,\tau_2,\chi\rangle\big\rangle$ is a $\Box$-proof given $B$ ($\vee $E).
\item If $\big\langle \Box,\langle \sigma_1,\dots,\sigma_n\rangle\big\rangle$ is a $\Box$-proof given $B$, $\sigma_i$ is a formula  $\psi$, and $\tau$ is a proof beginning with $\varphi$ and ending with $\neg\psi$, then $\big\langle \Box,\langle \sigma_1,\dots,\sigma_n,\tau,\neg\varphi\rangle\big\rangle$ is a $\Box$-proof given $B$ ($\neg$I).
\item If $\big\langle \Box,\langle \sigma_1,\dots,\sigma_n\rangle\big\rangle$ is a $\Box$-proof given $B$ and $\sigma_i$ and $\sigma_j$ are formulas of the form $\varphi$ and $\neg\varphi$, respectively, then for any formula $\psi$, $\big\langle \Box,\langle \sigma_1,\dots,\sigma_n,\psi\rangle\big\rangle$ is a $\Box$-proof given $B$ ($\neg$E). 
\item If $\big\langle \Box,\langle \sigma_1,\dots,\sigma_n\rangle\big\rangle$ is a $\Box$-proof given $B$ and $\tau=\big\langle \Box,\langle \tau_1,\dots,\tau_m\rangle\big\rangle$ is a $\Box$-proof given some $B'\subseteq \{\sigma_j\mid \sigma_j\mbox{ a formula}\}$, then $\big\langle \Box,\langle \sigma_1,\dots,\sigma_n,\tau, \Box \tau_m\rangle\big\rangle$ is a $\Box$-proof given $B$ ($\Box$I).
\end{itemize}

\noindent For the T and $4$ rules, we add the following clauses:
\begin{itemize}
\item If $\langle \sigma_1,\dots,\sigma_n\rangle$ is a proof and $\sigma_i$ is a formula of the form $\Box\varphi$, then $\langle \sigma_1,\dots,\sigma_n,\varphi\rangle$ is a proof (T).
\item If $\langle \sigma_1,\dots,\sigma_n\rangle$ is a proof and $\sigma_i$ is a formula of the form $\Box\varphi$, then $\langle \sigma_1,\dots,\sigma_n,\Box\Box\varphi\rangle$ is a proof (4).
\item If $\big\langle \Box,\langle \sigma_1,\dots,\sigma_n\rangle\big\rangle$ is a $\Box$-proof given $B$ and $\sigma_i$ is a formula of the form $\Box\varphi$, then $\big\langle \Box,\langle \sigma_1,\dots,\sigma_n,\varphi\rangle\big\rangle$ is a $\Box$-proof given $B$ (T).
\item If $\big\langle \Box,\langle \sigma_1,\dots,\sigma_n\rangle\big\rangle$ is a $\Box$-proof given $B$ and $\sigma_i$ is a formula of the form $\Box\varphi$, then $\big\langle \Box,\langle \sigma_1,\dots,\sigma_n,\Box\Box\varphi\rangle \big\rangle$ is a $\Box$-proof given $B$ (4).
\end{itemize}

When defining the quantified orthomodal logic $\mathsf{QOS4}$ in Section~\ref{FirstOrderExt1}, we extend the RAA rule and the introduction and elimination rules for the quantifiers to $\Box$-proofs given $B$ as follows:

\begin{itemize}
\item If $\big\langle \Box,\langle \sigma_1,\dots,\sigma_n\rangle\big\rangle$ is a $\Box$-proof given $B$, $\sigma_i$ is a formula  $\psi$, and $\tau$ is a proof beginning with $\neg\varphi$ and ending with $\neg\psi$, then $\big\langle \Box,\langle \sigma_1,\dots,\sigma_n,\tau,\varphi\rangle\big\rangle$ is a $\Box$-proof given $B$ (RAA).
\item If $\big\langle \Box,\langle \sigma_1,\dots,\sigma_n\rangle\big\rangle$ is a $\Box$-proof given $B$, $\sigma_i$ is a formula $\varphi$, and $v$ does not occur free in $\sigma_1$ or in any formula in $B$, then $\big\langle \Box,\langle \sigma_1,\dots,\sigma_n,\forall v\varphi\rangle\big\rangle$ is a $\Box$-proof given $B$ ($\forall$I).
\item If $\big\langle \Box,\langle \sigma_1,\dots,\sigma_n\rangle\big\rangle$ is a $\Box$-proof given $B$, $\sigma_i$ is a formula of the form $\forall v\varphi$, and $u$ is substitutable for $v$ in $\varphi$, then $\big\langle \Box,\langle \sigma_1,\dots,\sigma_n,\varphi^v_u\rangle\big\rangle$ is a $\Box$-proof given $B$ ($\forall$E).
\item If $\big\langle \Box,\langle \sigma_1,\dots,\sigma_n\rangle\big\rangle$ is a $\Box$-proof given $B$, $\sigma_i$ is a formula of the form $\varphi^v_u$, and $u$ is substitutable for $v$ in $\varphi$, then $\big\langle \Box,\langle \sigma_1,\dots,\sigma_n,\exists v\varphi\rangle\big\rangle$ is a $\Box$-proof given $B$ ($\exists$I).
\item If $\big\langle \Box,\langle \sigma_1,\dots,\sigma_n\rangle\big\rangle$ is a $\Box$-proof given $B$, $\sigma_i$ is a formula of the form $\exists v\varphi$, $\tau$ is a proof beginning with $\varphi$ and ending with $\psi$, and $v$ does not occur free in $\psi$, then $\big\langle \Box,\langle \sigma_1,\dots,\sigma_n,\tau,\psi\rangle\big\rangle$ is a $\Box$-proof given $B$ ($\exists$E).
\end{itemize}

\subsection{Formal Proofs for \textsf{FN4}}\label{AppendixFN4}

The $\Box$-Reiteration rule of $\mathsf{FN4}$ is simply the restriction of the intuitionistic Reiteration rule to boxed formulas: 
\begin{itemize}
\item If $\langle \sigma_1,\dots,\sigma_n\rangle$ is a proof given $R$ and $\Box\varphi\in R$, then $\langle \sigma_1,\dots,\sigma_n,\Box\varphi\rangle$ is a proof given $R$ ($\Box$-Reiteration).
\end{itemize}

To define $\mathsf{FN4}$, we define by a simultaneous induction the set of \textit{proofs given $R$}, as in Appendix~\ref{AppendixIPC} but with $\Box$-Reiteration in place of Reiteration, and the set of \textit{$\Box$-proofs given $B$} as in Appendix~\ref{AppendixModal} but with modifications to the clauses for $\vee$E and $\neg$I to allow $\Box$-Reiteration into $\vee$E-subproofs and $\neg$I-subproofs inside $\Box$-subproofs:
\begin{itemize}
\item If $\big\langle \Box,\langle \sigma_1,\dots,\sigma_n\rangle\big\rangle$ is a $\Box$-proof given $B$, $\sigma_i$ is a formula of the form $\varphi\vee\psi$, $\tau_1$ is a proof given some $R_1\subseteq\{\sigma_j\mid \sigma_j\mbox{ a formula}\}$ beginning with $\varphi$ and ending with $\chi$, and $\tau_2$ is a proof given some $R_2\subseteq\{\sigma_j\mid \sigma_j\mbox{ a formula}\}$ beginning with $\psi$ and ending with $\chi$, then $\big\langle \Box,\langle \sigma_1,\dots,\sigma_n, \tau_1,\tau_2,\chi\rangle\big\rangle$ is a $\Box$-proof given $B$ ($\vee $E).
\item If $\big\langle \Box,\langle \sigma_1,\dots,\sigma_n\rangle\big\rangle$ is a $\Box$-proof given $B$, $\sigma_i$ is a formula  $\psi$, and $\tau$ is a proof given some $R\subseteq\{\sigma_j\mid \sigma_j\mbox{ a formula}\}$ beginning with $\varphi$ and ending with $\neg\psi$, then $\big\langle \Box,\langle \sigma_1,\dots,\sigma_n,\tau,\neg\varphi\rangle\big\rangle$ is a $\Box$-proof given $B$ ($\neg$I).
\end{itemize}

To define $\mathsf{FC4}$, we follow the same approach but we include the $\to$I and $\to$E rules from Appendix~\ref{AppendixIPC}, also allowing their application inside $\Box$-proofs.

To define $\mathsf{QFC4}$, we follow the same approach but we define \textit{proofs given $R$} using the same $\forall$I, $\forall$E, $\exists$I, and $\exists$E clauses as for $\mathsf{IQC}$ in Appendix~\ref{AppendixQuant}, and we define \textit{$\Box$-proofs given $B$} using the same $\forall$I, $\forall$E, and $\exists$I clauses as for $\mathsf{QOS4}$ from the end of Appendix~\ref{AppendixModal}, plus the following:
\begin{itemize}
\item if $\big\langle \Box,\langle \sigma_1,\dots,\sigma_n\rangle\big\rangle$ is a $\Box$-proof given $B$, $\sigma_i$ is a formula of the form $\exists v\varphi$, $\tau$ is a proof given some $R\subseteq \{\sigma_j\mid \sigma_j\mbox{ a formula}\}$ beginning with $\varphi$ and ending with $\psi$, and $v$ does not occur free in $\psi$ or in any formula in $R$, then $\big\langle \Box,\langle \sigma_1,\dots,\sigma_n,\tau,\psi\rangle\big\rangle$ is a $\Box$-proof given $B$ ($\exists$E).
\end{itemize}

\subsection{Translation of Intuitionistic Proofs into \textsf{FN4} Proofs}\label{ProofTransAppendix}

In this appendix, we give constructive, proof-theoretic arguments showing how proofs in intuitionistic logic can be translated into proofs in fundamental logic with necessity. To do so, it will be convenient to move to a more constrained proof system for intuitionistic logic in which a given subproof may only reiterate a \textit{single} formula that occurs in the proof to which the subproof directly belongs.

\begin{definition} We define the set of $\underline{\mathsf{IPC}}^{-}$-proofs given $R$, for a set $R$ with $|R|\leq 1$, in the same way as we defined $\mathsf{IPC}^{-}$-proofs given $R$ in Appendix~\ref{AppendixIPC}, except that in the clause for $\vee$E, we require that $R_1\cup R_2\subseteq \{\sigma_j\mid \sigma_j\mbox{ a formula}\}$, and in the clause for $\neg$I, we require $R'\subseteq \{\sigma_j\mid \sigma_j\mbox{ a formula}\}$, i.e., we only allow the reiteration of a formula that occurs in the proof to which the subproof directly belongs.\end{definition}
\noindent The inductive clauses for $\underline{\mathsf{IPC}}^{-}$ are written out in full in the proof of Theorem~\ref{ProofTransform} below.

Now we need a measure of proof complexity on which we can induct, for which we choose the following.

\begin{definition} The \textit{complexity} of an $\mathsf{IPC}^{-}$-proof (resp.~$\underline{\mathsf{IPC}}^{-}$-proof) $\langle \sigma_1,\dots,\sigma_n\rangle$ is the sum of the complexities of each $\sigma_i$, where if $\sigma_i$ is a formula, its complexity is $1$.
\end{definition}

\begin{fact} For any proof $\sigma=\langle \sigma_1,\dots,\sigma_n\rangle$ and $\sigma_i$, if $\sigma_i$ is a proof, then its complexity is less than that of $\sigma$, and if $\sigma_i$ is a formula with $i<n$, then the complexity of $\langle \sigma_1,\dots,\sigma_i\rangle$ is less than that of $\sigma$.\end{fact}

Since we have defined proofs inductively as certain sequences, the following definition and lemma will be workhorses for proof-theoretic arguments.

\begin{definition} Given sequences $\langle\sigma_1,\dots,\sigma_n\rangle$ and $\langle \tau_1,\dots,\tau_m\rangle$, we define:
\begin{align*}
\langle\sigma_1,\dots,\sigma_n\rangle^\frown \langle \tau_1,\dots,\tau_m\rangle &= \langle \sigma_1,\dots,\sigma_n,\tau_1,\dots,\tau_m\rangle; \\
\langle \sigma_1,\dots,\sigma_n\rangle \bowtie \langle \tau_1,\dots,\tau_m\rangle &= \begin{cases} \langle \sigma_1,\dots,\sigma_n,\tau_2,\dots,\tau_m\rangle & \text{if }\sigma_n=\tau_1 \\ \text{undefined} & \text{otherwise};\end{cases} \\
\langle \sigma_1,\dots,\sigma_n\rangle * \langle \tau_1,\dots,\tau_m\rangle &= \begin{cases} \langle \sigma_1,\dots,\sigma_n,\tau_2,\dots,\tau_m\rangle & \text{if }\sigma_1=\tau_1 \\ \text{undefined} & \text{otherwise}.\end{cases}
\end{align*}
Thus, $\bowtie$ and $*$ have the same effect but we use them in different situations: $\bowtie$ when the conclusion of one proof matches the assumption of a second and $*$ when the assumptions of two proofs are the same. This notational distinction will help make clear why we are performing the operation in a given case.\end{definition}

\begin{lemma}\label{ProoFEM} For $\mathsf{L}\in \{\mathsf{IPC}^{-}, \underline{\mathsf{IPC}}^{-},\mathsf{FN4}\}$:
\begin{enumerate}
    \item\label{ProofRestrict} If $\langle \sigma_1,\dots,\sigma_n\rangle$ is an $\mathsf{L}$-proof given $R$ and $\sigma_i$ is a formula, then $\langle \sigma_1,\dots,\sigma_i\rangle$ is an $\mathsf{L}$-proof given $R$.
\item\label{ProofFuse} If $\langle \sigma_1,\dots,\sigma_n\rangle$ and $\langle \tau_1,\dots,\tau_m\rangle$ are $\mathsf{L}$-proofs given $R$ such that $\sigma_n=\tau_1$, then so is $\langle \sigma_1,\dots,\sigma_n\rangle\bowtie \langle \tau_1,\dots,\tau_m\rangle$.
\item\label{SameAssump} If $\langle\sigma_1,\dots,\sigma_n\rangle$ and $\langle \tau_1,\dots,\tau_m\rangle$ are $\mathsf{L}$-proofs given $R$ with $\sigma_1=\tau_1$, then $\langle\sigma_1,\dots,\sigma_n\rangle*\langle \tau_1,\dots,\tau_m\rangle$ is an $\mathsf{L}$-proof given $R$.
\end{enumerate}
\end{lemma}

Although reiteration is significantly restricted in $\underline{\mathsf{IPC}}^{-}$, the derivability relation does not change.

\begin{lemma}\label{SameInt} For any formula $\varphi,\psi\in\mathcal{L}(\wedge,\vee,\neg)$, $\varphi \vdash_{\mathsf{IPC}^{-}}\psi$ iff $\varphi \vdash_{\underline{\mathsf{IPC}}^{-}}\psi$.
\end{lemma}

\begin{proof} The right-to-left direction is trivial.  For the left-to-right direction, we prove by strong induction on the complexity of proofs for $\mathsf{IPC}^{-}$ that for any such proof $\langle\sigma_1,\dots,\sigma_n\rangle$ given a finite set $R$ of reiterables, there is an $\underline{\mathsf{IPC}}^{-}$-proof  $\langle\tau_1,\dots,\tau_m\rangle$ given $\{\bigwedge R\}$ (where this is empty if $R$ is empty) such that $\sigma_1=\tau_1$ and $\sigma_n=\tau_m$. For the base case, if $\langle \varphi\rangle$ is an $\mathsf{IPC}^{-}$-proof given $R$, then $\langle \varphi\rangle$ is an $\underline{\mathsf{IPC}}^{-}$-proof given $\{\bigwedge R\}$. Then we must check the inductive clause for each rule of $\mathsf{IPC}^-$, given in Appendix~\ref{AppendixIPC}, which could be the last rule applied in a proof of $\langle\sigma_1,\dots,\sigma_n\rangle$ given a set $R$ of reiterables. The only non-trivial cases are Reiteration, $\vee$E, and $\neg$I. 

For Reiteration, if we have an $\mathsf{IPC}^{-}$-proof $\langle \sigma_1,\dots,\sigma_n\rangle$ given $R$  and $\varphi\in R$, the inductive hypothesis gives us an $\underline{\mathsf{IPC}}^{-}$-proof $\langle \tau_1,\dots,\tau_m\rangle$ given $\{\bigwedge R\}$ with $\sigma_1=\tau_1$ and $\sigma_n=\tau_m$. Then by Reiteration in $\underline{\mathsf{IPC}}^{-}$, $\langle \tau_1,\dots,\tau_m,\bigwedge R\rangle$ is an $\underline{\mathsf{IPC}}^{-}$-proof, which we can extend to an $\underline{\mathsf{IPC}}^{-}$-proof $\langle \tau_1,\dots,\tau_m,\bigwedge R,\dots,\varphi\rangle$ by sufficiently many applications of $\wedge$E, since $\varphi$ is a conjunct of $\bigwedge R$. 

For $\neg$I, suppose we have an $\mathsf{IPC}^{-}$-proof $\langle\sigma_1,\dots,\sigma_n\rangle$, $\sigma_i$ is a formula $\psi$, and we have an $\mathsf{IPC}^{-}$-proof $\tau$ given some $R'\subseteq R\cup\{\sigma_j\mid\sigma_j\mbox{ a formula}\}$ beginning with $\varphi$ and ending with $\neg\psi$. Let us write $R'\cap \{\sigma_j\mid\sigma_j\mbox{ a formula}\}$ as $\{\xi_1,\dots,\xi_s\}$ and $R'\setminus \{\sigma_j\mid\sigma_j\mbox{ a formula}\}$ as $\{\rho_1,\dots,\rho_u\}$. For each $\xi_t$, $\langle \sigma_1,\dots,\xi_t\rangle$ is an $\mathsf{IPC}^{-}$-proof given $R$ by Lemma~\ref{ProoFEM}.\ref{ProofRestrict} for $\mathsf{IPC}^{-}$. Hence by the inductive hypothesis, for each $\xi_t$, we have an $\underline{\mathsf{IPC}}^{-}$-proof $\Pi_t$ given $\{\bigwedge R\}$ with assumption $\sigma_1$ and conclusion $\xi_t$. By the same reasoning, we have an $\underline{\mathsf{IPC}}^{-}$-proof $\Psi$ given $\{\bigwedge R\}$ with assumption $\sigma_1$ and conclusion $\psi$. Now by repeated application of Lemma~\ref{ProoFEM}.\ref{SameAssump}, $\Psi*\Pi_1 *\dots * \Pi_s$ (with implicit left association of $*$) is an $\underline{\mathsf{IPC}}^{-}$-proof given $\{\bigwedge R\}$ (if $s=0$, this proof is simply $\Psi$). Next, using Reiteration in $\underline{\mathsf{IPC}}^{-}$,  $(\Psi*\Pi_1 *\dots * \Pi_s)^\frown \langle \bigwedge R\rangle$ is an $\underline{\mathsf{IPC}}^{-}$-proof given $\{\bigwedge R\}$. Then by repeated application of $\wedge$E starting with $\bigwedge R$ and then repeated application of $\wedge$I to deduce $\rho$ formulas from $R$ and $\xi$ formulas at the end of the $\Pi_j$ proofs, we can deduce $\bigwedge R'$ at the end of an $\underline{\mathsf{IPC}}^{-}$-proof $\Pi$ given $\{\bigwedge R\}$ that starts with $\sigma_1$. Applying the inductive hypothesis to the $\mathsf{IPC}^{-}$-proof $\tau$, we obtain an $\underline{\mathsf{IPC}}^{-}$-proof $\Theta$ given $\{\bigwedge R'\}$ with assumption $\varphi$ and conclusion $\neg\psi$. Then by $\neg$I in  $\underline{\mathsf{IPC}}^{-}$, $\Pi^\frown \langle \Theta, \neg\varphi\rangle$ is an $\underline{\mathsf{IPC}}^{-}$-proof given $\{\bigwedge R\}$ with assumption $\sigma_1$ and conclusion $\neg\varphi$, as desired.

The proof for $\vee$E is similar in spirit to that for $\neg$I.\end{proof}

As a final preliminary, it will help to take a detour through a variant of the standard GMT translation (also used in Section~\ref{ClassicalSect}).

\begin{definition}\label{mu+Def} Define the translation $\mu^+$ from $\mathcal{L}(\wedge,\vee,\neg)$ to $\mathcal{L}(\wedge,\vee,\neg,\Box)$ as follows:
\begin{enumerate}
    \item $\mu^+(\top)=\Box\top$ and $\mu^+(p) = \Box p$;
    \item $\mu^+(\phi \me \psi) = \Box(\mu^+(\phi) \me \mu^+(\psi))$;
    \item $\mu^+(\phi \jo \psi) = \Box(\mu^+(\phi) \jo \mu^+(\psi))$;
    \item $\mu^+(\neg \phi) = \Box \neg \mu^+(\phi)$.
\end{enumerate}
\end{definition}

It is easy to see that $\mathsf{FN4}$ can prove the equivalence of $\mu(\varphi)$ and $\mu^+(\varphi)$.

\begin{lemma}\label{PlusLem} For any formula $\varphi$, $\mu(\varphi)\dashv \vdash_{\mathsf{FN4}} \mu^+(\varphi)$.
\end{lemma}

Another useful lemma is the following.

\begin{lemma}\label{mu+IntroElim} For any $\varphi,\psi\in\mathcal{L}(\wedge,\vee,\neg)$:
\begin{enumerate}
\item\label{mu+IntroElim1} $\mu^+(\varphi)\wedge\mu^+(\psi)\vdash_{\mathsf{FN4}}\mu^+(\varphi\wedge\psi)$;
\item $\mu^+(\varphi\wedge\psi) \vdash_{\mathsf{FN4}} \mu^+(\varphi)$ and $\mu^+(\varphi\wedge\psi) \vdash_{\mathsf{FN4}} \mu^+(\psi)$; 
\item $\mu^+(\varphi)\vdash_{\mathsf{FN4}}\mu^+(\varphi\vee\psi)$ and $\mu^+(\psi)\vdash_{\mathsf{FN4}}\mu^+(\varphi\vee\psi)$;
\item $\mu^+(\varphi)\wedge\mu^+(\neg\varphi)\vdash_{\mathsf{FN4}}\mu^+(\psi)$.
\end{enumerate}
\end{lemma}

\begin{proof} For the first part, the proof is given in Figure~\ref{MuAndFig}. The other parts are also straightforward.
\end{proof}

\begin{figure}[H]
\[\begin{nd}
\have [1] {1} {\mu^+(\varphi)\wedge\mu^+(\psi)}
\have [2] {2}  {\mu^+(\varphi)} \ae{1}
\have [3] {3}  {\Box\mu^+(\varphi)}  \FourAx{2}
\have [4] {4}  {\mu^+(\psi)} \ae{1}
\have [5] {5}  {\Box\mu^+(\psi)} \FourAx{4}
\open
\have [6] {6}   {\hspace{-.24in}\Box\;\;\;\mu^+(\varphi)}\boxe{3} 
\have [7] {7}   {\mu^+(\psi)} \boxe{5}
\have [8] {8}   {\mu^+(\varphi)\wedge \mu^+(\psi)} \ai{6,7}
\close
\have [9] {9}   {\Box (\mu^+(\varphi)\wedge\mu^+(\psi))}\boxi{6-8}
\end{nd}\]
\caption{$\mathsf{FN4}$ proof for Lemma~\ref{mu+IntroElim}.\ref{mu+IntroElim1}. Officially a $\Box$-subproof begins with $\top$ (see Appendix~\ref{AppendixModal}), but we suppress this step in proof diagrams.}\label{MuAndFig}
\end{figure}

We are now ready for the main proof-theoretic work.

\begin{theorem}\label{ProofTransform} For every $\underline{\mathsf{IPC}}^{-}$-proof $\langle \sigma_1,\dots,\sigma_n\rangle$ given $R$, there is an $\mathsf{FN4}$ proof $\langle \pi_1,\dots,\pi_m\rangle$ given $\{\mu^+(\rho)\mid \rho\in R\}$ such that $\pi_1=\mu^+(\sigma_1)$ and $\pi_m=\mu^+(\sigma_n)$.
\end{theorem}

\begin{proof} By strong induction on the complexity of proofs for $\underline{\mathsf{IPC}}^{-}$. For the base case, given an $\underline{\mathsf{IPC}}^{-}$-proof $\langle\varphi\rangle$, $\langle \mu^+(\varphi)\rangle$ is an $\mathsf{FN4}$-proof. For the inductive step, we consider each of the following rules of $\underline{\mathsf{IPC}}^{-}$ that could be the last rule applied, where $R$, $R_1$, $R_2$, and $R'$ have cardinality at most one:
\begin{itemize}
\item If $\langle \sigma_1,\dots,\sigma_n\rangle$ is a proof given $R$ and $\varphi\in R$, then $\langle \sigma_1,\dots,\sigma_n,\varphi\rangle$ is a proof given $R$ (Reiteration).
\item If $\langle \sigma_1,\dots,\sigma_n\rangle$ is a proof given $R$, then so is $\langle \sigma_1,\dots,\sigma_n,\top\rangle$ ($\top$I).
\item If $\langle \sigma_1,\dots,\sigma_n\rangle$ is a  proof given $R$ and $\sigma_i,\sigma_j$  are formulas, then $\langle \sigma_1,\dots,\sigma_n,\sigma_i\wedge\sigma_j\rangle$ is a proof given $R$ ($\wedge$I).
\item If $\langle \sigma_1,\dots,\sigma_n\rangle$ is a  proof given $R$ and $\sigma_i$ is a formula of the form $\varphi\wedge\psi$, then $\langle \sigma_1,\dots,\sigma_n,\varphi\rangle$ and $\langle \sigma_1,\dots,\sigma_n,\psi\rangle$ are proofs given $R$ ($\wedge$E).
\item If $\langle \sigma_1,\dots,\sigma_n\rangle$ is a  proof given $R$ and  $\sigma_i$ is a formula, then for any formula $\varphi$,  both $\langle \sigma_1,\dots,\sigma_n,\sigma_i\vee\varphi\rangle$ and $\langle \sigma_1,\dots,\sigma_n,\varphi\vee \sigma_i\rangle$ are proofs given $R$ ($\vee$I).
\item If $\langle \sigma_1,\dots,\sigma_n\rangle$ is a  proof given $R$, $\sigma_i$ is a formula of the form $\varphi\vee\psi$, $\tau_1$ is a proof given some $R_1\subseteq \{\sigma_j\mid \sigma_j\mbox{ a formula}\}$ beginning with $\varphi$ and ending with $\chi$, and $\tau_2$ is a proof given some $R_2\subseteq \{\sigma_j\mid \sigma_j\mbox{ a formula}\}$  beginning with $\psi$ and ending with $\chi$, then $\langle \sigma_1,\dots,\sigma_n, \tau_1,\tau_2,\chi\rangle$ is a proof given $R$ ($\vee $E).
\item If $\langle \sigma_1,\dots,\sigma_n\rangle$ is a  proof given $R$, $\sigma_i$ is a formula  $\psi$, and $\tau$ is a proof given some $R'\subseteq \{\sigma_j\mid \sigma_j\mbox{ a formula}\}$ beginning with $\varphi$ and ending with $\neg\psi$, then $\langle \sigma_1,\dots,\sigma_n,\tau,\neg\varphi\rangle$ is a proof given $R$ ($\neg$I).
\item If $\langle \sigma_1,\dots,\sigma_n\rangle$ is a  proof given $R$ and $\sigma_i$ and $\sigma_j$ are formulas of the form $\varphi$ and $\neg\varphi$, respectively, then for any formula $\psi$, $\langle \sigma_1,\dots,\sigma_n,\psi\rangle$ is a proof given $R$ ($\neg$E). 
\end{itemize}

For the rule of Reiteration, assume $\langle \sigma_1,\dots,\sigma_n\rangle$ is an $\underline{\mathsf{IPC}}^{-}$-proof given $R$, so by the inductive hypothesis there is an $\mathsf{FN4}$-proof $\Pi$ given $\mu^+[R]$ with assumption $\mu^+(\sigma_1)$ and conclusion $\mu^+(\sigma_n)$. Then for $\varphi\in R$, since $\mu^+(\varphi)$ begins with a $\Box$, we have that $\Pi ^\frown \langle\mu^+(\varphi)\rangle$ is an $\mathsf{FN4}$-proof given $\mu^+[R]$ by $\Box$-Reiteration. 

For $\top$I, the inductive hypothesis gives us a proof $\Pi$ as in the previous paragraph. In $\mathsf{FN4}$, since $\langle \top\rangle$ is a proof given $\varnothing$, $\langle \Box, \langle \top\rangle\rangle$ is a $\Box$-proof given $\varnothing$ by the first bullet point in Appendix~\ref{AppendixModal}. Then by $\Box$I, $\Pi^\frown\big\langle\langle \Box, \langle \top\rangle\rangle, \Box\top \big\rangle$ is an $\mathsf{FN4}$-proof given $\mu^+[R]$ with assumption $\mu^+(\sigma_1)$ and conclusion $\mu^+(\top)$. 

For $\wedge $I, for any formulas $\sigma_i,\sigma_j$ in $\langle \sigma_1,\dots,\sigma_n\rangle$, both $\langle \sigma_1,\dots,\sigma_i\rangle$ and $\langle \sigma_1,\dots,\sigma_j\rangle$ are $\underline{\mathsf{IPC}}^{-}$-proofs given $R$ by Lemma~\ref{ProoFEM}.\ref{ProofRestrict}. Then by the inductive hypothesis, there are $\mathsf{FN4}$-proofs $\Pi_i$ and $\Pi_j$ given $\mu^+[R]$ where $\mu^+(\sigma_1)$ is the assumption of both, $\mu^+(\sigma_i)$ is the conclusion of $\Pi_i$, and $\mu^+(\sigma_j)$ is the conclusion of $\Pi_j$. Then by Lemma \ref{ProoFEM}.\ref{SameAssump}, $\Pi_i * \Pi_j$ is an $\mathsf{FN4}$-proof given $\mu^+[R]$. Then by $\wedge$I in $\mathsf{FN4}$, \[\Pi:=(\Pi_i * \Pi_j)^\frown \langle \mu^+(\sigma_i)\wedge\mu^+(\sigma_j)\rangle\] is an $\mathsf{FN4}$-proof given $\mu^+[R]$. Now by Lemma \ref{mu+IntroElim}.\ref{mu+IntroElim1}, there is an $\mathsf{FN4}$-proof $\Theta$ from $\mu^+(\sigma_i)\wedge \mu^+(\sigma_j)$ to $\mu^+(\sigma_i\wedge\sigma_j)$ given $\mu^+[R]$. So by Lemma~\ref{ProoFEM}.\ref{ProofFuse}, we can extend $\Pi$ to the $\mathsf{FN4}$-proof $\Pi\bowtie\Theta$  given $\mu^+[R]$ with assumption $\mu^+(\sigma_1)$ and conclusion $\mu^+(\sigma_i\wedge\sigma_j)$, as desired.

The proofs for $\wedge$E, $\vee$I, and $\neg$E are similar in spirit to the proof for $\wedge$I, using the other parts of Lemma~\ref{mu+IntroElim}.

For $\vee$E, assume $\langle \sigma_1,\dots,\sigma_n\rangle$ is an $\underline{\mathsf{IPC}}^{-}$-proof given $R$, $\sigma_i$ is a formula of the form $\varphi\vee\psi$, $\tau_1$ is an $\underline{\mathsf{IPC}}^{-}$-proof given some $R_1\subseteq \{\sigma_j\mid \sigma_j\mbox{ a formula}\}$ beginning with $\varphi$ and ending with $\chi$, and $\tau_2$ is an $\underline{\mathsf{IPC}}^{-}$-proof given some $R_2\subseteq \{\sigma_j\mid \sigma_j\mbox{ a formula}\}$  beginning with $\psi$ and ending with $\chi$. By Lemma~\ref{ProoFEM}.\ref{ProofRestrict}, $\langle \sigma_1,\dots,\sigma_i\rangle$ is also an $\underline{\mathsf{IPC}}^{-}$-proof given $R$. Let us assume that $R_1$ and $R_2$ are both nonempty; the proof is easy to adjust if either or both are empty. Then for the  $\rho_1\in R_1$ and $\rho_2\in R_2$, $\langle \sigma_1,\dots,\sigma_i, \rho_1\rangle$ and $\langle \sigma_1,\dots,\sigma_i, \rho_2\rangle$ are also $\underline{\mathsf{IPC}}^{-}$-proofs given $R$ by Lemma~\ref{ProoFEM}.\ref{ProofRestrict}.

By the inductive hypothesis, the following $\mathsf{FN4}$-proofs  exist:
\begin{itemize}
\item a proof $\Pi_1$ given $\mu^+[R]$ with assumption $\mu^+(\sigma_1)$ and conclusion $\mu^+(\rho_1)$;
\item a proof $\Pi_2$ given $\mu^+[R]$ with assumption $\mu^+(\sigma_1)$ and conclusion $\mu^+(\rho_2)$;
\item a proof $\Delta$  given $\mu^+[R]$ with assumption $\mu^+(\sigma_1)$ and conclusion $\mu^+(\varphi\vee\psi)$;
\item a proof $\Phi$ given $\{\mu^+(\rho_1)\}$ with assumption $\mu^+(\varphi)$ and conclusion $\mu^+(\chi)$;
\item a proof $\Psi$ given $\{\mu^+(\rho_2)\}$ with assumption $\mu^+(\psi)$ and conclusion $\mu^+(\chi)$.
\end{itemize}
By Lemma \ref{ProoFEM}.\ref{SameAssump}, since the first three types of proofs start with the same formula, $\Pi_1 * \Pi_2 * \Delta := (\Pi_1 * \Pi_2)*\Delta$
is an $\mathsf{FN4}$-proof given $\mu^+[R]$. Now since there is a one-step $\mathsf{FN4}$-proof
$\langle \mu^+(\varphi\vee\psi), \mu^+(\varphi)\vee \mu^+(\psi)\rangle$
using the \textsf{T} Rule, it follows by Lemma \ref{ProoFEM}.\ref{ProofFuse} that 
$(\Pi_1 * \Pi_2 * \Delta)^\frown \langle \mu^+(\varphi)\vee \mu^+(\psi)\rangle$ 
is an $\mathsf{FN4}$-proof given $\mu^+[R]$. Then by the $\vee$E rule in $\mathsf{FN4}$, 
\[\big((\Pi_1 * \Pi_2 * \Delta) ^\frown \langle \mu^+(\varphi)\vee \mu^+(\psi)\rangle\big) ^\frown \big\langle \Phi,\Psi,\mu^+(\chi)\big\rangle\] 
is an $\mathsf{FN4}$-proof given $\mu^+[R]$, since $\mu^+(\rho_1)$ and $\mu^+(\rho_2)$ occur in the proof before $\Phi$ and $\Psi$, respectively, at the ends of $\Pi_1$ and $\Pi_2$.

For $\neg$I, assume $\langle \sigma_1,\dots,\sigma_n\rangle$ is an $\underline{\mathsf{IPC}}^{-}$-proof given $R$, $\sigma_i$ is a formula $\psi$, and $\tau$ is a proof given some $R'\subseteq \{\sigma_j\mid \sigma_j\mbox{ a formula}\}$ beginning with $\varphi$ and ending with $\neg\psi$. By Lemma \ref{ProoFEM}.\ref{ProofRestrict}, $\langle\sigma_1,\dots,\psi\rangle$ is also an $\underline{\mathsf{IPC}}^{-}$-proof given $R$. Let us assume that $R'$ is nonempty; the proof is easy to adjust if $R'$ is empty. Then for the $\rho\in R'$, $\langle \sigma_1,\dots,\rho\rangle$ is also an $\underline{\mathsf{IPC}}^{-}$-proof given~$R$ by Lemma~\ref{ProoFEM}.\ref{ProofRestrict}.  

By the inductive hypothesis, the following $\mathsf{FN4}$-proofs exist:
\begin{itemize}
\item a proof $\Pi$ given $\mu^+[R]$ with assumption $\mu^+(\sigma_1)$ and conclusion $\mu^+(\rho)$;
\item a proof $\Psi$ given $\mu^+[R]$  with assumption $\mu^+(\sigma_1)$ and conclusion $\mu^+(\psi)$;
\item a proof $\Phi$ given $\{\mu^+(\rho)\}$ with assumption $\mu^+(\varphi)$ and conclusion $\mu^+(\neg\psi)$.
\end{itemize}
Then we claim that the following is an $\mathsf{FN4}$-proof given $\mu^+[R]$ with assumption $\mu^+(\sigma_1)$ and conclusion $\mu^+(\neg \varphi)=\Box\neg\mu^+(\varphi)$, as desired:
\[ (\Pi * \Psi) ^\frown \bigg\langle  \underset{\text{by \textsf{4} Rule}}{\underbrace{\Box\mu^+(\rho),\Box\mu^+(\psi)}}, \Big\langle \Box,\big\langle \top, \underset{\text{by $\Box$E}}{\underbrace{\mu^+(\rho),\mu^+(\psi)}}, \Phi ^\frown \langle \underset{\text{by \textsf{T} Rule}}{\underbrace{\neg \mu^+(\psi)}}\rangle, \underset{\text{by $\neg$I}}{\underbrace{\neg \mu^+(\varphi)}} \big\rangle \Big\rangle,  \underset{\text{by $\Box$I}}{\underbrace{\Box \neg \mu^+(\varphi)}}\bigg\rangle\]
Since $\mu^+(\rho)$, the last formula of $\Pi$, and $\mu^+(\psi)$, the last formula of $\Psi$, both begin with a $\Box$, we can use the \textsf{4} Rule to extend $\Pi*\Psi$ with $\langle \Box\mu^+(\rho), \Box \mu^+(\psi)\rangle$. Then since our goal is to prove $\mu^+(\neg\varphi)$, which is $\Box\neg\mu^+(\varphi)$, we open a $\Box$-subproof, which begins with $\top$ as always. We then immediately apply $\Box$E to derive $\mu^+(\rho)$ and $\mu^+(\psi)$ in the $\Box$-subproof. Our goal is to obtain $\neg \mu^+(\varphi)$ as the conclusion of the $\Box$-subproof, so we open a subproof for $\neg$I inside the $\Box$-subproof. That $\neg$I-subproof consists of $\Phi$, which is a proof given $\{\mu^+(\rho)\}$ with assumption $\mu^+(\varphi)$ and conclusion $\mu^+(\neg\psi)=\Box \neg\mu^+(\psi)$, which we can extend with the new conclusion $\neg\mu^+(\psi)$ using the \textsf{T} Rule. Then by the $\neg$I rule, we are entitled to leave the $\neg$I-subproof and derive $\neg\mu^+(\varphi)$ as the conclusion of the $\Box$-subproof, as desired.\end{proof}

Theorem~\ref{ProofTransform}, together with Lemmas \ref{SameInt} and \ref{PlusLem}, gives us another proof of Corollary~\ref{IPC-toFN4}.

\subsubsection{Conditional Extension}

To extend the results of the previous section to $\mathsf{IPC}$, as defined in Appendix~\ref{AppendixIPC}, we define $\underline{\mathsf{IPC}}$ similarly but only allow a subproof to reiterate a single formula from its parent subproof, so $|R'|\leq 1$ and $R'\subseteq \{\sigma_j\mid \sigma_j\mbox{ a formula}\}$ in $\to$I. Clearly the extension of Lemma~\ref{ProoFEM} to $\mathsf{IPC}$, $\underline{\mathsf{IPC}}$, and $\mathsf{FC4}$  continues to hold, and the proof of the following is similar to that of Lemma~\ref{SameInt}.

\begin{lemma}\label{SameInt2} For any formula $\varphi,\psi\in\mathcal{L}(\wedge,\vee,\neg,\to)$, $\varphi \vdash_{\mathsf{IPC}}\psi$ iff $\varphi \vdash_{\underline{\mathsf{IPC}}}\psi$.
\end{lemma}

Next, we add the clause $\mu^+(\varphi\to\psi)=\Box (\mu^+(\varphi)\to\mu^+(\psi))$ to the variant $\mu^+$ of the GMT translation, for which we have the following extension of Lemma \ref{PlusLem}.

\begin{lemma}\label{PlusLemFC4} For any $\varphi\in \mathcal{L}(\wedge,\vee,\neg,\to)$, $\mu(\varphi)\dashv \vdash_{\mathsf{FC4}} \mu^+(\varphi)$.
\end{lemma}

Now we can extend Theorem~\ref{ProofTransform} to the following.

\begin{theorem}\label{ProofTransformIPC} For every $\underline{\mathsf{IPC}}$-proof $\langle \sigma_1,\dots,\sigma_n\rangle$ given $R$, there is an $\mathsf{FC4}$ proof $\langle \pi_1,\dots,\pi_m\rangle$ given $\{\mu^+(\rho)\mid \rho\in R\}$ such that $\pi_1=\mu^+(\sigma_1)$ and $\pi_m=\mu^+(\sigma_n)$.\end{theorem}

\begin{proof} We need only extend the proof of Theorem~\ref{ProofTransform} with cases for $\to$I and $\to$E. The case for $\to$E is obvious, using $\to$E-steps in $\mathsf{FC4}$ to match those in $\mathsf{IPC}$. The case for $\to$I follows the same kind of strategy as the case for $\neg$I in the proof of Theorem~\ref{ProofTransform}: since $\mu^+(\varphi\to\psi)=\Box(\mu^+(\varphi)\to\mu^+(\psi))$, we need to use a $\Box$-subproof, into which we pull necessary materials using $\Box$E, so those materials can then be $\Box$-reiterated into the $\to$-subproof given by the inductive hypothesis.\end{proof}

\subsubsection{First-Order Extension}\label{AppendixIQC}

To extend the results of the previous section to $\mathsf{IQC}$, as defined in Appendix~\ref{AppendixQuant}, we define $\underline{\mathsf{IQC}}$ similarly but only allow a subproof to reiterate a single formula from its parent subproof, so $|R'|\leq 1$ and $R'\subseteq \{\sigma_j\mid \sigma_j\mbox{ a formula}\}$ in $\exists$E. The following can be proved by extending the inductive proofs of Lemmas~\ref{SameInt} and \ref{SameInt2} with cases for the quantifier rules.

\begin{lemma}\label{SameInt3} For any formula $\varphi,\psi\in\mathcal{L}(\wedge,\vee,\neg,\to,\forall,\exists)$, $\varphi \vdash_{\mathsf{IQC}}\psi$ iff $\varphi \vdash_{\underline{\mathsf{IQC}}}\psi$.
\end{lemma}

 As in the previous sections, it is helpful to work with the variant $\mu^+$ of the standard GMT translation.

\begin{definition} Define a translation $\mu^+$ from $\mathcal{L}(\wedge,\vee,\neg,\to,\forall,\exists)$ to $\mathcal{L}(\wedge,\vee,\neg,\to,\forall,\exists,\Box)$ as follows:
\begin{enumerate}
\item $\mu^+(P(t_1,\dots,t_n))=\Box P(t_1,\dots,t_n)$ for atomic formulas $P(t_1,\dots,t_n)$;
\item clauses for $\top$, $\wedge$, $\vee$, and $\neg$ from Definition~\ref{mu+Def}, plus
  $\mu^+(\varphi\to\psi)=\Box(\mu^+(\varphi)\to\mu^+(\psi))$;
\item $\mu^+(\forall x\varphi)=\Box \forall x\mu^+(\varphi)$;
\item $\mu^+(\exists x\varphi)=\Box \exists x\mu^+(\varphi)$.
\end{enumerate}
\end{definition}

The following lemma is easy to prove, using the argument in Figure~\ref{ForallExistBox} for the key step.

\begin{lemma}\label{PlusLemQFC4} For any $\varphi\in \mathcal{L}(\wedge,\vee,\neg,\to,\forall,\exists)$, $\mu(\varphi)\dashv \vdash_{\mathsf{QFC4}} \mu^+(\varphi)$.
\end{lemma}

\begin{figure}[H]

\[\begin{nd}
\have [1] {1} {\exists x \mu^+(\varphi)}
\open
\hypo [2] {2}   {\mu^+(\varphi)}
\have [3] {3}   {\Box\mu^+(\varphi)} \FourAx{2}
\open
\have [4] {4}   {\hspace{-.24in}\Box\;\;\; \mu^+(\varphi)} \boxe{3}
\have [5] {5}   {\exists x \mu^+(\varphi)} \Ei{4}
\close
\have [6] {6}   {\Box\exists x \mu^+(\varphi)} \boxi{4-5}
\close
\have [7] {7} {\Box \exists x \mu^+(\varphi)} \Ee{1, 2-6}
\end{nd}
\]
\caption{Proof of $\Box\exists x \mu^+(\varphi)$ from $\exists x \mu^+(\varphi)$ in $\mathsf{QFC4}$.}\label{ForallExistBox}
\end{figure}

In the first-order case, we must slightly revise the second part of Lemma~\ref{ProoFEM} in the presence of $\forall$I.

\begin{lemma}\label{ProoFEM2} For $\mathsf{L}\in \{\mathsf{IQC}, \underline{\mathsf{IQC}},\mathsf{QFC4}\}$:
\begin{enumerate}
    \item\label{QProofRestrict} If $\langle \sigma_1,\dots,\sigma_n\rangle$ is an $\mathsf{L}$-proof given $R$ and $\sigma_i$ is a formula, then $\langle \sigma_1,\dots,\sigma_i\rangle$ is an $\mathsf{L}$-proof given $R$.
\item\label{QProofFuse} If $\langle \sigma_1,\dots,\sigma_n\rangle$ and $\langle \tau_1,\dots,\tau_m\rangle$ are $\mathsf{L}$-proofs given $R$ such that $\sigma_n=\tau_1$, then there is an $\mathsf{L}$-proof $\langle \pi_1,\dots,\pi_\ell\rangle$ given $R$ such that $\pi_1=\sigma_1$ and $\pi_\ell=\tau_m$.
\item\label{QSameAssump} If $\langle\sigma_1,\dots,\sigma_n\rangle$ and $\langle \tau_1,\dots,\tau_m\rangle$ are $\mathsf{L}$-proofs given $R$ with $\sigma_1=\tau_1$, then $\langle\sigma_1,\dots,\sigma_n\rangle*\langle \tau_1,\dots,\tau_m\rangle$ is an $\mathsf{L}$-proof given $R$.
\end{enumerate}
\end{lemma}
\begin{proof} The only part that is not straightforward is part \ref{QProofFuse}. First consider $\mathsf{L}\in \{\mathsf{IQC}, \mathsf{QFC4}\}$. The problem is that if $\forall$I is applied to obtain some $\tau_i$, which requires that the quantified variable $v$ not be free in $\tau_1$, then when we try to fuse $\sigma=\langle \sigma_1,\dots,\sigma_n\rangle$ and $\tau=\langle \tau_1,\dots,\tau_m\rangle$ together, we are changing the initial assumption of the resulting proof to $\sigma_1$, and $v$ may have a free occurrence in $\sigma_1$, in which case the $\forall$I step to derive $\tau_i$ will become illegal. The trick to overcome this problem is to use a vacuous $\exists$I step: where $u$ is a fresh variable not appearing in the original proofs, the sequence 
\[\langle \sigma_1,\dots,\sigma_n, \exists u \sigma_n, \tau, \tau_m\rangle\]
is an $\mathsf{L}$-proof given $R$, where $\tau_m$ is obtained by $\exists$E from $\exists u\sigma_n$ and the subproof $\tau$; note that the assumption of $\tau$ is $\tau_1=\sigma_n$, and $u$ does not occur free in $\tau_m$ or in any formula in $R\cup \{\sigma_j\mid \sigma_j\mbox{ a formula}\}\cup \{\exists u\sigma_n\}$, so this is a legal application of $\exists$E to obtain an $\mathsf{L}$-proof of $\tau_m$ given $R$. The problem and solution are essentially the same for $\underline{\mathsf{IQC}}$ except that if $\tau$ is a proof given $R=\{\rho\}$, then we need to make sure $\rho$ occurs in the parent proof if $\rho$ is reiterated into $\tau$ as the $\exists$E-subproof:
\[\langle \sigma_1,\dots,\sigma_n, \exists u \sigma_n, \rho, \tau, \tau_m\rangle.\]
Here we simply use Reiteration to deduce $\rho$ just before the $\tau$ subproof.\end{proof}

We will also use the following lemma about the relation between proofs and $\Box$-proofs.

\begin{lemma}\label{ProofToBoxproof} If $\langle \sigma_1,\dots,\sigma_n\rangle$ is a proof in $\mathsf{QFC4}$ given $R=\{\Box\alpha\}$, then $\langle \Box, \langle \top, \Box\alpha,  \sigma_1,\dots,\sigma_n\rangle\rangle$ is a $\Box$-proof in $\mathsf{QFC4}$ given $B=\{\Box\sigma_1,\Box\Box\alpha\}$.
\end{lemma}
\noindent Lemma~\ref{ProofToBoxproof} is provable by induction on proofs using the fact that for each rule for proofs we have a corresponding rule for $\Box$-proofs, and for $\forall$I and $\exists$E, the side conditions are preserved since the formulas in $B$ have the same free variables as in $\sigma_1$ and $\Box\alpha$. Note that we get $\Box \alpha$ at the top level of the $\Box$-proof by $\Box$E and from there it is available for $\Box$-Reiteration into further subproofs just as it was in the original proof.

Now we can prove the first-order analogue of Theorem~\ref{ProofTransformIPC}.

\begin{theorem}\label{ProofTransformIQC} For every $\underline{\mathsf{IQC}}$-proof $\langle \sigma_1,\dots,\sigma_n\rangle$ given $R$, there is a $\mathsf{QFC4}$ proof $\langle \pi_1,\dots,\pi_m\rangle$ given $\{\mu^+(\rho)\mid \rho\in R\}$ such that $\pi_1=\mu^+(\sigma_1)$ and $\pi_m=\mu^+(\sigma_n)$.\end{theorem}

\begin{proof} We need only extend the proof of Theorem~\ref{ProofTransformIPC} with cases for $\forall$I, $\forall$E, $\exists$I, and $\exists$E. The proofs for the latter three are analogous to those for $\wedge$E, $\vee$I, and $\vee$E, now using Lemma~\ref{ProoFEM2}.

For the $\forall$I case, assume $\langle\sigma_1,\dots,\sigma_n\rangle$ is an $\underline{\mathsf{IQC}}$ proof given $R$ and $\sigma_i$ is a formula $\psi$ such that $v$ does not occur free in $\sigma_1$ or in any formula in $R$. Suppose $R=\{\rho\}$; one can easily adapt the proof for the case where $R=\varnothing$. Then by the inductive hypothesis and Lemma~\ref{ProoFEM2}.\ref{QProofRestrict}, we obtain a proof $\Psi$ given $\{\mu^+(\rho)\}$  with assumption $\mu^+(\sigma_1)$ and conclusion $\mu^+(\psi)$. Then the following is a $\mathsf{QFC4}$-proof given $\{\mu^+(\rho)\}$ with assumption $\mu^+(\sigma_1)$ and conclusion $\mu^+(\forall v \psi)=\Box \forall v \mu^+(\psi)$, as desired:
\[ \big\langle \mu^+(\sigma_1), \underset{\text{by 4 Rule}}{\underbrace{\Box\mu^+(\sigma_1)}}, \underset{\text{by $\Box$-Reit}}{\underbrace{\mu^+(\rho)}},  \underset{\text{by 4 Rule}}{\underbrace{\Box\mu^+(\rho)}}, \Big\langle \Box,\big\langle \top, \mu^+(\rho)\rangle ^\frown \Psi^\frown\langle \underset{\text{by $\forall$I rule}}{\underbrace{{\forall v \mu^+(\psi)}}}\rangle  \Big\rangle,  \underset{\text{by $\Box$I}}{\underbrace{\Box \forall v \mu^+(\psi)}}\big\rangle,\]
where the claimed $\Box$-proof is indeed a $\Box$-proof given $B=\{\Box\mu^+(\sigma_1), \Box \mu^+(\rho)\}$ by Lemma~\ref{ProofToBoxproof} with $\mu^+(\rho)=\Box\alpha$, followed by an application of $\forall$I. We get $\mu^+(\rho)$ immediately after $\top$ by $\Box$E, and we get $\mu^+(\sigma_1)$ at the beginning of $\Psi$ by $\Box$E. Since $v$ does not occur free in $\sigma_1$ or in $\rho$, it also does not occur free in any formula in $B$, so the application of $\forall$I immediately after $\Psi$ is legal given how $\forall$I is stated for use inside $\Box$-proofs in Appendix~\ref{AppendixModal}.\end{proof} 

Combining Lemmas~\ref{SameInt3} and \ref{PlusLemQFC4} with Theorem \ref{ProofTransformIQC}, we obtain the following result.

\begin{corollary}\label{IQCtoFun} For any $\varphi,\psi\in\mathcal{L}(\wedge,\vee,\neg,\to,\forall,\exists)$, if $\varphi\vdash_{\mathsf{IQC}}\psi$, then $\mu(\varphi)\vdash_{\mathsf{QFC4}} \mu(\psi)$.
\end{corollary}

\subsection{Soundness and Completeness of \textsf{FN4}}\label{FN4CompAppendix}

In this Appendix, we prove that $\mathsf{FN4}$ is sound and complete with respect to the class of $\mathbf{FN4}$-frames (recall Section~\ref{RelationalFN4}). 

\subsubsection{Soundness}

The soundness of $\mathsf{FN4}$ with respect to $\mathbf{FN4}$-frames follows from the soundness of $\mathsf{FN4}$ with respect to $\mathbf{FN4}$-algebras together with the fact that the dual algebra of every $\mathbf{FN4}$-frame is an $\mathbf{FN4}$-algebra.

\begin{proposition}\label{DirectEnrichProp} Let $(X,\op,\shortdashleftarrow)$ be a fundamental modal frame with $\shortdashleftarrow$ transitive. If $(X,\op,\shortdashleftarrow)$ satisfies the direct enrichment condition of Definition~\ref{FNframes}, then for all propositions $A,B,C$,
\[\mbox{if } A\cap \Box C\subseteq \neg B\mbox{, then }B\cap \Box C\subseteq \neg A.\]
\end{proposition}
\begin{proof} Assume $A\cap \Box C\subseteq \neg B$ and $x\in B\cap \Box C$, which implies $x\in \Box\Box C$ by the transitivity of $\shortdashleftarrow$. Suppose toward a contradiction that $x\not\in \neg A$, so there is a $y\op x$ with $y\in A$. Then by direct enrichment, there is a prerefinement $y^+$ of $y$ such that $x\po y^+ \shortdashleftarrow x$. Since $y^+$ is a prerefinement of $y$, $y^+\in A$, and since $y^+\shortdashleftarrow x$, $y^+\in \Box C$. Thus, $y^+\in \neg B$. But this is inconsistent with $x\in B$, $x\po y^+$, and the pseudosymmetry of~$\op$. Thus, we conclude that there is no such $y$, so $x\in \neg A$.\end{proof}

\begin{proposition}\label{MergProp} If a modal frame $(X,\op,\shortdashleftarrow)$ satisfies the indirect enrichment condition of Definition~\ref{FNframes}, then for all propositions $A,B,C$,
\[(A\vee B)\cap\Box C\subseteq (A\cap C)\vee (B\cap C).\]
\end{proposition}
\begin{proof} Suppose $x\in (A\vee B)\cap \Box C$. We must show that for every $z\op x$, there is a $v\po z$ such that $v\in (A\cap C)\cup (B\cap C)$. Suppose $z\op x$. Then by indirect enrichment, there is a postrefinement $z^-$ of $z$ such that $x\po z^-$ and for any $y\po z^-$, there is a prerefinement $y^+$ of $y$ such that $z^-\op y^+\shortdashleftarrow x$. Since $x\in A\vee B$ and $x\po z^-$, there is a $y\po z^-$ such that $y\in A\cup B$. Let $y^+$ be a prerefinement of $y$ as above, so $y^+\in A\cup B$. Since $y^+\shortdashleftarrow x$ and $x\in \Box C$, we have $y^+\in C$. Hence $y^+\in (A\cup B)\cap C = (A\cap C)\cup (B\cap C)$. Then since $z^-\op y^+$ and $z^-$ is a postrefinement of $z$, we have $z\op y^+$, so $y^+$ is our desired $v$.\end{proof}

\begin{corollary}\label{MergCor} If a modal frame $(X,\op,\shortdashleftarrow)$ satisfies the indirect enrichment condition and $\shortdashleftarrow$ is transitive,  then for all propositions $A,B,C$,
\[(A\vee B)\cap\Box C\subseteq (A\cap \Box C)\vee (B\cap \Box C).\]
\end{corollary}
\begin{proof} Since $\shortdashleftarrow$ is transitive, $(A\vee B)\cap\Box C \subseteq (A\vee B)\cap\Box\Box C$, and by Proposition \ref{MergProp} with $\Box C$ in place of $C$, we have $(A\vee B)\cap\Box\Box C\subseteq (A\cap \Box C)\vee (B\cap \Box C)$.\end{proof}

Together Proposition~\ref{DirectEnrichProp} and Corollary~\ref{MergCor} yield the following. 

\begin{proposition} The dual algebra of an $\mathbf{FN4}$-frame is an $\mathbf{FN4}$-algebra.
\end{proposition}

\subsubsection{Completeness}

The completeness of $\mathsf{FN4}$ with respect to $\mathbf{FN4}$-frames follows from the completeness of $\mathsf{FN4}$ with respect to $\mathbf{FN4}$-algebras together with the fact that every $\mathbf{FN4}$-algebra can be embedded into the dual algebra of an $\mathbf{FN4}$-frame. Recall from Theorem~\ref{FunModalRep} that given an $\mathbf{FN4}$-algebra $(L,\neg,\Box)$, the map $a\mapsto \{(F,I)\in X\mid a\in F\}$ is a fundamental modal embedding of $(L,\neg,\Box)$ into the dual modal algebra of the canonical modal frame $(X,\op,\shortdashleftarrow)$ of $(L,\neg,\Box)$. We just need to prove that the canonical modal frame is an $\mathbf{FN4}$-frame.

The following is a standard exercise in modal logic.

\begin{lemma} The canonical modal frame $(X,\op,\shortdashleftarrow)$ of an $\mathbf{FN4}$-algebra  $(L,\neg,\Box)$ is such that $\shortdashleftarrow$ is reflexive and transitive.
\end{lemma}

It only remains to show that the canonical modal frame of an $\mathbf{FN4}$-algebra satisfies the enrichment conditions of Definition~\ref{FNframes}:
\begin{itemize}
\item direct enrichment: if $(F,I)\po(G,J)$, then there is a prerefinement $(G^+,J^+)$ of $(G,J)$ such that $(F,I)\po(G^+,J^+)\shortdashleftarrow(F,I)$;
\item indirect enrichment: if $(F,I)\po(H,K)$, then there is a postrefinement $(H^-,K^-)$ of $(H,K)$ such that $(F,I)\po(H^-,K^-)$ and for any $(G,J)\po(H^-,K^-)$, there is a prerefinement $(G^+,J^+)$ of $(G,J)$ such that $(H^-,K^-)\op(G^+,J^+)\shortdashleftarrow(F,I)$.
\end{itemize}

\begin{lemma}\label{CanDirect} The canonical modal frame $(X,\op,\shortdashleftarrow)$ of an $\mathbf{FN4}$-algebra  $(L,\neg,\Box)$ satisfies direct enrichment.
\end{lemma}

\begin{proof} Suppose $(F,I) \po (G,J)$, so $F\cap J=\varnothing$. Let $G^+$ be the filter generated by 
\[\{\Box a\mid \Box a\in F\}\cup G\]
and $J^+$ the ideal generated by $\{\neg c\mid c\in G^+\}$, so $(G^+,J^+)\in X$. Since $G^+\supseteq G$, $(G^+,J^+)$ prerefines $(G,J)$. Since $\Box a\in F$ implies $\Box a\in G^+$, which with $\Box a\leq a$ implies $a\in G^+$, we have $(G^+,J^+)\shortdashleftarrow (F,I)$.

It only remains to show that $(F,I) \po (G^+,J^+)$, i.e., $F\cap J^+=\varnothing$. Suppose, toward a contradiction, that $x\in F\cap J^+$. Since $x\in J^+$, there are some $c_1,\dots,c_n\in G^+$ such that  $x\leq \neg c_1\vee\dots\vee \neg c_n$. Let $c=c_1\wedge\dots\wedge c_n$, which implies $c\in G^+$ and $\neg c_1\vee\dots\vee \neg c_n\leq \neg c$, which in turn implies $x\leq\neg c$ and then $\neg c\in F$. Since $c\in G^+$, there are $\Box a\in F$ and $b\in G$ such that $\Box a\wedge b\leq c$. It follows that $\neg c\leq \neg (\Box a\wedge b)$ and hence $\neg(\Box a\wedge b)\in F$, which with $\Box a\in F$ yields $\neg(\Box a\wedge b)\wedge \Box a\in F$. Now since $b\wedge \Box a\leq \neg \neg (\Box a\wedge b)$, Definition~\ref{FN4AlgDef}.\ref{FN4AlgDef3} implies $\neg(\Box a\wedge b)\wedge \Box a\leq \neg b$, so $\neg b\in F$. But since $b\in G$, we have $\neg b\in J$, which contradicts the fact that $F\cap J=\varnothing$. Thus, we conclude that $F\cap J^+=\varnothing$, which completes the proof of direct enrichment.\end{proof}

\begin{lemma} The canonical modal frame $(X,\op,\shortdashleftarrow)$ of an $\mathbf{FN4}$-algebra  $(L,\neg,\Box)$ satisfies indirect enrichment.
\end{lemma}
\begin{proof} Suppose $(F,I)\po (H,K)$, so $F\cap K=\varnothing$. Let $H^-=H$ and \[K^-=\{b\in L\mid \mbox{for some }\Box a\in F: \Box a\wedge b\in K\}.\] We claim that  $K^-$ is an ideal. We have $0\in K^-$, since $\Box1=1\in F$ and $\Box1\wedge0=0\in K$. For downward closure, suppose $b'\leq b\in K^-$, so for some $\Box a\in F$, $\Box a\wedge b\in K$. Then $\Box a\wedge b'\leq \Box a\wedge b$, so $\Box a\wedge b'\in K$ and hence $b'\in K^-$. For closure under joins, suppose $b_1,b_2\in K^-$, so there are $\Box a_1,\Box a_2\in F$ such that $\Box a_1\wedge b_1,\Box a_2\wedge b_2\in K$ and hence $(\Box a_1\wedge b_1)\vee(\Box a_2\wedge b_2)\in K$. Let $a=a_1\wedge a_2$. Since $\Box a\leq\Box a_i$ for $i\in\{1,2\}$, we have $(\Box a\wedge b_1)\vee(\Box a\wedge b_2)\leq(\Box a_1\wedge b_1)\vee(\Box a_2\wedge b_2)$ and hence $(\Box a\wedge b_1)\vee(\Box a\wedge b_2)\in K$. By Definition~\ref{FN4AlgDef}.\ref{FN4AlgDef2}, $\Box a\wedge(b_1\vee b_2)\leq(\Box a\wedge b_1)\vee(\Box a\wedge b_2)$, so we have $\Box a\wedge(b_1\vee b_2)\in K$. Then since $\Box a_1,\Box a_2\in F$ implies $\Box a\in F$, we conclude that $b_1\vee b_2\in K^-$. Thus, $K^-$ is indeed an ideal. Note that $\Box 1=1\in F$ implies $K\subseteq K^-$. Then since $H^-=H$, and $h\in H$ implies $\neg h\in K\subseteq K^-$, it follows that $(H^-,K^-)\in X$. We claim that $(F,I)\po(H^-,K^-)$, i.e., $F\cap K^-=\varnothing$. Indeed, if $x\in F\cap K^-$, so there is some $\Box a\in F$ such that $\Box a\wedge x\in K$, then $\Box a\wedge x\in F$ as well, contradicting $F\cap K=\varnothing$. In addition, $K\subseteq K^-$ implies that $(H^-,K^-)$ postrefines~$(H,K)$.

Now suppose $(H^-,K^-)\op (G,J)$, so $G\cap K^-=\varnothing$. Let $G^+$ be the filter generated by 
\[\{\Box a\mid \Box a\in F\}\cup G\]
and $J^+$ the improper ideal. Then as in the proof of Lemma~\ref{CanDirect}, we have $(G^+,J^+)\in X$, $(G^+,J^+)$ prerefines $(G,J)$, and $(G^+,J^+)\shortdashleftarrow (F,I)$. It remains to show that $(H^-,K^-)\op (G^+,J^+)$, i.e., $G^+\cap K^-=\varnothing$. Suppose for contradiction that $x\in G^+\cap K^-$. Then since $x\in G^+$, there are $\Box a\in F$ and $ b\in G$ such that $\Box a\wedge  b\leq x$. Since $x\in K^-$, there is some $\Box c\in F$ such that $\Box c\wedge x\in K$. Since $\Box a\wedge b\leq x$, we have $\Box c\wedge \Box a\wedge b\leq \Box c\wedge x$, so $\Box c\wedge \Box a\wedge b\in K$. Then since $\Box (c\wedge a)=\Box c\wedge\Box a\in F$, it follows that $b\in K^-$, contradicting $G\cap K^-=\varnothing$. Thus, we conclude that $G^+\cap K^-=\varnothing$, which completes the proof of indirect enrichment.\end{proof}

\bibliographystyle{plainnat}
\bibliography{orthomodal}

\end{document}